\documentclass[final]{siamart}

\usepackage{amsfonts}
\usepackage{graphicx}
\usepackage{epstopdf}
\usepackage{amsmath}
\usepackage{amssymb}
\usepackage{enumitem}
\usepackage{empheq}
\usepackage{stmaryrd}        
\usepackage{tikz}
\usepackage{multicol}
\usepackage{ulem}         
\usetikzlibrary{decorations.pathmorphing,matrix,fit,backgrounds}
\usepackage{todonotes}
\usepackage{hyperref}
\usepackage{dsfont}
\ifpdf
  \DeclareGraphicsExtensions{.eps,.pdf,.png,.jpg}
\else
  \DeclareGraphicsExtensions{.eps}
\fi

\usepackage{xstring}
\usepackage{catchfile}
\usepackage{shellesc}
\usepackage{dutchcal}
\usepackage{pgfplots}
\usepackage{cite}
\usepackage[left = 4cm, right = 4cm]{geometry}

\immediate\write18{git describe --always --dirty=-modifie > revision}

\theoremstyle{plain}
\newtheorem{thm}{theorem}[section]
\newtheorem{prop}[thm]{proposition}
\newtheorem{cor}[thm]{Corollary}
\newtheorem{lem}[thm]{Lemma}

\newtheorem{rem}{Remark}[section]

\newtheorem{ddef}{Definition}[section]

\newcommand{\R}{\mathbb{R}}

\newcommand{\N}{\mathbb{N}}
\newcommand{\indicatrice}[1]{\mathds{1}_{#1}}
\newcommand{\mc}[1]{\mathcal{#1}}
\newcommand{\scalar}[3]{\left\langle #1,#2 \right\rangle_{#3}}

\newcommand{\norm}[2]{\left\Vert#1\right\Vert_{#2}}

\newcommand{\cost}{K}
\newcommand{\grush}{G_\mu}

\newcommand{\sphere}{\mathbb{S}^{d\scalebox{0.8}{-}1}}
\renewcommand{\epsilon}{\varepsilon}
\newcommand{\Hess}{\mathsf{Hess}}
\newcommand{\s}{\mathcal{s}}
\newcommand{\cc}{\mathcal{c}}

\newlength\marklength
\numberwithin{equation}{section}

\newcommand{\TheTitle}{On the critical time of observability of the multi-dimensional Baouendi-Grushin equation}
\newcommand{\TheAuthors}{M.T}

\title{{\TheTitle}}

\author{\vspace{0.5cm} Jérémi DARD\'E\footnote{\href{mailto:jeremi.darde@math.univ-toulouse.fr}{jeremi.darde@math.univ-toulouse.fr}} \and Mathilda TRABUT\footnote{\href{mailto:mathilda.trabut@math.univ-toulouse.fr}{mathilda.trabut@math.univ-toulouse.fr}} \\ \vspace{0.2cm}
{\scriptsize Univ Toulouse, INUC, UT2J, INSA Toulouse, TSE, CNRS, IMT, Toulouse, France.}}

\ifpdf
  \hypersetup{
    pdftitle={\TheTitle},
    pdfauthor={\TheAuthors}
  }
\fi

\begin{document}

\maketitle
\begin{abstract}
We investigate the observability properties of the Baouendi–Grushin equation on a tensorized domain $\Omega := \mc{B}_R \times \tilde \Omega$, where $\mc{B}_R$ is the open ball of radius $R$ in dimension $d \ge 2$, and $\tilde \Omega$ is a smooth, bounded, open set of arbitrary dimension. 
Our main result is a precise calculation of the minimal observability time $T^*$, for tensorized observation sets of the form $\omega \times \tilde \Omega$, with $\omega \subset \mc{B}_R$ (internal observation), and $\Gamma \times \tilde \Omega$, with $\Gamma \subset \partial \mc{B}_R$ (boundary observation). \newline
The main novelty regards the sufficient condition, that is observability of the system when $T>T^*$. This is established by combining refined observability inequalities on the annulus—or the entire boundary—using Carleman estimates, together with a Lebeau–Robbiano strategy to localize the observation sets.
\end{abstract}

\begin{keywords}
Observability, controllability, degenerate parabolic equations, Carleman estimates, Lebeau-Robbiano strategy.
\end{keywords}

\begin{AMS}
 35K65, 93B05, 93B07.  
\end{AMS}

\paragraph{Ackowledgements} Our heartfelt thanks go to Paul Alphonse, Franck Boyer, Sylvain Ervedoza, Morgan Morancey, and Roman Vanlaere. 
{Their enriching suggestions have been essential to the development of this work.}

\medskip

This work was supported by the  ANR  LabEx  CIMI  (under  grant  ANR-11-LABX-0040) within the French State Programme “Investissements d’Avenir,
by the University Research School EUR-MINT (ANR-18-EURE-0023) within the French State Programme “Investissements d’Avenir,
and by the ANR-Project TRECOS, ANR-20-CE40-0009.

\section{Introduction}
The goal of this article is to prove observability inequalities for the Baouendi-Grushin equation in a multi-dimensional setting, through tensorized observation sets.
\subsection{Baouendi-Grushin equation: statement, well-posedness and observability results}
Let $d \ge 2$ and $\tilde d \ge 1$, $\mc{B}_R \subset \R^d$ be the open ball centered at zero and of radius $R>0$, and $\tilde \Omega \subset \R^{\tilde d}$ be a smooth open bounded subset. In this study, we consider the Baouendi-Grushin equation on the domain $\Omega := \mc{B}_R \times \tilde \Omega$
  \begin{equation} \label{grushin}
   \left\{
   \begin{aligned}
 \partial_t y -\Delta_x y - \|x\|^2\Delta_{\tilde x} y &= 0 &&\text{in }  (0,T)\times \mathcal{B}_R \times \tilde \Omega, \\
 y &= 0 &&\text{in } (0,T) \times \partial \mathcal{B}_R \times \tilde \Omega, \\
 y(0,\cdot) &= y_0 &&\text{in } \mathcal{B}_R \times \tilde \Omega .
\end{aligned}\right.
\end{equation} 
where $(t,x,\tilde x) \in (0,T)\times \mathcal{B}_R \times \tilde \Omega$. \newline
{The Baouendi-Grushin system} is well-posed: for any $T>0$ and any initial data $y_0 \in L^2(\Omega)$, {equation \eqref{grushin} admits} a unique solution $y \in C^0([0,T],L^2(\Omega)) $. Furthermore, if the initial data $y_0$ belongs to $H^1_0(\Omega)$, the solution $y$ {gains smoothness as it lies in} $C^0([0,T],H^1_0(\Omega))$, with a normal derivative that verifies
$\frac{\partial y}{\partial n} \in L^2((0,T),L^2(\partial \mathcal{B}_R \times \tilde \Omega))$ -- see \cref{solutiongrushin} for further details. \newline
{In this study, we are interested in the observability properties of the Baouendi-Grushin system from tensorized observation set:}
\begin{ddef}[Internal observability]
\label{defobsinterne}
Let $T>0$, $\omega \subset \mc B_R$. We say that the Baouendi-Grushin equation is observable at time $T$, from $\omega \times \tilde \Omega$, with cost $ K_T > 0$, if for any $y_0 \in L^2(\Omega)$, the solution of \eqref{grushin} satisfies
\begin{equation} \label{obsinternegrushin}
\int_{\mc{B}_R \times \tilde \Omega} \vert y(T)\vert^2 \le K_T^2 \int_0^T \int_{\omega \times \tilde \Omega} \vert y \vert^2.
\end{equation}
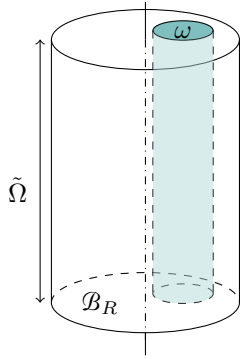
\begin{figure}[ht]
\centering
\begin{tikzpicture}
\draw (0,0) ellipse (1.25 and 0.4);
\draw (-1.25,0) -- (-1.25,-3.5);
\draw (-1.25,-3.5) arc (180:360:1.25 and 0.4);
\draw [dash dot] (0,0) -- (0,-4.);
\draw (0,0) -- (0,0.5);
\draw (0,-3.9) -- (0,-4.2);
\draw [dashed] (-1.25,-3.5) arc (180:360:1.25 and -0.4);
\draw (1.25,-3.5) -- (1.25,0);  
\draw [fill=teal!50] (0.5,0.12) ellipse (0.4 and 0.13);
\draw [dashed] (0.1,0.12) -- (0.1,-3.38);
\draw [dashed] (0.9,0.12) -- (0.9,-3.38);
\draw [dashed] (0.5,-3.38) ellipse (0.4 and 0.13);
\fill [teal!30,opacity=0.5]  (0.1,0.12) -- (0.1,-3.38) arc (180:360:0.4 and 0.13) -- (0.9,0.12) arc(0:180:0.4 and -0.13);
\node at (-0.6,-3.5) {$\mathcal{B}_R$};
\node at (0.5,0.1) {$\omega$};
\draw [<->] (-1.4,0) -- (-1.4,-3.5);
\node at (-1.7,-2) {$\tilde \Omega$};
\end{tikzpicture}
 \caption{Domain $\Omega$, observation set $\omega \times \tilde\Omega$}
  \label{cylobsinterne}
\end{figure}
\end{ddef}
\begin{ddef}[Boundary observability]
\label{defobsbord}
Let $T>0$, $\Gamma_{\!R} \subset \partial \mc{B}_R$. We say that the Baouendi-Grushin equation is observable at time $T$, from $\Gamma_{\!R} \times \tilde \Omega$, with cost $ K_T > 0$, if for any $y_0 \in H^1_0(\Omega)$, the solution of \eqref{grushin} satisfies
\begin{equation} \label{obsbordgrushin}
\int_{\mc{B}_R \times \tilde \Omega} \vert \nabla y(T) \vert^2 \le K_T^2 \int_0^T \int_{\Gamma_{\!R} \times \tilde \Omega} \left\vert \frac{\partial y}{\partial n} \right\vert^2.
\end{equation}
\end{ddef}
\begin{figure}[ht]
\centering
\scalebox{0.7}{
\begin{tikzpicture}
   \draw [dashed] (0,0) ellipse (1.5cm and 0.75cm);

   \draw (-1.5,0) -- (-1.5,4);
   \draw (1.5,0) -- (1.5,4);

   \draw [fill=teal!50] (0.5,0) rectangle (1,4);
   \draw [fill=white] (0,4) ellipse (1.5cm and 0.75cm);

   \begin{scope}
     \clip (-2,0) rectangle (2,-1);
     \draw (0,0) ellipse (1.5cm and 0.75cm);
   \end{scope}

   \begin{scope}
     \clip (0,0) ellipse (1.5cm and 0.75cm);
    \draw [fill=teal!50] (0.5,-2) rectangle (1,4);
  \end{scope}

  \begin{scope}
    \clip (0.5,-1) rectangle (1,4);
    \draw [very thick] (0,4) ellipse (1.5cm and 0.75cm);

  \end{scope}

  \node at (-0.75,4) {$\mc{B}_R$};
  \node at (0.75,3.6) {$\Gamma_{\!R}$};
  \node at (-1.75,2) {$\tilde \Omega$};

  \draw [dash dot] (0,0) -- (0,4);
\end{tikzpicture}}
 \caption{Domain $\Omega$, observation set $\Gamma_{\!R} \times \tilde\Omega$}
  \label{cyltoutomegay}
  \end{figure}
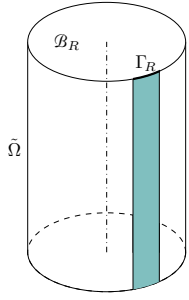
  
{With these definitions,} we restrict ourselves to tensorized observation sets where, in the $\tilde x$ variable, we observe on the whole domain. {This restriction is primarily for technical reasons, as our approach relies on a Fourier decomposition in the $\tilde{x}$ variable.}
However, this choice {aligns  with established results} in the one-dimensional setting $d= \tilde d = 1$. 
In this setting, {there exist known configurations in which the observation set does not cover the entire $\tilde{x}$ domain, and the corresponding system fails to be observable for any positive time $T$.
On the other hand, when $d = \tilde d = 1$ and the observation is tensorized as defined above, the system is known to be observable only for sufficiently large $T$, and moreover, the minimal observability time is now fully characterized.}
{For an overview of these one-dimensional results, we refer the reader to \cref{section_scientific_context}.}
 \newline
Unlike the case $d=1$, the case $d \ge 2$ is still largely unexplored. The first result is given in \cite[Theorem 4]{BCY} for the internal observation, {where it} is proved that if $\omega$ does not contain the origin, there exists a positive minimal time for observability -- {though no estimate of this time is provided}. In \cite{Beauchard2020}, the authors prove the boundary observability inequality \eqref{obsbordgrushin} in the case $\Gamma_{\!R} = \partial \mc B_R$, {and explicitly determine the minimal time. Specifically, their theorem—reformulated in our context—reads as follows:}
\begin{thm}
 \label{BDE}
Let $R>0$, and define $T^* = \frac{R^2}{2d}.$
If $T>T^*$, the Baouendi-Grushin equation \eqref{grushin} is observable from the boundary $\partial \mc B_R \times \tilde \Omega$ and if $T<T^*$ the Baouendi-Grushin equation \eqref{grushin} is not observable from the boundary $\partial \mc B_R \times \tilde \Omega$.
\end{thm}
Finally, the internal observation when $\omega$ is an annulus is obtained from the observation on the boundary, using cut-off arguments. It is done in \cite[Theorem 3.1.2.]{vanla2025}, which deals with a larger class of equations, including the Baouendi-Grushin equation \eqref{grushin}:
\begin{thm}
 \label{Roman}
Let $0<R_1<R_2$, and define  $\omega = \{x \in \mc B_R ~;~ R_1<\norm{x}{}<R_2\}$. {Set $T^* = \frac{R_1^2}{2d}$.}
If $T>T^*$, the Baouendi-Grushin equation \eqref{grushin} is observable from $\omega \times \tilde \Omega$ and if $T<T^*$ the Baouendi-Grushin equation \eqref{grushin} is not observable from $\omega \times \tilde \Omega$.
\end{thm}
Our paper improves those results. For the internal observability, we have the optimal minimal time for any configuration corresponding to \cref{defobsinterne}. {Additionally, we provide estimates of the observability cost for specific observation sets.} 
{Our results are sum up in the following Theorem:}
\begin{thm} \label{maininterne}
Let $\omega$ be a nonempty open subset of $\mathcal{B}_R$, and define $$T^* = \frac{\inf_{x\in \omega} \Vert x \Vert^2}{2d}.$$
If $T>T^*$, the Bouendi-Grushin equation \eqref{grushin} is observable from $\omega \times \tilde \Omega$.
Moreover if $\omega$ is a section of annulus, i.e.,
 \begin{equation} \label{sectionanneau} \omega = \left\{ x \in \R^d;~x= r\theta, ~R_1 <r< R_2,  ~\theta \in \Gamma \right\},\end{equation} with $0<R_1<R_2<R$ and $\Gamma \subset \sphere$ an open nonempty subset, then for any $\beta >1$ there exists $c_\beta > 0$ such that $$K_T \le \exp\left( \frac{c_\beta}{(T-T^*)^{2\beta}}\right).$$
If $T<T^*$ the Baouendi-Grushin equation \eqref{grushin} is not observable from $\omega \times \tilde \Omega$.
\end{thm}
For the boundary observability, we have the optimal minimal time for any subset of the boundary, with an estimate of the cost.
 \begin{thm}
 \label{mainbord}
Let $R>0$, $\Gamma_{\!R}$ be a nonempty open subset of $\partial \mc B_R$, and define $$T^* = \frac{R^2}{2d}.$$
If $T>T^*$, the Baouendi-Grushin equation \eqref{grushin} is observable from $\Gamma_{\!R} \times \tilde \Omega$. Moreover for any $\beta >1$ there exists $c_\beta > 0$ such that $$K_T \le \exp\left( \frac{c_\beta}{(T-T^*)^{2\beta}}\right).$$
If $T<T^*$ the Baouendi-Grushin equation \eqref{grushin} is not observable from $\Gamma_{\!R} \times \tilde \Omega$. 
\end{thm}
Let us mention that the necessary conditions (which prove the optimality of the minimal time) are direct consequences of \cref{BDE} and \cref{Roman}. 
{Nevertheless, we provide in \cref{optimality} a proof for the sake of completeness, and to deal with the case $T = T^*$ when $d = 2$ or 3.}
Indeed we have 
\begin{prop} \label{prop_necessary_condition_dim2and3}
Under the same assumptions as in \cref{maininterne}, if $d=2$ or $3$, the Baouendi-Grushin equation \eqref{grushin} is not observable from $\omega \times \tilde \Omega$ at $T=T^*$.
\end{prop}
The rest of the paper focuses only on the sufficient condition, i.e. proving the observability when $T > T^*$.

\begin{rem}
We provide an estimate for the blow-up in observability cost as $T$ approaches $T^*$, for $\omega$ defined by \eqref{sectionanneau} for the internal observability and any $\Gamma_{\!R} \subset \partial \mathcal{B}_R$ for the boundary observability. Indeed, we prove that in both cases,
$$
K_T \leq e^{\frac{c_\beta}{(T-T^*)^{2\beta}}}, \quad \beta > 1.
$$
This estimate is likely suboptimal. For instance, if $\omega$ is a complete annulus -- \textit{i.e.} $\omega = \left\lbrace x \in \R^d, \ R_1 < \Vert x \Vert < R_2\right\rbrace$ -- we obtain a sharper estimate:
$$
K_T \leq e^{\frac{c_\beta}{(T-T^*)^{\beta}}}, \quad \beta > 1.
$$
The optimality of these estimates is further discussed in \cref{subsection_blowup}
\end{rem}

\subsection{Related null-controllability results}
{Classicaly,} our observability results translate into null-controllability results, and the estimate of the observability cost gives an estimate on the control cost (see, among others, \cite{zabczyk20,Dolecki,Coron2007,Tucsnak2009}). 
\newline
{More precisely}, for the internal case, the controlled system is
  \begin{equation} \label{grushincontrolinterne}
   \left\{
   \begin{aligned}
 \partial_t y -\Delta_x y - \|x\|^2\Delta_{\tilde x} y &= \indicatrice{\omega \times \tilde \Omega}v &&\text{in }  (0,T)\times \mathcal{B}_R \times \tilde \Omega, \\
 y &= 0 &&\text{in } (0,T) \times \partial \mathcal{B}_R \times \tilde \Omega, \\
 y(0,\cdot) &= y_0 &&\text{in } \mathcal{B}_R \times \tilde \Omega .
\end{aligned}\right.
\end{equation} 
This system is well-posed: for any initial data in $L^2(\Omega)$ {and any $v$ in $L^2((0,T)\times \omega \times \tilde \Omega)$}, there exists a unique solution continuous in time with values in $L^2(\Omega)$ -- see \cref{solutiongrushin} for additional details. {Furthermore, it is} null-controllable if $T$ is larger than the critical time, in the following sense:
\begin{cor}
Let $\omega$ be a nonempty open subset of $\mathcal{B}_R$, and define $$T^* = \frac{\inf_{x\in \omega} \Vert x \Vert^2}{2d}.$$
If $T>T^*$, for any $y_0 \in L^2(\Omega)$, there exists $v \in L^2((0,T)\times \omega \times \tilde \Omega)$ such that the solution of \eqref{grushincontrolinterne} satisfies $y(T)=0$. 
Moreover if $\omega$ is a section of annulus, and $v$ is the null-control of minimal $L^2$-norm, we have that for any $\beta >1$, there exists $c_\beta >0$ such that
$$\norm{v}{L^2((0,T)\times \omega \times \tilde \Omega)} \le \exp\left( \frac{c_\beta}{(T-T^*)^{2\beta}}\right) \norm{y_0}{L^2(\Omega)}.$$
If $T<T^*$ there exists $y_0 \in L^2(\Omega)$, such that for any $v \in L^2((0,T)\times \omega \times \tilde \Omega)$ the solution of \eqref{grushincontrolinterne} does not identically vanish at time $T$. 
\end{cor}
For the boundary case the controlled system is 
  \begin{equation} \label{grushincontrolbord}
   \left\{
   \begin{aligned}
 \partial_t y -\Delta_x y - \|x\|^2\Delta_{\tilde x} y &= 0 &&\text{in }  (0,T)\times \mathcal{B}_R \times \tilde \Omega, \\
 y &= \indicatrice{\Gamma_{\!R} \times \tilde \Omega}v &&\text{in } (0,T) \times \partial \mathcal{B}_R \times \tilde \Omega, \\
 y(0,\cdot) &= y_0 &&\text{in } \mathcal{B}_R \times \tilde \Omega .
\end{aligned}\right.
\end{equation} 
To consider boundary control, the natural duality is between $H^1_0(\Omega)$ and $H^{-1}(\Omega)$, therefore we need this system to be well-posed for any initial data in $H^{-1}(\Omega)$. This is true when the solutions are defined in the weak sense (see \cref{solfaible}). Then, similarly to the internal controllability result, we have that \eqref{grushincontrolbord} is null-controllable if $T$ is larger than the critical time.
\begin{cor}
Let $R>0$, $\Gamma_{\!R}$ be a nonempty open subset of $\partial \mc{B}_R$, and define $$T^* = \frac{R^2}{2d}.$$
If $T>T^*$, for any $y_0 \in H^{-1}(\Omega)$, there exists $v \in L^2((0,T)\times \Gamma_{\!R} \times \tilde \Omega)$ such that the solution of \eqref{grushincontrolinterne} satisfies $y(T)=0$.
Moreover, if $v$ is the null-control of minimal $L^2$-norm, we have that for any $\beta >1$ there exists $c_\beta >0$ such that
$$\norm{v}{L^2((0,T)\times \Gamma_{\!R} \times \tilde \Omega)} \le \exp\left( \frac{c_\beta}{(T-T^*)^{2\beta}}\right) \norm{y_0}{H^{-1}(\Omega)}.$$
If $T<T^*$ there exists $y_0 \in H^{-1}(\Omega)$, such that for any $v \in L^2((0,T)\times \Gamma_{\!R} \times \tilde \Omega)$ the solution of \eqref{grushincontrolinterne} does not identically vanish at time $T$. 
\end{cor}
 
\subsection{Strategy of proof}
{We now} explain the main steps leading to \cref{maininterne} and \cref{mainbord}. Since {the strategy is} similar for the internal and the boundary case, we focus solely on the internal case, that is \cref{maininterne}.
\subsubsection{From harmonic heat to Baouendi-Grushin} \label{section_strategy_HH_Grushin}
Let us consider the harmonic heat equation with parameter $\mu >0$,
\begin{equation}
\label{harmonicheat}
\left\{
   \begin{aligned}
 \partial_t y_\mu -\Delta_x y_\mu + \mu^2\|x\|^2 y_\mu &= 0 &&\text{in }  (0,T)\times \mathcal{B}_R\\
 y_\mu &= 0 &&\text{in } (0,T) \times \partial \mathcal{B}_R, \\
 y_\mu(0,\cdot) &= y_{0,\mu} &&\text{in } \mathcal{B}_R.
\end{aligned}\right.
\end{equation}
For any initial data $y_{0,\mu} \in L^2(\mc{B}_R)$, this equation has a unique solution \linebreak $y_\mu \in C^0([0,T];L^2(\mc{B}_R))$ (see \cref{prop_wellposedness_harmonicheatequation}).
We say that the harmonic heat equation \eqref{harmonicheat} is observable from $\omega$, uniformly in $\mu>0$, if there exists a constant $K_T>0$ such that for any $\mu>0$, for any solution $y_\mu$ of \eqref{harmonicheat}, we have,
\begin{equation} \label{obsunifhh}
\int_{\mc{B}_R}|y_\mu(T,x)|^2   \le K_T^2 \int_0^T\int_{\omega}|y_\mu(t,x)|^2.
\end{equation} 
{We claim that if \eqref{obsunifhh} holds true, then}
we can deduce the observability of the Baouendi-Grushin equation \eqref{grushin} from $\omega \times \tilde \Omega$ with cost $K_T$, using the same strategy as \cite[Sections 2.2 and 2.3]{Beauchard2013}.
Indeed if we denote $(\mu_p^2)_{p\in\N} $ the eigenvalues of the Dirichlet laplacian on $\tilde \Omega$, and decompose the solution $y$ of \eqref{grushin} in an associated orthonormal basis of eigenfunctions $(\phi_p)_{p\in\N}$,
$$y(t,x,\tilde x) = \sum_{p = 0 }^\infty y_{\mu_p}(t,x)\phi_p(\tilde x),$$ one can see that $ y_{\mu_p}$ satisfies \eqref{harmonicheat} with $\mu = \mu_p$.
Then we can sum the observability inequalities \eqref{obsunifhh} using Parseval's identity and obtain the observability inequality \eqref{obsinternegrushin} for $y$.
\subsubsection{Uniform observability of the harmonic heat equation \eqref{harmonicheat}}
\label{ideeuniforme}
{
We now aim to establish the uniform observability inequality \eqref{obsunifhh}.
This uniform estimate, achievable for large enough $T$, results from the interplay between the system's natural dissipation and an observability estimate.
More precisely, we first note that, since we are dealing with a heat equation with potential, it is already known that \eqref{harmonicheat} is observable, although the cost $K_T$ depends on $\mu$.
Let us assume that this cost is bounded by $C_T^2 \exp (\mc{c} \mu)$, where the constant $\mc{c}$ does not depend on $T$ -- in other words, that for any $\mu>0$, any $y_\mu$ solution of \eqref{harmonicheat} satisfies.
}
$$\int_{\mc{B}_R}|y_\mu(T,x)|^2   \le C_T^2\exp(\mc{c}\mu) \int_0^T\int_{\omega}|y_\mu(t,x)|^2  .$$
{Using} the known dissipation estimate (see \cite[section 4.1]{Beauchard2020}) 
$$\int_{\mc{B}_R}|y_\mu(t,x)|^2   \le \exp(-2d\mu(t-s)) \int_{\mc{B}_R}|y_\mu(s,x)|^2  ,$$
{we claim that} the uniform estimate holds true for $T > \mc{c}/2d$. Indeed, {choosing} $\delta\in (0,T)$ such that $T-\delta> \mc{c}/2d$, {and combining  the dissipation estimate on  $[\delta,T]$ and the observability inequality on  $[0,\delta]$, easily leads to}
$$\int_{\mc{B}_R}|y_\mu(T,x)|^2   \le C_\delta^2 \exp\left(-\mu(2d(T-\delta)-\mc{c})\right) \int_0^\delta\int_{\omega}|y_\mu(t,x)|^2 .$$
Since $T-\delta> \mc{c}/2d$, we obtain the desired uniform estimate. \newline
As one can see, the crucial work to get the sharp minimal time, is to determine precisely the dependance in $\mu$ in the observability estimate (i.e. of the constant $\mc{c}$).\newline
This precise estimate is obtained in two steps: we start by proving an observability inequality with a precise dependance in $T$ and $\mu$ when the observation set is an annulus. The second step is to restrict the observation set to a section of annulus using a Lebeau-Robbiano strategy \cite{Lebeau1995}. The case of a generic $\omega$ is then easily deduced (see \cref{generalopensubset}).

\subsubsection{Observability from the whole annulus}
Let us give more details about the observability estimate on the annulus
\begin{equation} \label{annulus}
\mathcal A = \left\lbrace x \in \R^d, \ R_1 < \Vert x \Vert < R_2 \right\rbrace \subset \mathcal B_R.
\end{equation}
The observability result reads:

\begin{prop} \label{prop_observability_from_annulus_intro}
There exists a positive constant $\cc$, and for all $\beta > 1$, 
a positive constant $c_\beta$ verifying $c_\beta \xrightarrow{\beta\rightarrow 1} \infty$,  such that for all $T>0$,  $\mu>0$,  $\varepsilon\in (0,1]$, and all solution $y_\mu$ of \eqref{harmonicheat}, the following estimate holds:
$$
\int_{\mathcal{B}_R} \vert y_\mu(T) \vert^2 \leq \frac{\cc}{\varepsilon^6} (1+\mu) e^{\mu (R_1+\varepsilon)^2} e^{\frac{2 c_\beta}{T^\beta}}\int_0^T \int_{\mathcal A } \vert y_\mu \vert^2.
$$
\end{prop}
This is obtained by a combination of two Carleman-type estimate. The first one follows the standard strategy for the heat equation with potential, like the one in \cite{zhu2024}, whereas the second one is very specific to the problem of interest, and is a refinement of a strategy found in \cite{Beauchard2020}. These are stated and proven in \cref{carlemaneries}. Then we combine them, using {the first one} when $T\mu$ is small, and the other when $T\mu$ is large -- this is done in \cref{obsfromannulus}. This strategy allows us to be sharp in the dependance in $\mu$ with a cost in $ \exp(\mu (R_1+\varepsilon)^2)$ and to have a dependance in time of the form $\exp(c_\beta/T^\beta)$ which is needed for the restriction result.

\subsubsection{From observability through the annulus to a section of annulus} 
{The final step of the proof consists in deducing}
observability through a section of annulus $\omega$ as defined in \eqref{sectionanneau}, from observability through the whole annulus. This is done using a Lebeau-Robbiano type strategy on observability inequalities. For this, we need three assumptions: \medskip \newline
\noindent $\bullet$ Observation on the whole annulus, with a cost in $\exp(b/T^\beta)$. This is given by the previous step, which gives, for any $y_\mu$ solution of \eqref{harmonicheat},
\begin{equation}
\label{obsintro}
\int_{\mathcal{B}_R} \vert y_\mu(T) \vert^2 \leq C_{obs}(\mu) e^{\frac{2 c_\beta}{T^\beta}}\int_0^T \int_{\mathcal A } \vert y_\mu \vert^2, \quad  T \in (0,1].
\end{equation}
\noindent $\bullet$ A family of nondecreasing subspaces of $L^2(\mc{B}_R)$, $(\mc{F}_\sigma)_{\sigma>0}$, such that for any $z \in \mc{F}_\sigma$ we have an inequality relating the energy on the annulus and on $\omega$,
	\begin{equation}
	\label{specintro}
	\int_{\mc A}|z|^2 \le C_{rel} e^{2a\sqrt\sigma} \int_{\omega}|z|^2\quad \sigma >0, 	\end{equation} 
and for any initial data $y_{0,\mu} \in \mc{F}_\sigma^\perp$, the solution $y_\mu$ of \eqref{harmonicheat} satisfies the dissipation estimate
\begin{equation}
	\label{dissipintro}
	\int_{\mc B_R}|y_\mu(t)|^2 \le C_{dissip} e^{-2c\sigma t}\int_{\mc B_R}|y_{0,\mu}|^2 , \quad \sigma >0  \quad  t \in (0,1].
	\end{equation} 
{These properties are proven in \cref{applicationrestriction}, with the subspaces $\mc{F}_\sigma$ constructed via spherical harmonics. Consequently, the restriction argument requires the domain to be a ball.} \newline
One can see that combining \eqref{obsintro} and \eqref{specintro}, we obtain observability through $\omega$ for initial data in $\mc{F}_\sigma$, and thanks to the dissipation estimate, the part on $\mc{F}_\sigma^\perp$ is small.
{These form the core ingredients of} the so-called Lebeau-Robbiano strategy.
This strategy has been used in the controllability framework \cite{Benabdallah2007,Lebeau1995,Benabdallah2014,Allonsius2020}, and it is given in the observability framework in \cite{Seidman2008} and \cite[Theorem 2.2]{Miller2010}. {Actually}, applying \cite[Theorem 2.2]{Miller2010} {gives} that the harmonic heat equation \eqref{harmonicheat} is observable through $\omega$ with cost $K_T = \exp\left(\frac{C}{T^\beta}\right)$, but {the dependancy of $C$ on $\mu$ is not explicit}.
\newline
{We therefore adapt the stategy developed in \cite{Miller2010}, leading to \cref{LR}.}
{In the end}, we obtain a cost  
of the form $\exp(c_\beta/T^\beta) \times \exp(\mu (1+\gamma )(R_1+\epsilon)^2) $ for any $\epsilon >0$ and $\gamma >0$, with $c_\beta>0$ independant of $\mu$. Note that $\epsilon$ comes from the observation on the whole annulus, and $\gamma$ appears because of the restriction argument. 
{ With this precise estimate, we conclude following the strategy briefly explained in \cref{ideeuniforme}}.

\subsection{Scientific context} \label{section_scientific_context}

To appreciate how unusually the Baouendi-Grushin equation behaves, from an observability perspective,  as a parabolic equation, it is helpful to first recall what is known about the most typical parabolic equation: the heat equation.

Let $\Omega$ be a smooth bounded domain of $\R^d$, $\omega \subset \Omega$ a non-empty open subset, and $T>0$. We say that the heat equation is observable 
at time $T>0$ through $\omega$ if there exists a constant $\cc$ such that, for any function $y$ in $C^0([0,T];L^2(\Omega))\cap L^2((0,T);H^1_0(\Omega))$
satisfying $\partial_t y - \Delta y = 0$ in $(0,T) \times \Omega$, the following inequality holds:
$$
\int_\Omega \vert y(T) \vert^2 \leq \cc \int_0^T \int_\omega \vert y \vert^2.
$$
It is now well-established  that the heat equation is observable trough any non-empty open set and for any $T>0$ \cite{Fattorini1971,Lebeau1995,fursikovimanuvilov96}. In particular, no geometric or time-horizon conditions are required for the observability of the heat equation to hold. 

The parabolic Baouendi-Grushin equation with quadratic degeneracy, named after Mohamed Salah Baouendi \cite{baouendi67} and 
Victor Vasilievich Grushin \cite{grushin71}, behaves quite differently. This was first observed in a two-dimensional setting in the pionnering work \cite{Beauchard2013}: let $\Omega$ be the rectangle $(-1,1)\times(0,1)$, and $\omega$ the vertical strip $(a,b) \times (0,1)$, with $0<a<b<1$, see \cref{config1}.  

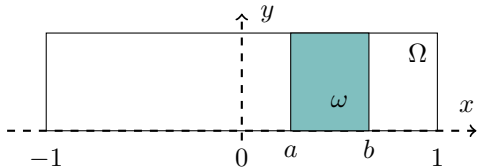
\begin{figure}[h!]
\label{config1}
\begin{center}
	\begin{tikzpicture}[scale=1.3]
	
	\node[] (Omega) at (1.8,0.8) {$\Omega$};
	\node[label={below:$-1$}] (lmoins) at (-2,0) {};
	\node[label={below:$1$}] (lplus) at (2,0) {};
	\draw[->,dashed,thick](-2.4,0) -- (2.4,0);
	\draw[](-2,0)rectangle(2,1);
	\node[label={above:$x$}] (x) at (2.3,0) {};
	\node[label={right:$y$}] (y) at (0,1.2) {};
	\node[label={below:$0$}] (origin) at (0,0) {};
	\draw[fill=teal!50] (0.5,0) rectangle (1.3,1);
	\node[] (omega) at (1,0.3) {$\omega$};
	\node[label={below:$a$}] (a) at (0.5,0.04) {};
	\node[label={below:$b$}] (b) at (1.3,0.1) {};
	\draw[->,dashed,thick](0,-0.1)  -- (0,1.2);
	\end{tikzpicture}
	\caption{Geometrical configuration studied in \cite{Beauchard2013}.}
\end{center}
\end{figure}

In this situation, the authors prove that there exists a positive critical time $T^*\geq \frac{a^2}{2}$, such that for any $T>T^*$,
there exists a constant $\cc$ such that for any function $y \in C^0([0,T]; L^2(\Omega))\cap L^2((0,T);H^1_0(\Omega))$ satisfying
\begin{equation} \label{grush1D} \partial_t y - \partial_{xx} y - x^2 \partial_{\tilde x\tilde x} y =0 \text{ in } (0,T) \times \Omega, \end{equation} the following observability inequality holds:
$$
\int_{\Omega} \vert y(T) \vert^2 \leq \cc \int_0^T \int_\omega \vert y \vert^2.
$$
Moreover, such inequality does not hold for any time horizon $T< T^*$. In other words, the two-dimensional Baouendi-Grushin equation \eqref{grush1D} is observable through $\omega$ for $T>T^*$, and is not for $T<T^*$.

A few years later, in the same rectangular setting, it was shown in \cite{Koenig2017} that the value of the critical time $T^*$ is not only related
to the distance of the observation set $\omega$ to the $\tilde x$-axis -- where the equation degenerates -- but also depends strongly on the geometry of $\omega$. Specifically, the author shows that if $\omega = (-1,1) \times ((0,1)\setminus [\tilde{x}_1,\tilde{x}_2])$, for some $0\leq \tilde{x}_1 <\tilde{x}_2 \leq 1$,
then the Baouendi-Grushin equation \eqref{grush1D} is not observable trough $\omega$ for any $T>0$.  

\begin{figure}[h!]
\begin{center}
\begin{tikzpicture}[scale = 1.5]
\draw (-2,0) rectangle (2,1);
\draw[fill=teal!50] (-2,0) -- (-2,0.15) -- (2,0.15) -- (2,0);
\draw[fill=teal!50] (-2,0.4) -- (-2,1) -- (2,1) -- (2,0.4) -- (-2,0.4);
\node (a) at (0.5,0.7) {$\omega$};
\draw[->] (-2.2,0) -- (2.2,0) node[above]{$x$};
\draw[->] (0,-0.2) -- (0,1.2) node[right]{$\tilde{x}$};
\end{tikzpicture}
\caption{Geometrical configuration studied in \cite{Koenig2017} -- In this situation, the Baouendi-Grushin equation is not observable for any $T>0$.}
\end{center}
\end{figure}
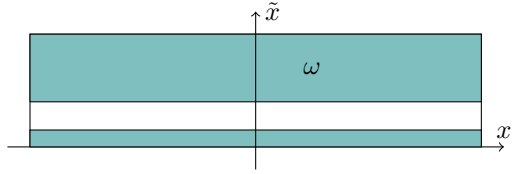

Since these pioneering works, significant effort has been devoted to understanding the observability properties of the Baouendi-Grushin equation. The two-dimensional rectangular setting has been the most extensively studied  \cite{Allonsius2021,duprezkoenig20,Beauchard2020,dardekoenigroyer23,laurentleautaud23,beauchardmillermorancey15}, and thus remains the best-understood case. In particular, the most general configuration for which the critical time is known reads as follows \cite[Theorem 1.6]{dardekoenigroyer23}: set $q \in C^3([-1,1];\R)$ such that $q(0) = 0$ and $\min_{[-1,1]} q' > 0$. Let $\gamma_1$, $\gamma_2$
be two continuous functions from $[0,1]$ to $[-1,1]$, such that for all $\tilde x \in [0,1]$, $\gamma_1(\tilde{x}) < \gamma_2(\tilde{x})$.
Define 
$$
\omega := \left\lbrace (x,\tilde{x}) \in [-1,1] \times [0,1], \ \gamma_1(\tilde{x}) < x < \gamma_2(\tilde{x}) \right\rbrace.
$$
Set 
$$
\beta = \max_{\tilde{x}\in[0,1]}\gamma_1(\tilde{x}), \quad \alpha = \min_{\tilde{x}\in [0,1]} \gamma_2(\tilde{x}),
$$
and define
$$
T^* = \frac{1}{q'(0)} \max \left( \int_0^{\min(\alpha,0)} q(s) ds, \int_0^{\max(\beta,0)} q(s) ds \right).
$$
Then the generalized Baouendi-Grushin equation $\partial_t y - \partial_{xx}y - q(x)^2 \partial_{\tilde{x}\tilde{x}}y=0$ is observable through $\omega$
for $T>T^*$, and is not for $T<T^*$.

\begin{figure}[h!]
\begin{center}
\begin{tikzpicture}[scale=1.8]
 \draw (-2,0) rectangle (2,1);
 \draw[fill=teal!50] (-1,1) .. controls (-1.9,0.7) and (-1.4,0.6) .. (-1.1,0.5) node (g1) {} .. controls (-0.8,0.4) and (0.5,0.3) .. (0.5,0.2) node (p) {} .. controls (0.5,0.1) and (-0.9,0.1) .. (-1,0) .. controls (-1,0) and (1.2,0) .. (1.2,0) .. controls (1.2,0.2) and (0.9,0.4) .. (0.6,0.5) node (g2) {} .. controls (0.3,0.6) and (-0.6,0.7) .. (-0.6,0.8) node(m){} .. controls (-0.6,0.9) and (1.1,0.9) .. (1.1,1) .. controls (1.1,1) and (-1,1) .. (-1,1);

 \node[] at (-0.9352,0.7553) {$\omega$};
 
 \draw[->] (-1.6,0.5)node[left]{$\gamma_1$} -- (g1.center);
 \draw[->] (1.3,0.6) node[right]{$\gamma_2$} -- (g2.center);

 \draw[dashed] (m.center) {} |- (0,0) node[pos=0.5,below left]{$\alpha$};
 \draw[dashed] (p.center) {} |- (0,0) node[pos=0.5,below right]{$\beta$};

 \draw[->] (-2.2,0) -- (2.2,0) node[above]{$x$};
 \draw[->] (0,-0.2) -- (0,1.2) node[right]{$\tilde{x}$};
\end{tikzpicture}
\caption{Geometrical configuration studied in \cite{dardekoenigroyer23}, in which the critical time of observability is precisely determined.}
\end{center}
\end{figure}
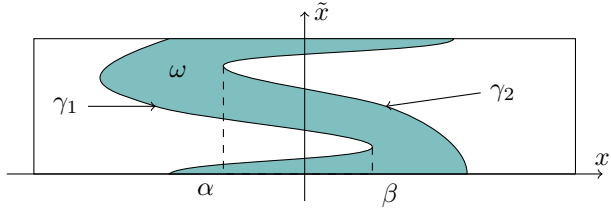

Note that, in the specific case where the observation set  $\omega$ is a vertical strip $(a,b) \times (0,1)$ with $a<b$, corresponding to 
$\gamma_1(\tilde{x}) = a$ and $\gamma_2(\tilde{x}) = b$, and $q(x) = x$ -- which was the situation studied \cite{Beauchard2013} -- the critical time is given by 
$$
T^* = \frac{\max(a,0)^2 + \min(b,0)^2}{2}.
$$ 
The non-rectangular two-dimensional case has also been studied, notably in \cite{tamekue22} on the sphere $\mathbb{S}^2$ and in \cite{vanlaere26} on arbitrary two-dimensional manifolds.

\medskip

The observability properties of the Baouendi-Grushin parabolic equation in higher-dimensional configurations have been far less studied \cite{BCY,Beauchard2020,vanla2025}. Since we {previously} discussed these results, we will not elaborate on them here.

\medskip

To conclude this section, we first emphasize that the unique continuation properties of the Baouendi-Grushin equation are now well understood \cite{laurentleautaud20,morancey15}.
 We also note that other parabolic equations exhibiting positive critical times of observability have been studied, including  the Kolmogorov$-v^2$  equation \cite{beauchard14,beauchardhelfferetal15,darderoyer21} and the three-dimensional Heisenberg-heat equation \cite{Beauchard2020}.  Additionally, positive critical observability times arise in certain systems of parabolic equations -- see \cite{assiafranckmorgan20} and the references therein. Finally, the observability properties of the Baouendi-Grushin wave equation and Baouendi-Grushin
Schr\"odinger equation have also been explored \cite{letrouit21,burqsun22,letrouitsun23}, and the observability and unique continuation properties of degenerate parabolic equations of Baouendi-Grushin type posed in the whole space are an active area of reseach \cite{lissy25,alphonseseelmann24,alphonseseelmann25,jamingwang25}.

\subsection{Outline}

{
The article is structured as follows: in \cref{sec_Carleman_and_consequences}, we first derive Carleman estimates for heat equations with
potential, and then apply them to obtain quantitative observability estimate for the harmonic heat equation observed from an annulus or the boundary of the ball. Then,  \cref{sec_restriction_arguments} is dedicated to the restriction argument, which is first derived in an abstract 
framework, and then applied to obtain the observability of the Baouendi-Grushin system in times greater than $T^*$. Finally, \cref{optimality} addresses optimality questions in the internal observation case, focusing first on the optimality of $T^*$, then on the blow-up rate of the observability cost.

The article includes four appendices.
In \cref{appendix_func_setting}, we discuss the well-posedness of the evolution equations considered in this study.
\Cref{proofdissipationgrushin} presents proofs of various useful energy-type estimates, while \cref{prooflemH10M} focuses on the decomposition of solutions to the harmonic heat equation in terms of spherical harmonics.
Finally, \cref{appendix_technical_results} contains proofs of several technical results.
}

\section{Observability inequalities for the harmonic heat equation from annuli and boundaries of balls}
\label{sec_Carleman_and_consequences}

\subsection{Three Carleman-type estimates}
\label{carlemaneries}

Let $\Omega$ be a smooth bounded domain of $\R^d$, and $V : \overline{\Omega} \mapsto \R^d$ be smooth. 
For $T > 0$, we set $Q = (0,T) \times \Omega$.

\begin{prop} \label{prop_Carleman_basis_comp}
Let  $s \geq 1$,  $\phi$ and $w$ two smooth functions defined on $Q$, satisfying the following properties:
\begin{enumerate}
\item $\displaystyle \partial_t w + s \partial_t \phi w - \Delta w - 2 s \nabla \phi \cdot \nabla w 
- s \Delta \phi w - s^2 \Vert \nabla \phi \Vert^2 w + V \, w  = F \in L^2(Q)$,
\item $w = 0$ on $(0,T) \times \partial \Omega$,
\item 
$
\displaystyle \Vert w(t,.) \Vert_{L^2(\Omega)} + \Vert \nabla w(t,.) \Vert_{L^2(\Omega)} + \Vert \partial_{t} \phi(t,.) \,  w (t,.)^2
\Vert_{L^1(\Omega)}  \xrightarrow{t\rightarrow 0, T} 0.
$
\end{enumerate}

\smallskip 

The following inequality holds:
\begin{multline} \label{eq_Carleman_basis_ineq}
s \int_0^T\int_{\partial \Omega} \frac{\partial \phi}{\partial n}  \left\vert \frac{\partial w}{\partial n} \right\vert^2
+ s \int_Q G_1 w^2 + s^3 \int_Q G_2 w^2 + s \int_Q H(\nabla w) \\ \leq \int_Q \vert F \vert^2 
+ s^2  \int_Q \vert \Delta \phi \, w \vert^2,
\end{multline}
where
$$
G_1  = \nabla V \cdot \nabla \phi - V \Delta \phi,
$$
\begin{multline*}
G_2  = - 2 \Hess(\phi)(\nabla \phi,\nabla \phi) + \Vert \nabla\phi \Vert^2 \Delta \phi
\\+ \frac{1}{s} \left( \nabla\cdot (\partial_t \phi \nabla \phi) -2 \partial_t \phi \Delta \phi + \partial_t \left(\Vert\nabla \phi\Vert^2 \right) \right)
- \frac{1}{2 s^2} \left(  \partial_{tt} \phi -2  \Delta^2 \phi  \right)  ,
\end{multline*}
and
$$
H(\xi) = -2 \Hess(\phi)(\xi,\xi) - \Delta \phi \Vert \xi \Vert^2, \ \forall \xi \in \R^d.
$$
\end{prop}

\begin{proof}
We set
$$
P_1 w = s\partial_t \phi w -\Delta w - s^2 \Vert \nabla \phi \Vert^2 w + V w,
\quad P_2 w = \partial_t w - 2 s \nabla \phi \cdot \nabla w -2 s \Delta \phi w.
$$
We have
$$
\Vert P_1 w \Vert_{L^2(Q)}^2 + 2 \underbrace{\left( P_1 w , P_2 w \right)_{L^2(Q)}}_{:= I}
+ \Vert P_2 w \Vert_{L^2(Q)}^2 
\leq 2 \Vert F \Vert_{L^2(Q)}^2 + 2 s^2  \Vert \Delta \phi \, w \Vert^2_{L^2(Q)}.
$$
We write $\displaystyle I = \sum_{p=1}^4\sum_{q = 1}^3 I_{pq}$, where each term $I_{pq}$ corresponds to the scalar product between the $p$-th term of $P_1 w$
and the $q$-th term of $P_2w$.

\medskip 

\noindent $\bullet$ Computation of $I_{11}$ -- We have $$
I_{11} =s  \int_Q \partial_t \phi w \partial_t w = \frac{s}{2} \int_Q \partial_t \phi \partial_t w^2 = - \frac{s}{2} \int_Q \partial_{tt}\phi w^2. 
$$

\noindent $\bullet$ Computation of $I_{12}$ -- We have
$$
I_{12} = - 2 s^2 \int_Q \partial_t \phi w \nabla \phi \cdot \nabla w = - s^2 \int_Q \partial_t \phi \nabla \phi \cdot \nabla w^2
= s^2 \int_Q \nabla \cdot (\partial_t \phi \nabla \phi) w^2.
$$

\noindent $\bullet$ Computation of $I_{13}$ -- We have $\displaystyle I_{13} =  -2 s^2 \int_\Omega \partial_t \phi \Delta \phi w^2$.

\noindent $\bullet$ Computation of $I_{21}$ -- We have
$$
I_{21} = -\int_Q \Delta w \partial_t w =  \int_Q \nabla w \nabla \partial_t w =  \frac{1}{2} \int_Q \partial_t \Vert \nabla w \Vert^2 = 0.
$$

\noindent $\bullet$ Computation of $I_{22}$ -- We have
$$
I_{22} = 2 s \int_Q \Delta w \nabla \phi \cdot \nabla w = 2 s \int_0^T \int_{\partial \Omega}  \frac{\partial w}{\partial n} \nabla \phi \cdot \nabla w 
- 2 s \int_Q \nabla w \cdot \nabla (\nabla \phi \cdot \nabla w).
$$
Here, we note that, as $w = 0$ on $\partial \Omega$, we have $\frac{\partial w}{\partial n}  \nabla \phi \cdot \nabla w = \frac{\partial \phi}{\partial n} \left\vert \frac{\partial w}{\partial n} \right\vert^2$. Furthermore, we have the identity
$$
\nabla w \cdot \nabla (\nabla \phi \cdot \nabla w)  = \Hess(\phi)(\nabla w, \nabla w) + \frac{1}{2} \nabla \phi \cdot \nabla \Vert \nabla w \Vert^2.
$$
Consequently, we obtain
\begin{multline*}
I_{22} = 2 s \int_0^T\int_{\partial \Omega} \frac{\partial \phi}{\partial n} \left\vert \frac{\partial w}{\partial n} \right\vert^2 -2 s \int_Q \Hess(\phi)(\nabla w, \nabla w)
- s \int_Q \nabla \phi \cdot \nabla \Vert \nabla w \Vert^2 \\
 =  s \int_0^T\int_{\partial \Omega} \frac{\partial \phi}{\partial n} \left\vert \frac{\partial w}{\partial n} \right\vert^2 -2 s \int_Q \Hess(\phi)(\nabla w, \nabla w)
 + s \int_Q \Delta \phi \Vert \nabla w \Vert^2.
\end{multline*}

\noindent $\bullet$ Computation of $I_{23}$ -- We have
\begin{multline*}
I_{23} = 2 s \int_Q \Delta w \Delta \phi w = -2s \int_Q \Delta \phi \Vert\nabla w \Vert^2 -2 s \int_Q \nabla (\Delta \phi) \cdot \nabla w w
\\ = -2s \int_Q \Delta \phi \Vert\nabla w \Vert^2 - s  \int_Q \nabla (\Delta \phi ) \cdot \nabla w^2
 = -2s \int_Q \Delta \phi \Vert\nabla w \Vert^2 + s \int_Q \Delta^2 \phi\,    w^2.
\end{multline*}

\noindent $\bullet$ Computation of $I_{31}$ -- We have
$$
I_{31} = -s^2 \int_Q \Vert \nabla \phi \Vert^2 w \partial_t w  = - \frac{s^2}{2} \int_Q \Vert \nabla \phi \Vert^2 \partial_t w^2
= \frac{s^2}{2} \int_Q \partial_t \left(\Vert \nabla \phi \Vert^2\right) w^2.
$$

\noindent $\bullet$ Computation of $I_{32}$ -- We have
$$
I_{32} = 2 s^3 \int_Q \Vert\nabla \phi \Vert^2 w \nabla \phi \cdot \nabla w = s^3 \int_Q \Vert \nabla \phi \Vert^2 \nabla \phi \cdot \nabla w^2
 = - s^3 \int_Q \nabla \cdot(\Vert \nabla \phi \Vert^2 \nabla \phi) w^2.
$$

\noindent $\bullet$ Computation of $I_{33}$ -- We have
$
\displaystyle I_{33} = 2 s^3 \int_Q \Vert \nabla \phi \Vert^2 \Delta \phi w^2.
$

\noindent $\bullet$ Computation of $I_{41}$ -- We have
$
\displaystyle I_{41} =  \int_Q V w \partial_t w = \frac{1}{2} \int_Q V \partial_t w^2 = 0.
$

\noindent $\bullet$ Computation of $I_{42}$ -- We have
$$
I_{42} = -2 s  \int_Q V w \nabla \phi \cdot \nabla w = - s \int_Q V \nabla \phi \cdot \nabla w^2
 = s  \int_Q \nabla \cdot (V \nabla \phi)  w^2.
$$

\noindent $\bullet$ Computation of $I_{43}$ -- We have
$
\displaystyle I_{43} = -2 s  \int_Q V \Delta \phi w^2.
$

\medskip  The result follows immediately.
\end{proof}

\medskip

We now define the function \( \vartheta \) as
$
\vartheta : s \in (0,1) \mapsto (4s(1-s))^{-1} \in [1, \infty).
$
This function satisfies the following bounds for all \( s \in (0,1) \):
\[
|\vartheta'(s)| \leq 4\,\vartheta(s)^2, \quad |\vartheta''(s)| \leq 32\,\vartheta(s)^3.
\]
For any \( T > 0 \), we introduce the rescaled function
\begin{equation} \label{eq_def_varpi}
\displaystyle \varpi : t \in (0,T) \mapsto \vartheta\left(\frac{t}{T}\right).
\end{equation}
Next, we introduce a smooth $\psi : \overline{\Omega} \mapsto [0,\infty[$ such that $\Vert \nabla \psi \Vert \geq c_\psi >0$ in $\overline{\Omega}$.
We define \begin{equation} \label{omegaplus} \partial \Omega_{+} = \left\lbrace x \in \partial \Omega, \ \nabla \psi \cdot n > 0 \right\rbrace. \end{equation}
For any $\lambda \geq 1$, we set 
\begin{equation} \label{eq_def_varphi}
\varphi = \exp(4\lambda \Vert \psi \Vert_\infty) - \exp(\lambda (\psi + 2 \Vert \psi \Vert_\infty)), 
\quad \eta = \exp(\lambda (\psi+ 2\Vert\psi \Vert_\infty)).
\end{equation}
By definition, $\varphi >0$ and $\eta \geq 1$ on $\overline{\Omega}$, and the following relations hold:
\begin{align*}
\nabla \varphi &= - \nabla \eta = - \lambda \nabla \psi \eta,\\
 \Delta \varphi &= - \lambda^2 \Vert \nabla \psi \Vert^2 \eta
 - \lambda\Delta \psi \eta, 
 \\  \Hess(\varphi) &= - \lambda^2 \nabla \psi^T \nabla \psi \eta - \lambda \Hess(\psi) \eta.
\end{align*}
Finally, we define the function $\phi : (t,x) \in (0,T) \times \overline{\Omega} \mapsto \varpi(t) \varphi(x)$.

\begin{lem} \label{lemma_Carleman_G2H}
There exist constants \(\lambda_0 \geq 1\), $\s_0 \geq 1$ and \(c > 0\)---each depending only on \(\psi\)---such that  for any  $\lambda \geq \lambda_0$,  for any \(s \geq \left(1 + \frac{1}{T}\right) \s_0\) and for all \(\xi \in \mathbb{R}^d\), the following inequalities hold:
\[
G_2 \geq c \lambda^4 \varpi^3 \eta^3, \quad H(\xi) \geq c \lambda^2 \varpi \eta \Vert \xi \Vert^2,
\]
where \(G_2\) and \(H\) are defined in \cref{prop_Carleman_basis_comp}.

\end{lem}

\begin{proof}
For all $\xi$ in $\R^d$, we have
\begin{multline*}
H(\xi) = \lambda^2 \left( 2 \vert \nabla \xi \cdot \nabla \eta \vert^2 + \frac{2}{\lambda} \Hess(\psi)(\xi,\xi) + \Vert \nabla \psi\Vert^2 \Vert \xi \Vert^2 +\frac{1}{\lambda} \Delta \psi \Vert \xi \Vert^2 \right)  \eta \varpi
\\
\geq \lambda^2  \left( c_\psi - \frac{1}{\lambda} \left( 2 \frac{\Hess(\psi)(\xi,\xi)}{\Vert \xi \Vert^2} +  \vert \Delta \psi \vert  \right)  \right) \Vert \xi \Vert^2  \varpi\eta .
\end{multline*}
As a consequence, there exists $\lambda_1$, depending only on $\psi$, such that for all $\lambda \geq \lambda_1$, 
$$
H(\xi) \geq \frac{c_\psi}{2}  \lambda^2 \Vert \xi \Vert^2 \varpi \eta.
$$
We decompose $G_2$ as  $G_{21} + G_{22}$, where 
$$
G_{21} = - 2 \Hess(\phi)(\nabla \phi,\nabla \phi) + \Vert \nabla\phi \Vert^2 \Delta \phi,
$$
and 
$$
G_{22} = \frac{1}{s} \left( \nabla\cdot (\partial_t \phi \nabla \phi) - 2\partial_t \phi \Delta \phi + \partial_t \left(\Vert\nabla \phi\Vert^2 \right) \right)
- \frac{1}{2 s^2} \left(  \partial_{tt} \phi - 2  \Delta^2 \phi  \right).
$$
We have
$$
G_{21} =  \lambda^4 \Vert \nabla \psi \Vert^4 \varpi^3 \eta^3 + 2 \lambda^3 \Hess( \psi)(\nabla \psi,\nabla \psi) \varpi^3 \eta^3  -\lambda^3 \Vert \nabla \psi \Vert^2 \Delta \psi \varpi^3 \eta^3.
$$
Consequently, we obtain
$$
G_{21} \geq   \lambda^4 \varpi^3 \eta^3 \left(c_\psi^4 -\frac{1}{\lambda} \left( 2 \vert \Hess(\psi)(\nabla \psi,\nabla \psi)\vert +  \Vert \nabla \psi \Vert^2 \vert \Delta \psi \vert \right) \right),
$$
which implies the existence of $\lambda_2$, depending only on $\psi$, such that for all $\lambda\geq \lambda_2$, 
$$
G_{21} \geq \frac{c_\psi^4}{2} \lambda^4 \varpi^3 \eta^3 c_\psi^4.
$$
We now fix $\lambda \geq  \max(\lambda_1,\lambda_2)$, and turn our attention to $G_{22}$. 
As by construction, we have $\varphi \leq \eta^2$, it is readily seen that there exists
there exists a constant $\mathcal{c}$, depending only on $\psi$, such that 
$$
\vert G_{22}\vert \leq \mathcal{c} \left(\frac{1}{sT} + \frac{\left(1+ \frac{1}{T^2}\right)}{s^2} \right) \lambda^4 \varpi^3 \eta^3.
$$
As for all $s \geq \left(\frac{1}{T}+1\right)\s_0$, we have
$$
\left(\frac{1}{sT} + \frac{\left(1+ \frac{1}{T^2}\right)}{s^2} \right) \leq \frac{1}{\s_0} + \frac{1}{\s_0^2},
$$
we deduce the existence of $\s_0 \geq 1$, depending on $\psi$, such that
for all $s \geq   \left(\frac{1}{T}+1\right)s_0$ gives
$$
G_{2} \geq G_{21} - \vert G_{22} \vert \geq  \frac{c_\psi^4}{4} \lambda^4 \omega^3 \eta^3  s^3.
$$
The result follows.
\end{proof}

\begin{prop} \label{prop_generic_Carleman_V}
There exits three constants $c$, $\lambda_0$ and $s_0$, greater than one, and depending only on $\psi$, such that for all $\lambda\geq \lambda_0$, 
for all $s \geq \left(1+ \frac{1}{T} + \Vert V \Vert_{W^{1,\infty}}^{\frac{1}{2}} \right)s_0$, for all $y \in C^0([0,T];H^1_0(\Omega))\cap L^2((0,T),H^2(\Omega))\cap H^1((0,T),L^2(\Omega))$, we have
\begin{multline*}
s^3 \lambda^4 \int_Q \varpi^3 \eta^3 \vert y \vert^2 e^{-2 s \phi} + s \lambda^2 \int_Q \varpi \eta \Vert \nabla y \Vert^2 e^{-2s\phi}
\leq c \bigg( \int_Q \vert \partial_t y - \Delta y + V y \vert^2 e^{-2 s \phi} 
\\ + s \lambda \int_0^T \int_{\partial \Omega_+} \varpi \eta \left\vert \frac{\partial \psi}{\partial n} \right\vert \left\vert \frac{\partial y}{\partial n} \right\vert^2  e^{-2 s \phi} \bigg),
\end{multline*}
where $\partial \Omega_+$ is defined in \eqref{omegaplus}.
\end{prop}

\begin{proof}
We define \( f = \partial_t y - \Delta y + V y \in L^2(Q) \). For \( s \geq 1 \), let the conjugated functions be \( w = y e^{-s\phi} \) and \( F = f e^{-s\phi} \). It follows immediately that \( w \) satisfies all the assumptions of \cref{prop_Carleman_basis_comp}, so equation~\eqref{eq_Carleman_basis_ineq} holds.

Our focus now turns to the term \( G_1 \) in~\eqref{eq_Carleman_basis_ineq}. A direct computation yields a constant \(\kappa\), depending only on \(\psi\), such that
\[
|G_1| \leq \kappa \|V\|_{W^{1,\infty}} \lambda^4 \varpi^3 \eta^3.
\]
Together with \cref{lemma_Carleman_G2H}, this implies the existence of \(\lambda_0 \geq 1\), \(\s_0 \geq 1\), and \(\mathcal{c} > 0\), all depending only on \(\psi\), such that for all \(\lambda \geq \lambda_0\) and all \(s \geq \left(1 + \frac{1}{T}\right)\s_0\),
\[
G_2 - |G_1| \geq \mathcal{c} \lambda^4 \varpi^3 \eta^3 \left(1 - \frac{\|V\|_{W^{1,\infty}}}{s^2}\right).
\]
It follows immediately that there exists \(s_0 \geq 1\) such that, for all \(\lambda \geq \lambda_0\) and all \(s \geq \left(1 + \frac{1}{T} + \|V\|_{1,\infty}^{\frac{1}{2}}\right)s_0\),
\[
G_2 - |G_1| \geq \frac{\mathcal{c}}{2} \lambda^4 \varpi^3 \eta^3.
\]
By combining this with \cref{prop_Carleman_basis_comp} and \cref{lemma_Carleman_G2H}, and noting that by definition
\[
\nabla y\, e^{-s \phi} = \nabla w - s \lambda \nabla \psi \varpi \eta w,
\]
the proof is complete.
\end{proof}

\begin{prop} \label{prop_Carleman_V_positive}
Assume $0 < c_- \leq V$ on $\overline{\Omega}$, and $\Vert V \Vert_{W^{1,\infty}} \leq c_+$, for some positive constants $c_-$, $c_+$.
There exists two constants $\lambda_0 \geq 1$ and $c>0$, depending only on $c_-$, $c_+$ and $\psi$, and a constant $s_0$, depending only on $\psi$, such that for all $\lambda\geq \lambda_0$, $s \geq \left(1 + \frac{1}{T} \right) s_0$, for all $\mu \geq 0$ and all $y \in C^0([0,T];H^1_0(\Omega))\cap L^2((0,T),H^2(\Omega))\cap H^1((0,T),L^2(\Omega))$
\begin{multline*}
s^3 \lambda^4 \int_Q \varpi^3 \eta^3 \vert y \vert^2 e^{-2 s \phi} +  s\lambda^2 \mu^2 \int_Q \varpi \eta \vert y \vert^2 e^{-2 s \phi} + s \lambda^2 \int_Q \varpi \eta \Vert \nabla y \Vert^2 e^{-2s\phi}
\\ \leq c \bigg( \int_Q \vert \partial_t y - \Delta y + \mu^2 V  y \vert^2 e^{-2 s \phi} 
+ s \lambda \int_0^T \int_{\partial \Omega_+} \varpi \eta \left\vert \frac{\partial \psi}{\partial n} \right\vert \left\vert \frac{\partial y}{\partial n} \right\vert^2  e^{-2 s \phi} \bigg),
\end{multline*}
where $\partial \Omega_+$ is defined in \eqref{omegaplus}.
\end{prop}

\begin{proof}
As before, we define \( f = \partial_t y - \Delta y + \mu^2 V y \in L^2(Q) \). For \( s \geq 1 \), introduce the conjugated functions \( w = y e^{-s\phi} \) and \( F = f e^{-s\phi} \). It follows that \( w \) satisfies all the assumptions of \cref{prop_Carleman_basis_comp} with \( V  \) replaced by $\mu^2 V$, ensuring that equation~\eqref{eq_Carleman_basis_ineq} holds.

Turning our attention again to the term \( G_1 \) in~\eqref{eq_Carleman_basis_ineq}, we note that, in the present case, it can be expressed as
\[
G_1 = \mu^2 (\nabla V \cdot \nabla \phi - V \Delta \phi) = \mu^2 \tilde{G}_1.
\]
The key insight is that \(\tilde{G}_1\) can be made positive by choosing \(\lambda\) sufficiently large, since $V$ is bounded away from zero in \(\overline \Omega\). Indeed, we have
$$
\tilde G_1 = \lambda^2 \varpi \eta \Vert \nabla \psi \Vert^2 V + \lambda \varpi \eta \Delta \psi V - \lambda \varpi \eta \nabla \psi \nabla V,
$$
which immediately implies
$$
\tilde G_1 \geq \lambda^2 \varpi \eta c_\psi^2 c_- \left( 1 - \frac{1}{\lambda} \frac{\Vert \psi \Vert_{W^{2,\infty}}c_+}{c_\psi^2\, c_-} \right).
$$
Consequently, there exist \(\lambda_1 \geq 1\) and \(\mathcal{c} > 0\), depending only on \(c_-\), \(c_+\), and \(\psi\), such that for all \(\lambda \geq \lambda_1\),
\[
\tilde{G}_1 \geq \mathcal{c} \lambda^2 \varpi \eta.
\]
The proof is then concluded as before.
\end{proof}

\begin{prop} \label{prop_BDE20}
Let $\mu > 0$. Consider a function $w : (0,\infty) \times \mathcal{B}_r\mapsto \R$ such that $w \in C^0([0,T];H^1_0(\mathcal{B}_r))\cap L^2((0,T);H^2(\mathcal{B}_r))\cap H^1((0,T);L^2(\mathcal{B}_r))$ for all $T>0$. Assume the following properties hold:
\begin{enumerate}
\item $\displaystyle \partial_t w - \Delta w + \mu \coth\left(2 \mu t\right) \left(2 x \cdot \nabla w + d\, w\right) - \frac{r^2 \mu^2}{\sinh\left(2\mu t\right)^2}w = F \in L^2((0,T)\times \mathcal{B}_r)$, $\forall T >0$.
\item $\Vert w(t,.) \Vert_{L^2(\mathcal{B}_r)} + t \Vert \nabla w (t,. ) \Vert_{L^2(\mathcal{B}_r)} \xrightarrow{t\rightarrow 0}0$.
\end{enumerate}

\medskip 
Then for all $T>0$,
\begin{multline*}
\int_{\mathcal{B}_r} \left( \sinh(2 \mu T)^2 \Vert \nabla w(T) \Vert^2 - \mu^2 r^2 \vert w(T) \vert^2 \right)
\leq \int_0^T \sinh(2\mu t)^2 \int_{\mathcal{B}_r} \vert F \vert^2 
\\ + \mu r  \int_0^T \sinh(4\mu t) \int_{\partial \mathcal{B}_r}  \left\vert \frac{\partial w}{\partial n} \right\vert^2.
\end{multline*}
\end{prop}

\begin{proof}
The proof follows the same reasoning as in \cite[Proposition 2.1]{Beauchard2020}, with a minor adjustment to account for the source term. As such, we only outline the essential steps here.

We set $\theta(t) = \mu \coth(2\mu t)$, and for $\mathcal{w} \in H^2 \cap H^1_0(\mathcal{B}_r)$,
$$
S \mathcal{w} = - \Delta \mathcal{w} + \theta'(t) \frac{r^2}{2} \mathcal{w}, \quad
A \mathcal{w} = \theta(t) \left(2 x \cdot \nabla  \mathcal{w} + d\,  \mathcal{w}\right).
$$
By assumption, we have $\partial_t w + S w + A w = F$ on $(0,\infty) \times \mathcal{B}_r$. For $t>0$, we define
$$
D(t) = \int_{\mathcal{B}_r} Sw \, w = \int_{\mathcal{B}_r} \left(\Vert \nabla w \Vert^2 + \theta'(t) \frac{r^2}{2} \vert w \vert^2\right).
$$
A direct computation gives
$$
D'(t) = \theta''(t) \frac{r^2}{2} \int_{\mathcal{B}_r} \vert w \vert^2 -  2\int_{\mathcal{B}_r} \vert Sw \vert^2  
- 2 \int_{\mathcal{B}_r} Sw \, Aw + 2 \int_{\mathcal{B}_r} Sw \, F,
$$
which immediately implies
$$
D'(t) + \int_{\mathcal{B}_r} \vert S w \vert^2 \leq \theta''(t) \frac{r^2}{2} \int_{\mathcal{B}_r} \vert w \vert^2
- 2 \int_{\mathcal{B}_r} Sw \, Aw + \int_{\mathcal{B}_r} \vert F \vert^2.
$$
As 
\begin{multline*}
-2 \int_{\mathcal{B}_r} Sw\, Aw = 2 \theta(t) \int_{\mathcal{B}_r} \Delta  w (2 x \cdot \nabla w + d w)
 \\ = - 4 \theta(t) \int_{\mathcal{B}_r} \Vert \nabla w \Vert^2 + 2 \theta(t) r \int_{\partial \mathcal{B}_r}
  \left\vert \frac{\partial w}{\partial n} \right\vert^2,
\end{multline*}
using that $\theta''(t) = - 4 \theta(t) \theta'(t)$, we obtain
$$
D'(t) + 4 \theta(t) D(t) + \int_{\mathcal{B}_r} \vert Sw \vert^2 \leq \int_{\mathcal{B}_r} \vert F \vert^2 + 2 \theta(t) r \int_{\partial \mathcal{B}_r} \left\vert \frac{\partial w}{\partial n} \right\vert^2.
$$
We deduce that 
$$
(D(t) e^{\Theta(t)})' \leq e^{\Theta(t)} \int_{\mathcal{B}_r} \vert F \vert^2
+ 2 \theta(t) r e^{\Theta(t)} \int_{\partial \mathcal{B}_r}
 \left\vert \frac{\partial w}{\partial n} \right\vert^2,
$$
where we have set
$$
\Theta(t) = -4 \int_t^T \theta(s)  = \ln \left( \frac{\sinh(2\mu t)^2}{\sinh(2\mu T)^2} \right).
$$
The result follows from a straightforward integration over the interval \([0, T]\), using the behavior of \(w\) at \(0\).
\end{proof}

\begin{cor} \label{coro_BDE20}
Let \(\beta > 1\), \(\alpha = 1 - \frac{1}{\beta} \in (0,1)\), and let $R$ be a positive constant.
There exist a positive constant \(c\), depending only on $R$, and a positive constant $a$, depending on $\alpha$ and $R$, such that, for all $r$, $0 < r \leq R$, and for all \(T\) and \(\mu\) such that \(T\mu \geq a\mu^\alpha\),
any function \(w\) satisfying the assumptions of \cref{prop_BDE20} fulfills the following estimates
:
$$
\int_{\mathcal{B}_r} \vert w(T) \vert^2 \leq c \bigg( \int_0^T \frac{\sinh(2\mu t)^2}{\sinh(2\mu T)^2} \int_{\mathcal{B}_R} \vert F \vert^2 
 + \mu   \int_0^T \frac{\sinh(4\mu t)}{\sinh(2\mu T)^2} \int_{\partial \mathcal{B}_R} \left\vert \frac{\partial w}{\partial n} \right\vert^2 \bigg), 
$$
and 
\begin{multline*}
\int_{\mathcal{B}_r} \Vert \nabla w (T) \Vert^2 \leq c \left(1 + \frac{1}{T^2}\right)
\bigg( \int_0^T \frac{\sinh(2\mu t)^2}{\sinh(2\mu T)^2} \int_{\mathcal{B}_R} \vert F \vert^2 
 \\ + \mu   \int_0^T \frac{\sinh(4\mu t)}{\sinh(2\mu T)^2} \int_{\partial \mathcal{B}_R} \left\vert \frac{\partial w}{\partial n} \right\vert^2 \bigg).
\end{multline*}
\end{cor}

\begin{proof}
For any $\rho > 0$, we denote 
$$
c_\rho = \inf_{\mathcal{w} \in H^1_0(\mathcal{B}_\rho)} \frac{\int_{\mathcal{B}_\rho} \Vert \nabla \mathcal{w}\Vert^2}{\int_{\mathcal{B}_\rho} \vert \mathcal{w}\vert^2}.
$$
Then $c_\rho>0$ is the square of the reciprocal of the optimal Poincar\'e constant, and for all $\mathcal{w} \in H^1_0(\mathcal{B}_\rho)$, 
$
\displaystyle c_\rho \int_{\mathcal{B}_\rho}\vert \mathcal{w} \vert^2 \leq \int_{\mathcal{B}_\rho} \Vert \nabla \mathcal{w}\Vert^2.
$
It is readily seen that, for all $0< \rho_1 \leq \rho_2$, $c_{\rho_2} \leq c_{\rho_1} $.

\medskip

Now, let $\mu> 0$, $T>0$, and $w$ be a function satisfying the assumptions of  \cref{prop_BDE20}. We have
\begin{multline*}
\int_{\mathcal{B}_R} \left( \Vert \nabla w(T) \Vert^2 - \frac{\mu^2 r^2}{\sinh(2 \mu T)^2} \vert w(T) \vert^2 \right)
\leq \int_0^T \frac{\sinh(2\mu t)^2}{\sinh(2\mu T)^2} \int_{\mathcal{B}_R} \vert F \vert^2 
\\ + \mu r  \int_0^T \frac{\sinh(4\mu t)}{\sinh(2\mu T)^2} \int_{\partial \mathcal{B}_R} \left\vert \frac{\partial w}{\partial n} \right\vert^2.
\end{multline*}
As $w(T,.)$ belongs to $H^1_0(\mathcal{B}_r)$, we have
\begin{multline*}
\int_{\mathcal{B}_R} \left( \Vert \nabla w(T) \Vert^2 - \frac{\mu^2 r^2}{\sinh(2 \mu T)^2} \vert w(T) \vert^2 \right)
\geq \int_{\mathcal{B}_R} \left(c_r - \frac{\mu^2 r^2}{\sinh(2 \mu T)^2}  \right) \vert w(T) \vert^2
\\ \geq \left(c_{R} - \frac{\mu^2 R^2}{\sinh(2 \mu T)^2} \right) \int_{\mathcal{B}_r} \vert w(T) \vert^2.
\end{multline*}
We set 
$$
K(\mu,T) = 1- \frac{\mu^2 R^2}{c_R \sinh(2 \mu T)^2}.
$$
Our goal is to show that there exists a sufficiently large constant \(a > 0\) such that, for all \(T\) and \(\mu\) satisfying \(T\mu \geq a\mu^\alpha\),
$K(\mu,T) \geq \frac{3}{4}$, or equivalently $\sinh(2\mu T) \geq 2 R c_R^{-\frac{1}{2}}\mu$. This follows directly from Lemma \ref{lemma_existence_a}, with $\mathcal c = 2 R c_R^{-\frac{1}{2}}$, which implies that $a$ depends solely on $\alpha$ and $R$.

For such a constant $a$, and for any $T$ and $\mu$ satisfying $T\mu \geq a \mu^\alpha$, we obtain the following estimate:
\begin{multline*}
\int_{\mathcal{B}_r} \vert w(T) \vert^2 \leq \frac{4}{3\,c_R}  
\bigg( \int_0^T \frac{\sinh(2\mu t)^2}{\sinh(2\mu T)^2} \int_{\mathcal{B}_R} \vert F \vert^2 
\\ + \mu R  \int_0^T \frac{\sinh(4\mu t)}{\sinh(2\mu T)^2} \int_{\partial \mathcal{B}_R} \left\vert \frac{\partial w}{\partial n} \right\vert^2\bigg).
\end{multline*}
Additionally, we have
\begin{multline*}
\int_{\mathcal{B}_r} \Vert \nabla w(T) \Vert^2 \leq \frac{\mu^2 r^2}{\sinh(2 \mu T)^2} \int_{\mathcal{B}_r} \vert w(T) \vert^2
+ \int_0^T \frac{\sinh(2\mu t)^2}{\sinh(2\mu T)^2} \int_{\mathcal{B}_R} \vert F \vert^2 
\\ + \mu r  \int_0^T \frac{\sinh(4\mu t)}{\sinh(2\mu T)^2} \int_{\partial \mathcal{B}_R} \left\vert \frac{\partial w}{\partial n} \right\vert^2.
\end{multline*}
The result follows immediately.
\end{proof}

\subsection{Observation from an annulus}
\label{obsfromannulus}
We define three positive constants, $R_1$, $R_2$, and $R$, such that $R_1 < R_2 \leq R$, and set 
$$
\mathcal A = \left\lbrace x \in \R^d, \ R_1 < \Vert x \Vert < R_2 \right\rbrace \subset \mathcal B_R.
$$
Our main result reads:

\begin{prop} \label{prop_observability_from_annulus}
There exists a positive constant $\cc$, and for all $\beta > 1$, 
a positive constant $c_\beta$ verifying $c_\beta \xrightarrow{\beta\rightarrow 1} \infty$,  such that for all $T>0$,  $\mu>0$,  $\varepsilon\in (0,1]$, and all $y$ in $C^0([0,T];L^2(\Omega))\cap L^2((0,T);H^1_0(\Omega))$ satisfying $\partial_t y - \Delta y + \mu^2 \Vert x \Vert^2 y = 0$ in $(0,T) \times \mathcal B_R$, the following estimate holds:
$$
\int_{\mathcal{B}_R} \vert y(T) \vert^2 \leq \frac{\cc}{\varepsilon^6} (1+\mu) e^{\mu (R_1+\varepsilon)^2} e^{\frac{2 c_\beta}{T^\beta}}\int_0^T \int_{\mathcal A } \vert y \vert^2.
$$
\end{prop}

It is enough to prove \cref{prop_observability_from_annulus} for $\varepsilon \in (0,1]$ such that $R_1 + \varepsilon < R_2$,
and, by density arguments, for $y \in C^0([0,T],H^1_0(\Omega))\cap L^2((0,T);H^2(\Omega)) \cap H^1((0,T);L^2(\Omega))$.
Throughout, we assume $\varepsilon$ is chosen in this way. We also fix $T$, $\mu$ positive, and consider $y$ in $C^0([0,T],H^1_0(\Omega))\cap L^2((0,T);H^2(\Omega)) \cap H^1((0,T);L^2(\Omega))$
satisfying $\partial_t y - \Delta y + \mu^2 \Vert x \Vert^2 y = 0$ in $(0,T) \times \mathcal B_R$. 
\cref{prop_observability_from_annulus} is a direct consequence of the following \cref{prop_observability_from_annulus_interior_subcase} and \cref{prop_observability_from_annulus_exterior_subcase}, as 
$$
\int_{\mathcal{B}_R} \vert y(T) \vert^2 
\leq \int_{\mathcal{B}_{R_1 + \frac{\varepsilon}{3}}} \vert y(T) \vert^2 +
\int_{\mathcal{B}_R \setminus \mathcal{B}_{R_1+\frac{\varepsilon}{6}}} \vert y(T) \vert^2.
$$

\begin{prop} \label{prop_observability_from_annulus_interior_subcase}
There exists a constant $c$, depending only on $R$, and a constant $\mathcal{c}_\beta$, depending on $\beta$ and $R$, and verifying $c_\beta \xrightarrow{\beta\rightarrow 1} \infty$, such that 
$$
\int_{\mathcal{B}_{R_1+\frac{\varepsilon}{3}}} \vert y(t) \vert^2
\leq \frac{c}{\varepsilon^6} (1+\mu) e^{\mu (R_1+\varepsilon)^2} e^{\frac{2c_\beta}{T^\beta}} \int_{\mathcal{A}} \vert y\vert^2.
$$
\end{prop}

\begin{proof} In what follows, \( \cc \) denotes a constant that depend only on \( R \), and whose value may vary from line to line.

Set $\alpha = 1 - \frac{1}{\beta}$. 
 Let $\chi_1 : \R^d \mapsto \R$ be smooth, and such that $\chi_1 = 1$ in $\mathcal{B}_{R_1+\frac{\varepsilon}{3}}$, and
$\chi_1 = 0$ in $\R^d \setminus \mathcal{B}_{R_1+\frac{2\varepsilon}{3}}$.
In the following, we denote 
$$
\mathcal{A}_\varepsilon = \left\lbrace x \in \R^d, \ R_1+\frac{\varepsilon}{3}< \Vert x \Vert < R_1 +\frac{2\varepsilon}{3} \right\rbrace
\subset \mathcal A.
$$
The definition of $\chi_1$ implies
$$
\Vert \nabla \chi_1 \Vert_{L^\infty} \leq \frac{\cc}{\varepsilon} \indicatrice{\mathcal{A}_\varepsilon}, \quad \Vert \Delta \chi_1 \Vert_{L^\infty} \leq \frac{\cc}{\varepsilon^2} \indicatrice{\mathcal{A}_\varepsilon}.
$$
Set $y_1 = \chi_1 y$. Then $y_1$ belongs to $C^0([0,T];H^1_0(\mathcal{B}_{R_1+\varepsilon}))\cap L^2((0,T);H^2(\mathcal{B}_{R_1+\varepsilon}))\cap H^1((0,T);L^2(\mathcal{B}_{R_1+\varepsilon}))$, satisfies $\nabla y_1 \cdot \nu = 0$
on $\partial \mathcal{B}_{R_1+\varepsilon}$ and 
$$
\partial_t y_1 - \Delta y_1 + \mu^2 \Vert x \Vert^2 y_1 = -2\nabla \chi_1 \cdot \nabla y - \Delta \chi_1 y \text{ in } (0,T) \times \mathcal{B}_{R_1+\varepsilon}.
$$
From this point onward, we will employ two distinct Carleman-type strategies, depending on the regime of \( T \) and \( \mu \).

\textbf{Case $T\mu$ large --}
Set $w_1 = \exp\left(- \frac{\mu}{2} \coth\left(2\mu t\right)((R_1+\varepsilon)^2 - \Vert x\Vert^2 ) \right) \, y_1$. Then $w_1$ satisfies the assumptions of \cref{coro_BDE20}, from which we deduce the existence of $a$, depending only on $R$ and $\alpha$ such that, if $T$, $\mu$ are such that $T \mu \geq a \mu^\alpha$,
\begin{multline*}
\int_{\mathcal{B}_{R_1+\varepsilon}} \vert y_1(T) \vert^2 \exp\left(- \mu \coth\left(2\mu T\right)((R_1+\varepsilon)^2 - \Vert x\Vert^2 )\right)
 \\ \leq \cc \int_0^T \frac{\sinh(2\mu t)^2}{\sinh(2\mu T)^2} \int_{\mathcal{B}_{R_1+\varepsilon}} \vert 
 2\nabla \chi_1 \cdot \nabla y - \Delta \chi_1 y \vert^2 \exp\left(- \mu \coth\left(2\mu t\right)((R_1+\varepsilon)^2 - \Vert x\Vert^2) \right).
\end{multline*}
As an immediate consequence, we obtain
$$
\int_{\mathcal{B}_{R_1+\varepsilon}} \vert y_1(T) \vert^2 
\leq  \frac{\cc}{\varepsilon^4} e^{\mu \coth(2\mu T) (R_1+\varepsilon)^2} \int_0^T \int_{\mathcal{A}_\varepsilon} \frac{\sinh(2\mu t)^2}{\sinh(2\mu T)^2} \left( \vert y \vert^2
 + \Vert \nabla y \Vert^2\right).
$$
By applying \cref{lemma_obsnablau->u} with $\mathcal{O}= \mathcal{A}_\varepsilon$, $\omega = \mathcal{A}$, and
$\vartheta = \frac{\sinh(2\mu t)^2}{\sinh(2\mu T)^2}$, and combining it with \cref{lemma_boundonhyperbfunc}, we obtain, for all $T$, $\mu$ such that $T\mu \geq a \mu^\alpha$,
\begin{multline} \label{eq_proof_obs_annulus_eq1}
\int_{\mathcal{B}_{R_1 +\frac{\varepsilon}{3}}} \vert y(T) \vert^2 \leq \int_{\mathcal{B}_{R_1+\varepsilon}} \vert y_1(T) \vert^2 \leq \frac{\mathcal c}{\varepsilon^6} \left(\mu + \frac{1}{T}\right) e^{\mu (R_1+\varepsilon)^2}e^{\frac{(R_1+\varepsilon)^2}{2T}} \int_0^T \int_{\mathcal{A}} \vert y \vert^2
\\ \leq \frac{\mathcal c}{\varepsilon^6} (1+\mu) e^{\mu (R_1+\varepsilon)^2} e^{\frac{\cc}{T}} \int_0^T \int_{\mathcal{A}} \vert y \vert^2.
\end{multline}

\textbf{Case $T\mu$ small --}
Now, we apply \cref{prop_generic_Carleman_V} with $Q = (0,T) \times \mathcal{B}_{R_1+\varepsilon}$, $V = \mu^2 \Vert x \Vert^2$, $\psi = \frac{1}{2} \Vert x - 2 R e \Vert^2$
with $e\in \mathbb{S}^{d-1}$, so that $\Vert \nabla \psi \Vert \geq R $ on $\mathcal{B}_R$, $\lambda = \lambda_0$
and $s = \left( 1+\frac{1}{T} + \Vert V \Vert_{W^{1,\infty}}^{\frac{1}{2}}\right) s_0$. This gives
\begin{multline} \label{eq_proof_obs_annulus_eq2}
\int_0^T \int_{\mathcal{B}_{R_1+\varepsilon}}  \vert y_1 \vert^2e^{-2\left(1+ \frac{1}{T} + \mu \right)s_0\phi}
\\ \leq \cc \int_0^T \int_{\mathcal{B}_{R_1+\varepsilon}}  \vert 2 \nabla \chi_1 \cdot \nabla y + \Delta \chi_1 \, y \vert^2 
e^{-2\left(1+ \frac{1}{T} + \mu \right)s_0\phi},
\end{multline}
where we recall that $\phi = \varpi \varphi$ with $\varpi$  defined by \eqref{eq_def_varpi} and $\varphi$ by \eqref{eq_def_varphi}.
Using that $\varpi \leq \frac{9}{8}$ on $\left[ \frac{T}{3}, \frac{2 T}{3} \right]$, and that $\varphi$ is bounded from below
by a positive constant depending only on $R$, we obtain
\begin{multline*}
\int_0^T \int_{\mathcal{B}_{R_1+\varepsilon}}  \vert y_1 \vert^2e^{-2\left(1+ \frac{1}{T} + \mu \right)s_0\phi}
\geq \int_{\frac{T}{3}}^{\frac{2T}{3}} \int_{\mathcal{B}_{R_1+\varepsilon}}  \vert y_1 \vert^2e^{-2\left(1+ \frac{1}{T} + \mu \right)s_0\phi}
\\ \geq e^{-\mathcal c \left(1+\frac{1}{T}+\mu\right) }\int_{\frac{T}{3}}^{\frac{2T}{3}} \int_{\mathcal{B}_{R_1+\varepsilon}}
\vert y_1\vert^2.
\end{multline*}
Equation \eqref{eq_proof_obs_annulus_eq2} then gives
\begin{multline} \label{eq_proof_obs_annulus_eq3}
e^{-\mathcal c \left(1+\frac{1}{T}+\mu\right) } \int_{\frac{T}{3}}^{\frac{2T}{3}} \int_{\mathcal{B}_{R_1+\varepsilon}}
\vert y_1\vert^2
\\ \leq \cc
 \int_0^T \int_{\mathcal{B}_{R_1+\varepsilon}}  \vert 2 \nabla \chi_1 \cdot \nabla y + \Delta \chi_1 \, y \vert^2 
e^{-2\left(1+ \frac{1}{T} + \mu \right)s_0\phi}.
\end{multline}
Using that 
that for all $\kappa \geq 1$,  all $c>0$ and $x \geq 0$, $x e^{-\kappa c x} \leq c^{-1} e^{-1}$, we obtain
\begin{multline} \label{eq_proof_obs_annulus_eq4}
 \int_0^T \int_{\mathcal{B}_{R_1+\varepsilon}}  \vert 2 \nabla \chi_1 \cdot \nabla y + \Delta \chi_1 \, y \vert^2 
e^{-2\left(1+ \frac{1}{T} + \mu \right)s_0\phi} \\ \leq  \cc\int_0^T \int_{\mathcal{B}_{R_1+\varepsilon}} \varpi^{-1} 
\vert 2 \nabla \chi_1 \cdot \nabla y + \Delta \chi_1 \, y \vert^2  .
\end{multline}
We choose $\theta : [0,1] \mapsto [0,\infty[$ such that $\theta(0) = 0$ and $\theta = 1$ on $\left[\frac{1}{3},1\right]$. Combining Lemma \ref{lemma_standard_energ_est_source}
with $\Theta(t) = \theta\left(\frac{t}{T}\right)$, \eqref{eq_proof_obs_annulus_eq3} and \eqref{eq_proof_obs_annulus_eq4} gives
$$
\int_{\mathcal{B}_{R_1+\varepsilon}} \vert y_1(T)\vert^2 
\leq \mathcal{c}\left(1+\frac{1}{T} \right) e^{\mathcal{c}\left(1+\frac{1}{T}+\mu\right)} 
\int_0^T \int_{\mathcal{B}_{R_1+\varepsilon}} \left( \Theta + \varpi^{-1}  \right)\vert 2 \nabla \chi_1 \cdot \nabla y + \Delta \chi_1 \, y \vert^2.
$$
The definition of $\chi_1$ and \cref{lemma_obsnablau->u} imply
\begin{multline*}
\int_{\mathcal{B}_{R_1+\frac{\varepsilon}{3}}} \vert y(T) \vert^2 \leq \int_{\mathcal{B}_{R_1+\varepsilon}} \vert y_1(T)\vert^2 \\ \leq \frac{\cc}{\varepsilon^4} \left(1 + \frac{1}{T}\right) e^{\mathcal{c}\left(1+\frac{1}{T}+\mu\right)} 
\int_0^T \int_{\mathcal{A}_\varepsilon} \left(\Theta + \varpi^{-1}\right) \left( \Vert \nabla y \Vert^2 + \vert y \vert^2 \right)\\
\leq \frac{\cc}{\varepsilon^6} \left(1 + \frac{1}{T^2}\right) e^{\mathcal{c}\left(1+\frac{1}{T}+\mu\right)}  \int_{\mathcal{A}} \vert y \vert^2.
\end{multline*}
As a consequence, if $T$, $\mu$ are such that $T \mu \leq a \mu^\alpha$, we obtain that
\begin{equation} \label{eq_proof_obs_annulus_eq5}
\int_{\mathcal{B}_{R_1+\frac{\varepsilon}{3}}} \vert y(T) \vert^2  \leq 
\frac{\cc}{\varepsilon^6} \left(1 + \frac{1}{T^2}\right)  e^{\mathcal{c}\left(1+\frac{1}{T}+\frac{a^\beta}{T^\beta}\right)}  \int_{\mathcal{A}} \vert y \vert^2\leq  \frac{\cc}{\varepsilon^6} e^{\cc \frac{a^\beta}{T^\beta}} \int_{\mathcal{A}} \vert y \vert^2.
\end{equation}

\textbf{Combining the estimates --} Combining \eqref{eq_proof_obs_annulus_eq1} and \eqref{eq_proof_obs_annulus_eq5} gives the desired inequality, with a constant $c_\beta = \cc a^\beta$ which tends to infinity as $\beta$
goes to 1 according to \cref{lemma_existence_a}.
\end{proof}

\begin{prop} \label{prop_observability_from_annulus_exterior_subcase}
There exists a positive constant $c$, depending only on $R_1$ and $R$, such that 
$$
\int_{\mathcal{B}_R \setminus \mathcal{B}_{R_1+\frac{\varepsilon}{6}}} \vert y(T) \vert^2
\leq \frac{c}{\varepsilon^6} e^{\frac{c}{T}} \int_0^T \int_{\mathcal{A}} \vert y \vert^2.
$$
\end{prop}

\begin{proof} In what follows, \( \cc \) denotes a constant that may depend only on $R_1$ and \( R \), and whose value may vary from line to line.

Let $\chi_2 : \R^d \mapsto \R$ be smooth, and such that $\chi_2 = 0$ in $\mathcal{B}_{R_1+\frac{\varepsilon}{12}}$, and
$\chi_2 = 1$ in $\R^d \setminus \mathcal{B}_{R_1+\frac{\varepsilon}{6}}$.
In the following, we denote 
$$
\mathcal{A}_\varepsilon = \left\lbrace x \in \R^d, \ R_1+\frac{\varepsilon}{12}< \Vert x \Vert < R_1 +\frac{\varepsilon}{6} \right\rbrace
\subset \mathcal A.
$$
The definition of $\chi_2$ implies
$$
\Vert \nabla \chi_2 \Vert_{L^\infty} \leq \frac{\cc}{\varepsilon} \indicatrice{\mathcal{A}_\varepsilon}, \quad \Vert \Delta \chi_2 \Vert_{L^\infty} \leq \frac{\cc}{\varepsilon^2} \indicatrice{\mathcal{A}_\varepsilon}.
$$
Set $y_2 = \chi_2 y$. Then  $y_2\in C^0([0,T],H^1_0(\mathcal{A}_\varepsilon))\cap L^2((0,T);H^2(\mathcal{A}_\varepsilon)) \cap H^1((0,T); L^2(\mathcal{A}_\varepsilon))$, verifies $\nabla y_2 \cdot x = 0$ on $\partial \mathcal{B}_{R_1+\frac{\varepsilon}{12}}$, and satisfies
$$
\partial_t y_2 - \Delta y_2 + \mu^2 \Vert x \Vert^2 y_2 = - 2 \nabla \chi_2 \cdot \nabla y - \Delta \chi_2\, y \text{ in } (0,T) \times
\Omega_\varepsilon,
$$
where we have set $\Omega_\varepsilon = \mathcal{B}_R \setminus \overline{\mathcal{B}_{R_1+\frac{\varepsilon}{12}}}$.  

We set $\psi(x) = - \frac{1}{2}\Vert x \Vert^2$, which satisfies $\Vert \nabla \psi \Vert \geq R_1$ in $\Omega_\varepsilon$, and for which
$$
\partial {\Omega_\varepsilon}_+ = \left\lbrace x \in \partial \Omega_\varepsilon, \ \nabla \psi \cdot n > 0 \right\rbrace = \partial \mathcal{B}_{R_1+\frac{\varepsilon}{12}}.
$$
Applying \cref{prop_Carleman_V_positive}, with $V = \Vert x \Vert^2$, $c_- = R_1$, $c_+ = R^2 + 2R$, $\lambda = \lambda_0$
and $s = \left(1+\frac{1}{T}\right) s_0$ gives
$$
\int_0^T \int_{\Omega_\varepsilon} \vert y_2 \vert^2 e^{-2s  \phi}\leq \cc \int_0^T\int_{\Omega_\varepsilon}
\vert 2 \nabla \chi_2 \cdot \nabla y + \Delta \chi_2 y \vert^2 e^{-2 s \phi},
$$
with $\phi = \varpi \varphi$, where $\varpi$ is defined by \eqref{eq_def_varpi} and $\varphi$ by \eqref{eq_def_varphi}.

{From this point onward, the proof follows closely the argument used for the previous result.} Using the properties of both $\varpi$ et $\varphi$, we first
obtain  
$$
\int_0^T \int_{\Omega_\varepsilon} \vert y_2 \vert^2 e^{-2s  \phi} \geq e^{-\cc\left(1+\frac{1}{T}\right)} \int_0^T \int_{\Omega_\varepsilon} \vert y_2 \vert^2,
$$
and
$$
\int_0^T\int_{\Omega_\varepsilon} \vert 2 \nabla \chi_2 \cdot \nabla y + \Delta \chi_2 y \vert^2 e^{-2 s \phi} 
\leq \cc \int_0^T \int_{\Omega_\varepsilon} \varpi^{-1} \vert 2 \nabla \chi_2 \cdot \nabla y + \Delta \chi_2 y \vert^2,
$$
which combined with the energy estimate given by \cref{lemma_standard_energ_est_source} gives
$$
\int_{\Omega_\varepsilon} \vert y_2(T) \vert^2 \leq \cc\left(1+\frac{1}{T}\right) e^{\cc \left(1+\frac{1}{T}\right)}
\int_0^T \int_{\Omega_\varepsilon} \left(\Theta + \varpi^{-1}\right) \vert 2 \nabla \chi_2 \cdot \nabla y + \Delta \chi_2 y \vert^2,
$$
for some smooth and nonnegative $\Theta$ such that $\Vert \Theta\Vert_{W^{1,\infty}} \leq \cc (1+T^{-1})$. Finally, the definition of $\chi_2$
and \cref{lemma_obsnablau->u} gives
$$
\int_{\mathcal{B}_R \setminus \mathcal{B}_{R_1+\frac{\varepsilon}{6}}} \vert y(T) \vert^2 \leq \int_{\Omega_\varepsilon} \vert y_2(T) \vert^2 \leq \frac{\cc}{\varepsilon^6} \left(1+\frac{1}{T}\right)^2 e^{\cc \left(1+\frac{1}{T}\right)}
\int_{\mathcal{A}} \vert y \vert^2.
$$
The result follows.
\end{proof}

\subsection{Observation from the boundary}

Finally, we obtain the equivalent of \cref{prop_observability_from_annulus} for boundary observation.

\begin{prop}
\label{prop_observability_from_boundary}
There exists a positive constant $\cc$, and for all $\beta> 1$,  a positive constant $c_\beta$ verifying $c_\beta \xrightarrow{\beta\rightarrow 1} \infty$, such that for all positive $T$ and $\mu$, for all $y$ in $C^0([0,T];H^1_0(\mathcal{B}_R))\cap L^2((0,T);H^2(\mathcal{B}_R)) \cap
H^1((0,T);L^2(\mathcal{B}_R))$
satisfying $\partial_t y - \Delta y + \mu^2 \Vert x \Vert^2 y = 0$ in $(0,T) \times \mathcal{B}_R$, the following estimate holds:
$$
\mu^2 \int_{\mathcal{B}_R} \vert y(T) \vert^2 + 
\int_{\mathcal{B}_R}    \Vert \nabla y (T) \Vert^2  \leq \cc (1+\mu)^3 e^{\mu R^2} e^{\frac{2 c_\beta}{T^\beta}} \int_0^T \int_{\partial \mathcal{B}_R}
\left\vert \frac{\partial y}{\partial n}\right\vert^2 .
$$
\end{prop}

\begin{proof}
The proof of this Proposition closely mirrors that of \cref{prop_observability_from_annulus_interior_subcase}.
For brevity, we omit some of the finer details. We set $\alpha = 1 - \frac{1}{\beta}$.
In what follows, \( \cc \) denotes a constant that may depend only on \( R \), and whose value may vary from line to line.

We fix $T>0$, $\mu>0$, and consider $y$ in $C^0([0,T];H^1_0(\mathcal{B}_R))\cap L^2((0,T);H^2(\mathcal{B}_R)) \cap
H^1((0,T);L^2(\mathcal{B}_R))$ such that $\partial_t - \Delta y + \mu^2 \Vert x \Vert^2 y = 0$ in $(0,T) \times \mathcal{B}_R$.

\textbf{Case $T\mu$ large --} Define $w = y \exp\left(-\frac{\mu}{2}\coth(2\mu t)(R^2 - \Vert x \Vert^2) \right)$. Then $w$ satisfies the assumptions of \cref{coro_BDE20}.
Combined with \cref{lemma_boundonhyperbfunc}, we obtain $a>0$, depending on $R$ and $\alpha$, such that, if $T$, $\mu$ verify $T \mu \geq a \mu^\alpha$, 
\begin{align*}
\int_{\mathcal{B}_R} \vert w(T) \vert^2 &\leq \cc \left(\mu + \frac{1}{T}\right) \int_0^T \int_{\partial \mathcal{B}_R}
\left\vert \frac{\partial w}{\partial n}\right\vert^2, \\
\int_{\mathcal{B}_R} \vert \nabla w (T) \vert^2 &\leq \cc \left(1+\mu + \frac{1}{T}\right)^3 \int_0^T \int_{\partial \mathcal{B}_R}
\left\vert \frac{\partial w}{\partial n}\right\vert^2.
\end{align*}
This implies
$$
\int_{\mathcal{B}_R} \vert y(T) \vert^2 e^{-\mu \coth(2\mu T) (R^2 - x^2)}\leq
\cc \left(\mu + \frac{1}{T}\right) \int_0^T \int_{\partial \mathcal{B}_R} \left\vert \frac{\partial y}{\partial n}\right\vert^2,
$$
and 
\begin{align*}
\int_{\mathcal{B}_R} \Vert \nabla y(T) \Vert^2 e^{-\mu \coth(2\mu T) (R^2 - x^2)}
&\leq\cc \left( \int_{\mathcal{B}_R} \Vert \nabla w(T) \Vert^2 + \left(\mu + \frac{1}{T}\right)^2 \int_{\mathcal{B}_R }\vert w(T) \vert^2 \right) \\
&\leq \cc \left( 1+ \mu + \frac{1}{T}\right)^3 \int_0^T \int_{\partial \mathcal{B}_R} \left\vert \frac{\partial y}{\partial n}\right\vert^2.
\end{align*}
\cref{lemma_boundonhyperbfunc} finally gives
\begin{equation} \label{eq_proof_obs_whole_bnd}
\mu^2  \int_{\mathcal{B}_R} \vert y(T) \vert^2 + \int_{\mathcal{B}_R} \Vert \nabla y(T) \Vert^2  \leq \cc \left(1+\mu + \frac{1}{T}\right)^3 e^{\mu R^2} e^{\frac{R^2}{2T}}
\int_0^T \int_{\partial \mathcal{B}_R} \left\vert \frac{\partial y}{\partial n}\right\vert^2.
\end{equation}

\textbf{Case $T\mu$ small --} We now fix $e \in \mathbb{S}^{d-1}$ and define the function $\psi : x\in \R^d \mapsto \frac{1}{2}\Vert x - 2 R e \Vert^2$, which verifies 
$\Vert \nabla \psi\Vert \geq R  $ on $\mathcal{B}_R$. Applying \cref{prop_generic_Carleman_V} with $V = \mu^2 \Vert x \Vert^2$, $\lambda = \lambda_0$
and $s = \left( 1+ \frac{1}{T} + \Vert V \Vert_{1,\infty}^{\frac{1}{2}}\right) s_0$, gives
$$
\int_0^T \int_{\mathcal{B}_R} \left(s^3\vert y\vert^2 + \Vert \nabla y \Vert^2\right) e^{-2 s \phi}
\leq \cc s \int_0^T \int_{\partial \mathcal{B}_R} \varpi \left\vert \frac{\partial y}{\partial n}\right\vert^2 e^{-2s \phi},
$$
where $\phi = \varpi \varphi$ with $\varpi$ defined by \eqref{eq_def_varpi} and $\varphi$ by \eqref{eq_def_varphi}. On one side, we have
\begin{multline*}
\int_0^T \int_{\mathcal{B}_R} \left(s^3\vert y\vert^2 + \Vert \nabla y \Vert^2\right) e^{-2 s \phi}
\\ \geq\cc  \int_{\frac{T}{3}}^{\frac{2T}{3}} \int_{\mathcal{B}_R} \left((1+\mu^3)\vert y\vert^2 + \Vert \nabla y \Vert^2\right) e^{-2 s \phi}
\\ \ge\cc e^{-\cc \left(1+\frac{1}{T} +\mu\right)}\int_{\frac{T}{3}}^{\frac{2T}{3}} \int_{\mathcal{B}_R} \left(\mu^2 \vert y\vert^2 + \Vert \nabla y \Vert^2\right).
\end{multline*}
On the other side, as for all $c>0$, all $s\geq 1$ and all $x\geq 0$, we have $s\, x e^{-scx} \leq (c e)^{-1}$, we obtain
$$
s \int_0^T \int_{\partial \mathcal{B}_R} \varpi \left\vert \frac{\partial y}{\partial n}\right\vert^2 e^{-2s \phi} \leq
\cc \int_0^T \int_{\partial \mathcal{B}_R}  \left\vert \frac{\partial y}{\partial n}\right\vert^2.
$$
We deduce that, if $T$, $\mu$ verifies $T\mu \leq a \mu^\alpha$, 
$$
\int_{\frac{T}{3}}^{\frac{2T}{3}} \int_{\mathcal{B}_R} \left(\mu^2 \vert y\vert^2 + \Vert \nabla y \Vert^2\right)
\leq \cc e^{\cc \left(1+\frac{1}{T} +\frac{a^\beta}{T^\beta}\right)} \int_0^T \int_{\partial \mathcal{B}_R} \left\vert \frac{\partial y}{\partial n}\right\vert^2.
$$
Then \cref{lemma_standard_energ_est_source} and  \cref{lemma_estimate_nabla_T} gives
$$
\mu^2 \int_{\mathcal{B}_R} \vert y(T) \vert^2 + \int_{\mathcal{B}_R} \Vert \nabla y(T) \Vert^2 \leq \cc \left(\frac{1}{T} +\mu^2\right) e^{\cc \left(1+\frac{1}{T} +\frac{a^\beta}{T^\beta}\right)} \int_0^T \int_{\partial \mathcal{B}_R}  \left\vert \frac{\partial y}{\partial n}\right\vert^2,
$$
which combined with \eqref{eq_proof_obs_whole_bnd} and the behaviour of $a$ as $\beta$ goes to 1 given by \cref{lemma_existence_a}, gives the result.
\end{proof}

\section{A precised Lebeau-Robbiano strategy for observability, and application}
\label{sec_restriction_arguments}

\subsection{Lebeau-Robbiano strategy in semigroup framework}
\label{preuvemiller}
{In this section, we adapt the strategy of \cite{Miller2010} to our main objective. To stay close to the original framework of \cite{Miller2010}, we adopt a semigroup point of view.}
Let $\mc{E}$, $\mc{U}$ be two Hilbert spaces, and $\mc{F}$ be a third Hilbert space continuously and densely embedded in $\mc{E}$.
\begin{rem}
In \cite{Miller2010}, there is no subspace $\mc{F}$, everything is expressed in the norm $\norm{\cdot}{\mc{E}}$. Working with the subspace $\mc{F}$ does not change anything in the proof (except for the norms), but it is the good framework when we are looking at boundary observability problems.
\end{rem}
We consider an abstract evolution equation 
	\begin{equation} 
	\label{abstractequation} 
	\partial_t y + Ay = 0, \quad y(0) = y_0 \in \mc{F}, \quad t \ge 0,
	\end{equation} 
where $-A~:~\mc{D}(A) \subset \mc{E} \to \mc{E}$ is a generator of a strongly continuous semigroup on  $\mc{E}$. We also consider $B \in \mc{L}(\mc{D}(A),\mc{U})$ admissible with respect to $\mc{F}$ for this semigroup, i.e. $B$ is continuous from $\mc{D}(A)$ with the graph norm to $\mc{U}$ and satisfies 
	\begin{equation} 
	\label{adm}
	\int_0^T \norm{Be^{-tA}y}{\mc{U}}^2dt \le Adm_T\norm{y}{\mc{F}}^2, \quad y \in \mc{D}(A), \quad T>0.
	\end{equation}
Note that the equation under study \eqref{harmonicheat} fits in this general framework with \begin{equation} \label{notationharmonicheat} \mc{E} = L^2(\mathcal{B}_R), \quad {A =\, }
\grush = -\Delta +\mu^2 \|x\|^2 ,\quad \mathcal{D}(\grush) = H^2(\mathcal{B}_R)\cap H^1_0(\mathcal{B}_R),\end{equation}
and, for the internal observability, $\mc F = \mc E$ and $B = \indicatrice{\omega}$. \newline

In this abstract framework, observability reads:
\begin{ddef}[Observability]
Let $T>0$. We say that $(A,B) $ is observable at time $T$, with cost $ \sqrt{\cost(A,B,T)} \ge 0$, if the following inequality holds:
\begin{equation}\label{generalobs}
\norm{e^{-TA}y}{\mc{F}}^2 \le \cost(A,B,T) \int_0^T \norm{Be^{-tA}y}{\mc{U}}^2 dt, \quad \forall y \in \mc{D}(A).
\end{equation}
\end{ddef}

{Our goal is to deduce an observability inequality for the observation operator $B$, from an observability inequality valid for another related admissible observation operator $B_0 \in \mc{L}(\mc{D}(A),\mc{U})$ -- for the internal observability, $B_0$ is simply $\indicatrice{\mc A}$, where $\mc{A}$ is the whole annulus as defined in \eqref{annulus}.}\medskip

{We now} state the assumptions needed for the abstract restriction result. {First}, we assume that there exists {$(\mc{F}_\sigma)_{\sigma >0}$},  a nondecreasing family of semigroup invariant closed subspaces of $\mc{F}$ {-- that is satisfying $e^{-tA}\mc{F}_\sigma \subset \mc{F}_\sigma \subset \mc{F}_{\sigma'} \subset \mc{F}$ if $0\!<\!\sigma\!<\!\sigma'$ --} such that
	\begin{equation}
	\label{dissip}
	\norm{e^{-tA}y}{\mc{F}}^2 \le C_{dissip} e^{-2c\sigma t}\norm{y}				{\mc{F}}^2, \quad \sigma >0, \quad y \in \mc{F}_\sigma^\perp, \quad t \in 			(0,T_{max}],
	\end{equation} 
for some $C_{dissip}$, $c$, $T_{max} >0$, where the orthogonal {complement} is taken with respect to the scalar product on $\mc{F}$. \newline
We also assume that {$B_0$ and $B$ satisfy the relative bound}
	\begin{equation}
	\label{spec}
	\norm{B_0 y}{\mc{U}}^2 \le C_{rel} e^{2a\sigma^\delta}\norm{B y}{\mc{U}}^2, \quad \sigma >0, \quad y \in \mc{F}_\sigma,
	\end{equation} 
for some $C_{rel}$, $a > 0$, $\delta \in (0,1)$. \newline
Finally we assume that the operator $B_0$ satisfies the observability estimate
	\begin{equation}
	\label{obs}
	\norm{e^{-TA}y}{\mc{F}}^2 \le C_{obs}e^\frac{2b}{T^\beta} \int_0^T \norm{B_0 e^{-tA}y}{\mc{U}}^2 dt,  \quad y \in \mc{D}(A), \quad T \in (0,T_{max}],
	\end{equation}
for some $C_{obs}$, $b$ and $\beta>0$, independant of $\sigma$ and $T$.\newline

{
\begin{rem}
Note that the constants $\delta$ in \eqref{spec} or $\beta$ in \eqref{obs} may be increased to satisfy
 \begin{equation} \label{rkbetadelta}\beta = \frac{\delta}{1-\delta}.\end{equation}
For \eqref{spec} to remain true if $\delta$ is increased, we have to replace $C_{rel}$ by $C_{rel}e^{2a}$, whereas for \eqref{obs} to remain true if $\beta$ is increased, we have to replace $C_{obs}$ by $C_{obs}e^{2b}$.

In our application, we have $\delta = \frac12$ and $\beta >1$. Therefore $\delta$ is increased and $C_{obs}$ remains unchanged.
\end{rem}
}

\begin{thm}
\label{LR}
Under the assumptions \eqref{dissip}, \eqref{spec} and \eqref{obs}, there exists $s>0$ and $T'>0$ such that for any $T \in (0,T']$, $(A,B)$ is observable with
$$\cost(A,B,T) \le 4 C_{rel}C_{obs}e^\frac{2}{(sT)^\beta}.$$
Moreover there exists $0<\kappa_0<1$, depending on $a$, $b$, $c$, and $\beta$, such that for any $\kappa \in (0,\kappa_0)$, we can write explicitely $s$ and $T'$ as follows:
	$$s = s_2\left(1-\left(\frac{s_1}{s_2}\right)^\frac{\beta}{\beta +1}\right), \quad T' = \left\{ \begin{aligned} \left(\frac{C_1}{\ln(C_2)}\right)^\frac{1}{\beta} &\text{ if}& C_2 > 1 \\ T_{max} &\text{ if}& C_2 \le 1 \end{aligned} \right.$$
where 
	$$s_2 = \frac{\kappa}{(a+b)^\frac{1}{\beta}}, \quad s_1 = \kappa^\frac{\beta+1}{\beta} \left( \frac{2}{c(1-\kappa)}\right)^\frac1\beta,~ C_1 = \frac{2}{s_1^\beta}\left(1-\left(\frac{s_1}{s_2}\right)^\frac{\beta}{\beta +1}\right),$$
	$$ C_2 = 4C_{dissip}Adm_{T_{max}}C_{rel}C_{obs}+ 2C_{dissip}\exp\left(-\frac{2}{(s_2T_{max})^\beta}\right).$$
\end{thm}
{Though similar to \cite{Miller2010}, we present the full proof to explicitly derive the required expressions.}
We start by giving some preliminary results. The first one is taken from \cite[Lemma 2.1]{Miller2010} which gives a bound on the observability constant assuming an approximate observability estimate. 
\begin{lem}
\label{lemmeLR}
If the approximate observability estimate 
$$ f(t)\norm{e^{-tA}y}{\mc{F}}^2 \le  f(qt)\norm{y}{\mc{F}}^2 + \int_0^t \norm{B e^{-\tau A }y}{\mc{U}}^2 d\tau , \quad y \in \mc{D}(A), \quad t \in(0,T'],$$ holds with $f(t) \to 0$ as $t \to 0^+$ , $q \in (0,1)$ and $T' > 0$, then the following observability inequality holds
$$\norm{e^{-TA}y}{\mc{F}}^2 \le \frac{1}{f((1-q)T)} \int_0^T \norm{B e^{-tA}y}{\mc{U}}^2 dt, \quad y \in \mc{D}(A) \quad T \in (0,T']. $$
\end{lem}

Using this lemma, we can prove the following, which is an adaptation of \cite[Lemma 2.3]{Miller2010}.
 \begin{lem}
 \label{lemmeapprox}
 If the approximate observability estimate 
 $$ f(T)\norm{e^{-TA}y}{\mc{F}}^2  \le g(T)\norm{y}{\mc{F}}^2 +\int_0^T \norm{B e^{-t A }y}{\mc{U}}^2 dt, \quad y \in \mc{D}(A), \quad T \in (0,T_{max}],$$ holds with $f(T) = f_0\exp(-2/(s_2T)^\beta)$ and $g(T) = g_0\exp(-2/(s_1T)^\beta)$ where \linebreak $f_0$, $g_0 >0$, $s_2>s_1>0$, then setting
 $$s = s_2\left(1-\left(\frac{s_1}{s_2}\right)^\frac{\beta}{\beta +1}\right) \quad \text{and} \quad T' = \begin{cases} T_{max} &\text{ if } g_0 \le f_0, \\
 \left(\frac{2\left(1-\left(\frac{s_1}{s_2}\right)^\frac{\beta}{\beta +1}\right)}{s_1^\beta\log\left(\frac{g_0}{f_0}\right)}\right)^\frac{1}{\beta} &\text{ if } g_0 > f_0,
 \end{cases}$$ the following observability inequality holds
 $$\norm{e^{-TA}y}{\mc{F}}^2 \le \frac{1}{f_0}\exp\left(\frac{2}{(sT)^\beta}\right)\int_0^T \norm{Be^{-tA}y}{\mc{U}}^2 dt, \quad y \in \mc{D}(A), \quad T \in (0,T'].$$
 \end{lem}
 
 \begin{proof}
 We check that for $q = \left(\frac{s_1}{s_2}\right)^\frac{\beta}{\beta +1}$ and $T \in (0,T']$ with $T'$ as defined above, we have that $g(T) \le f(qT)$. Indeed if $g_0 \le f_0$, since $s_1 \le qs_2$ we have $g(T) \le f(qT)$. And if $g_0 > f_0$, we have 
 $$ T^\beta \le  \frac{2\left(1-\left(\frac{s_1}{s_2}\right)^\frac{\beta}{\beta +1}\right)}{s_1^\beta\log\left(\frac{g_0}{f_0}\right)},$$
 which implies first
 $$\frac{-2}{s_1^\beta T^\beta} \le \frac{-\ln\left(\frac{g_0}{f_0}\right)}{1-q},$$
 and then
 $$\frac{-2}{s_1^\beta T^\beta}+\frac{2q}{s_1^\beta T^\beta} \le \ln\left(\frac{f_0}{g_0}\right). $$
 This in turn implies 
 $$g_0\exp\left(\frac{-2}{s_1^\beta T^\beta}\right) \le f_0 \exp\left(\frac{-2q}{s_1^\beta T^\beta}\right), $$
which finally gives $g(T) \le f(qT), \text{ since } (q/s_1^\beta) = (1/(qs_2)^\beta).$
We conclude by applying \cref{lemmeLR}.
 \end{proof}
 {We may now proceed with the proof of \cref{LR}.}
We follow step by step the proof given in \cite[section 2.4]{Miller2010}, tracking the different parameters.
 \begin{proof}[of \cref{LR}]
Using the assumptions \eqref{adm}, \eqref{dissip}, \eqref{spec} and \eqref{obs} we prove the approximate observability estimate of \cref{lemmeapprox} and identify $f_0, g_0, s_1$ and $s_2$.
Let $\kappa \in (0,1)$, $T \in (0,T_{max}]$ and $\tau = \kappa T$.  We define the spectral threshold $\sigma$ {by the relation} $\sigma^\delta = \frac{1}{\tau^\beta}$, where $\delta$ and $\beta$ come from the assumptions \eqref{spec} and \eqref{obs}, and satisfy \cref{rkbetadelta}. For $y \in \mc{D}(A)$, we denote $y_\sigma$ the orthogonal projection of $y$ on $\mc{F}_\sigma$ and $y_\sigma^\perp = y - y_\sigma $. \newline
Combining \eqref{spec} and \eqref{obs} {applied} at time $\tau$ {gives}
$$\norm{e^{-\tau A}\tilde{y}_\sigma}{\mc{F}}^2 \le C_{obs}C_{rel}e^{\frac{2b}{\tau^\beta}+2a\sigma^\delta} \int_0^\tau \norm{Be^{-tA}\tilde{y}_\sigma}{\mc{U}}^2 dt,  \quad \tilde{y}_\sigma \in \mc{F}_\sigma. $$
{Since $\tau = \kappa T$}, applying this inequality to $\tilde{y}_\sigma = e^{-(1-\kappa)TA}y_\sigma$ leads to
$$ \norm{e^{-TA}y_\sigma}{\mc{F}}^2 \le \frac{1}{4f(T)}\int_{(1-\kappa)T}^T\norm{Be^{-tA}y_\sigma}{\mc{U}}^2dt,$$ with 
$$f(T) = \frac{1}{4C_{rel}C_{obs}} \exp\left(-\frac{2(a+b)}{(\kappa T)^\beta}\right).$$
Then we use the inequality $\norm{x+y}{}^2 \le 2(\norm{x}{}^2 + \norm{y}{}^2)$, the admissibility \eqref{adm}, and the fact that $y = y_\sigma + y_\sigma^\perp$, and we get
\begin{align*}f(T)\!\norm{e^{-TA}y}{\mc{F}}^2 &\le 2f(T)\norm{e^{-TA}y_\sigma}{\mc{F}}^2 + 2f(T)\norm{e^{-TA}y_\sigma^\perp}{\mc{F}}^2 \\
&\le \frac12 \int_{(1-\kappa)T}^T\!\!\!\!\!\norm{Be^{-tA}y_\sigma}{\mc{U}}^2dt + 2f(T)\norm{e^{-TA}y_\sigma^\perp}{\mc{F}}^2 \\
&\le \int_{(1-\kappa)T}^T\!\!\!\!\!\!\!\!\!\norm{Be^{-tA}y}{\mc{U}}^2dt + \int_{(1-\kappa)T}^T\!\!\!\!\!\norm{Be^{-tA}y_\sigma^\perp}{\mc{U}}^2dt + 2f(T)\norm{e^{-TA}y_\sigma^\perp}{\mc{F}}^2 \\
&\le \int_{(1-\kappa)T}^T\!\!\!\!\!\!\!\!\!\!\!\norm{Be^{-tA}y}{\mc{U}}^2dt + Adm_{\kappa T}\! \norm{e^{-(1-\kappa)TA}y_\sigma^\perp}{\mc{F}}^2 +2f(T)\!\norm{e^{-TA}y_\sigma^\perp}{\mc{F}}^2.\end{align*}
We apply the dissipation property \eqref{dissip} in the last two terms to obtain
\begin{multline*}f(T)\norm{e^{-TA}y}{\mc{F}}^2 \le  \int_0^T\norm{Be^{-tA}y}{\mc{U}}^2dt \\
+ C_{dissip}\left(Adm_{\kappa T}e^{-(1-\kappa)2c\sigma T} + 2f(T)e^{-2c \sigma T}\right)\norm{y_\sigma^\perp}{\mc{F}}^2.\end{multline*}
Finally, we use that $\norm{y_\sigma^\perp}{\mc{F}} \le \norm{y}{\mc{F}}$, $Adm_{\kappa T} \le Adm_{T_{max}}$ (since $Adm_T$ is a nondecreasing function of $T$), and $f(T) \le f(T_{max})$ {to conclude that}
$$f(T)\!\norm{e^{-TA}y}{\mc{F}}^2 \le  \int_0^T\!\!\!\!\norm{Be^{-tA}y}{\mc{U}}^2dt   + C_{dissip}\!\left(Adm_{T_{max}} + 2f(T_{max})\right)\!e^{-(1-\kappa)2c\sigma T}\norm{y}{\mc{F}}^2 .$$
This is the expected approximate observability estimate with 
$$f_0 = \frac{1}{4C_{rel}C_{obs}}, \quad g_0= C_{dissip}(Adm_{T_{max}} + 2f(T_{max})),$$
$$s_2 = \frac{\kappa}{(a+b)^\frac{1}{\beta}}, \quad s_1 = \kappa^\frac{\beta+1}{\beta} \left( \frac{1}{c(1-\kappa)}\right)^\frac1\beta.$$
There exists $0<\kappa_0<1$ depending on $a$, $b$, $c$, and $\beta$ such that for any $\kappa \in (0,\kappa_0)$ we have $s_2>s_1>0$. In this case we can apply \cref{lemmeapprox}, which ends the proof of the theorem. 
\end{proof}

\subsection{Proof of \cref{maininterne} -- internal observation}
\subsubsection{{Observability of the harmonic heat equation from a subset of annulus}}
\label{applicationrestriction}
For any open subset $\Gamma \subset \sphere$, we define the subset of the annulus
	\begin{equation}
	\label{defdesectionanneau}
	\omega_\Gamma := \left\{ x \in \R^d;~x= r\theta, ~R_1 <r< R_2,  ~\theta \in \Gamma \right\}.
	\end{equation}
	The goal of this section is {to give a precise estimate of the cost of observability of the harmonic heat equation from $\omega_\Gamma$. Our result reads:}
\begin{thm} 
\label{obsharmonicheatsectionannulus}
Let $B_{\omega_\Gamma} = \indicatrice{\omega_\Gamma}$, for any $T_{max}>0$, $\beta >1$ and $0<\gamma<1/2$, there exists $C,s >0$ such that for any $\epsilon>0$, $\mu>0$ and $0<T<T_{max}$, $(\grush,B_{\omega_\Gamma})$ is observable at time $T$ with
\begin{equation}
\label{observabilityharmonicheatsectionannulus}
\cost(\grush,B_{\omega_\Gamma},T) \le \frac{C}{\epsilon^6} (1+\mu) e^{\mu (1+\gamma )(R_1+\epsilon)^2}e^\frac{2}{(sT)^\beta}.
\end{equation}
Moreover, \begin{equation} \label{comportementdes} s \underset{\gamma \to 0}{\sim} \frac{\gamma}{C_1(\beta)} \quad \text{with} \quad C_1(\beta) \underset{\beta \to 1}{\longrightarrow} \infty \end{equation}
\end{thm}

To prove this result, we apply the {abstract restriction result given in \cref{preuvemiller}}, to the observability on the whole annulus given by \cref{prop_observability_from_annulus}, which rewrites in the semigroup framework \eqref{notationharmonicheat} as follows.
\begin{cor} 
\label{obsharmonicheatwholeannulus}
Let $B_{\mathcal A} = \indicatrice{\mathcal A}$. There exists $C>0$ and, for any $\beta >1$, a positive constant $c_\beta$ verifying $c_\beta \xrightarrow{\beta \rightarrow 1} \infty$, such that for any $\epsilon\in (0,1)$, $\mu>0$ and $T>0$, $(\grush,B_{\mc A})$ is observable at time $T$ with
the following cost estimate:
\begin{equation}
\label{observabilityharmonicheatwholeannulus}
\cost(\grush,B_{\mc A},T) \le \frac{C}{\epsilon^6}(1+\mu) e^{\mu (R_1+\epsilon)^2}e^\frac{2 c_\beta}{T^\beta}.
\end{equation}
\end{cor}

\begin{proof}[of \cref{obsharmonicheatsectionannulus}] We want to apply \cref{LR}, to deduce the observability for the operator $B_{\omega_\Gamma}$ from the observability for $B_\mc{A}$ which is given by \cref{obsharmonicheatwholeannulus}. We start by checking the assumptions, with $\mc{E} = \mc{F} = L^2(\mc{B}_R)$ endowed with the usual scalar product for the state space, and $U =  L^2(\mc{B}_R)$ for the observation space. We define for any $\sigma > 0$,
\begin{equation} \label{esigma} \mc{F}_\sigma = \left\{ \sum_{\underset{k \in \llbracket1, \beta(m,d)\rrbracket}{m<\sqrt{\sigma}}}\scalar{y}{Y_{mk}}{L^2(\sphere)}Y_{mk} \quad \Big{|} \quad y \in L^2(\mc{B}_R) \right\} \subset \mc{F},\end{equation}
\newline
\noindent $\bullet$ Admissibility: for any $y \in \mathcal{D}(\grush)$ and $T>0$ 
\begin{equation}
\label{adminterne}
\int_0^T\int_{\omega_\Gamma}|e^{-t\grush}y|^2 \le T \int_{\mathcal{B}_R}|y|^2
\end{equation}
which gives the admissibility inequality with $Adm_T = T$. \newline \medskip 
\noindent $\bullet$ Dissipation: it involves the spaces $\mc{F}_\sigma$ defined in \eqref{esigma}.
We prove that the eigenvalues of $G_\mu|_{\mc{F}_\sigma^\perp}$ are larger than $\frac{\sigma}{R^2}$ and then apply \cref{dissipationgrushin}. \newline
Let $y \in \mc{F}_\sigma^\perp $, $y \ne 0$ and $\lambda$ such that $\grush y = \lambda y$. Then, since by \cref{lemme H10M} $$y = \sum_{\underset{k \in \llbracket1, \beta(m,d)\rrbracket}{m \ge \sqrt{\sigma}}}y_{mk}Y_{mk} \quad \text{and} \quad y \ne 0, $$ there exists $m,k$ such that $y_{mk} \ne 0$. Computing $\grush y $ we obtain
$$-\frac{1}{r^{d-1}}(\partial_r(r^{d-1}\partial_r y_{mk})) + m(m+d-2)\frac{y_{mk}}{r^2} + \mu^2r^2y_{mk}= \lambda y_{mk}.$$
We multiply this equality by $r^{d-1}y_{mk}$, integrate over $(0,R)$ and do an integration by part to obtain
$$ \int_0^R r^{d-1}(y_{mk}'(r))^2 dr = \int_0^R (\lambda r^{d-1} - m(m+d-2)r^{d-3}- \mu^2 r^{d+1})y_{mk}(r)^2 dr.$$
As $y_{mk} \ne 0$ and cannot be constant since $y_{mk}(R) =0$, the left handside is positive and it implies that there exists $r_0 \in (0,R)$ such that 
$$ \lambda \ge \frac{m(m+d-2)}{r_0^2} + \mu^2r_0^2.$$
Since $m \ge \sqrt{\sigma}$ we finally have 
$$ \lambda \ge \frac{\sigma}{R^2}.$$
Let  $ y_0 \in \mc{F}_\sigma^\perp$, $t>0$, we denote $y_p(t) = e^{-t\grush}y_0$ the solution of \eqref{harmonicheat}. 
Since the eigenvalues of $\grush$ are greater than $\sigma/R^2$ we have by \cref{dissipationgrushin}
\begin{equation}
\label{dissipinterne}
\norm{y_\mu(t)}{L^2(\mathcal{B}_R)}^2 \le \exp\left(-2\frac{\sigma}{R^2}t\right)\norm{y_0}{L^2(\mathcal{B}_R)}^2
\end{equation}
Therefore we have the dissipation estimate with $C_{dissip} = 1$ and $c = 1/R^2$. \newline \medskip 
\noindent $\bullet$ Relation between $B_\mathcal{A}$ and $B_{\omega_\Gamma}$ -- Let $y \in \mc{F}_\sigma$. We have to
{bound $\norm{y}{L^2(\mathcal{A})}^2$ by $ \norm{y}{L^2(\omega_\Gamma)}^2$.}
Let us momentarily distinguish $y$ the function expressed in cartesian coordinates, and $\hat{y}$ the function expressed in polar coordinates $(r,\theta) \in (0,R) \times \sphere$. We have that 
$$\hat{y}(r,\theta) = \sum_{\underset{k \in \llbracket1, \beta(m,d)\rrbracket}{m < \sqrt{\sigma}}}\hat{y}_{mk}(r) Y_{mk}(\theta).$$ 
Notice that for any $r \in (0,R)$, {$\hat{y}(r,.)$} is a finite sum of eigenfunctions of the Laplace-Beltrami operator on $\sphere$. 
{Consequently, we may} apply the spectral inequality \cite[Theorem 14.6]{Jerison}, which gives a constant $C>0$ depending on $\Gamma$ such that
$$ \norm{ \sum_{\underset{k \in \llbracket1, \beta(m,d)\rrbracket}{m < \sqrt{\sigma}}}\hat{y}_{mk}(r) Y_{mk}}{L^2(\sphere)}^2 \le Ce^{C\sqrt{\sqrt{\sigma}(\sqrt{\sigma}+d-2)}}\norm{ \sum_{\underset{k \in \llbracket1, \beta(m,d)\rrbracket}{m < \sqrt{\sigma}}}\hat{y}_{mk}(r) Y_{mk}}{L^2\left(\Gamma\right)}^2.$$
{Here, we use that} the eigenvalue associated to $Y_{mk}$ is $m(m+d-2)$ (see \eqref{laplacienvp}).
{In conclusion}, there exists $C_{rel}>0$ and $a>0$, depending only on $\Gamma$ and $d$, such that for any $r \in (0,R)$,
$$\int_{\sphere} |\hat{y}(r,\theta)|^2d\theta \le C_{rel}e^{2a\sqrt{\sigma}}\int_\Gamma |\hat{y}(r,\theta)|^2d\theta .$$
We multiply by $r^{d-1}$ and integrate over $(R_1,R_2)$ to obtain
$$\int_{R_1}^{R_2}\int_{\sphere} r^{d-1}|\hat{y}(r,\theta)|^2d\theta dr \le C_{rel}e^{2a\sqrt{\sigma}}\int_{R_1}^{R_2}\int_\Gamma r^{d-1}|\hat{y}(r,\theta)|^2d\theta dr$$
Since $\hat{y}(r,\theta) = y(r\theta)$, and using the polar change of variable, we finally obtain
\begin{equation}
\label{relinterne}
\norm{y}{L^2(\mathcal{A})}^2 \le C_{rel}e^{2a\sqrt{\sigma}}\norm{y}{L^2\left(\omega_\Gamma\right)}^2,
\end{equation}
which gives the expected inequality. \newline \medskip 
\noindent $\bullet$ Observability for $B_{\mathcal{A}}$: it is given by \cref{obsharmonicheatwholeannulus}, with for any $0<\epsilon<1$, $C_{obs} = \frac{C}{\epsilon^6}(1+ \mu) e^{\mu (R_1+\epsilon)^2}$.
\bigskip \newline
Now we can apply \cref{LR}: there exists $s>0$ (which does not depend on $\mu$) and $T'>0$ (which depends on $\mu$ as we will see in the expression given below) such that for any $T \in (0,T']$, $(\grush,B_{\omega_\Gamma})$ is observable with
\begin{equation} \label{applithmharmonicheat} \cost(\grush,B_{\omega_\Gamma},T) \le  4C_{rel}\frac{C}{\epsilon^6}(1+ \mu) e^{\mu (R_1+\epsilon)^2}e^\frac{2}{(sT)^\beta}.\end{equation}
If $T \le T'$, then the result is proven. The problem is that $T'$ depends on $\mu$ and goes to $0$ as $\mu$ goes to infinity. Indeed using the notations of \cref{LR} we have
$$T' = \left\{ \begin{array}{cc} \left(\dfrac{C_1}{\ln(C_2)}\right)^\frac{1}{\beta} &\text{ if } C_2 > 1 \\ T_{max} &\text{ if } C_2 \le 1 \end{array} \right.$$
and $$C_2 = 4C_{rel}T_{max}\frac{C}{\epsilon^6}(1+ \mu) e^{\mu (R_1+\epsilon)^2} + 2\exp(-2/s_2^\beta).$$
One can see that $C_2 \to +\infty $ as $\mu \to +\infty$. 
Therefore, we need to be more precise to tackle the case $T>T'$. The strategy is to use the parameter $\kappa$ given in \cref{LR} to get the expected estimate.
Let $0<\gamma <1/2$ and $0<\tilde{\gamma} <1$ such that $\gamma = \tilde{\gamma}^2+ \tilde{\gamma}$ and let $x\in (0,1)$ such that $$\dfrac{x^\beta}{\left(1-x^\frac{\beta}{\beta +1}  \right)^{\beta +1} } = \tilde{\gamma}.$$ 
{Note that, when $\tilde \gamma$ is close to zero, $x \sim \tilde \gamma^{\frac{1}{\beta}}$.}
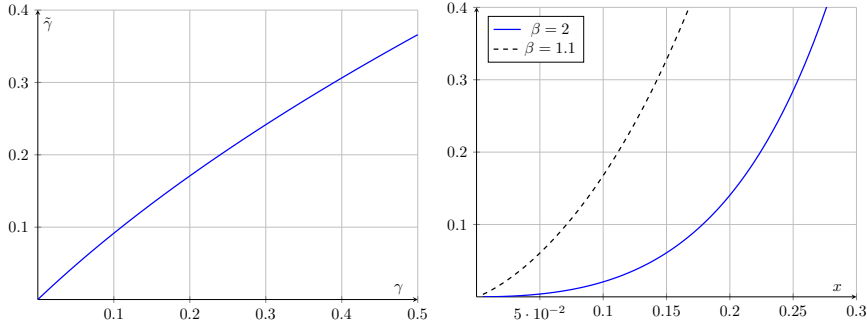
\begin{figure}[h!]
\begin{center}
\begin{tikzpicture}[scale=0.6]
    \begin{axis}[
        axis lines = middle,
        xlabel = $\gamma$,
        xlabel style={at={(rel axis cs:0.95,0)}, anchor=south},
        ylabel = {$\tilde{\gamma}$},
        ylabel style={at={(rel axis cs:0,0.95)}, anchor=west}, 
        xmin = 0,
        xmax = 0.5,
        ymin = 0,
        ymax = 0.4,
        domain = 0:0.5,
        samples = 200,
        smooth,
        legend pos = north west,
        grid = major,
        width=10cm,
        height=8cm,
    ]
    \addplot[blue, thick] {sqrt(x + 1/4) - 1/2};
    \end{axis}
\end{tikzpicture}
\begin{tikzpicture}[scale=0.6]
    \begin{axis}[
        axis lines = middle,
        xlabel = $x$,
        xlabel style={at={(rel axis cs:0.95,0)}, anchor=south}, 
        xmin = 0,
        xmax = 0.3,
        ymin = 0,
        ymax = 0.4,
        domain = 0.005:0.3, 
        samples = 300,
        smooth,
        legend pos = north west,
        grid = major,
        width=10cm,
        height=8cm,
    ]
    \addplot[blue, thick] {x^2 / (1 - x^(2/(2+1)))^(2+1)};
    \addplot[black, thick, dashed] {x^(1.1) / (1 - x^(1.1/(1.1+1)))^(1.1+1)};
    \addlegendentry{$\beta = 2$};
    \addlegendentry{$\beta = 1.1$};
    \end{axis}
\end{tikzpicture}
\caption{Left: $\tilde \gamma$ as a function of $\gamma$. Right: the function $x \mapsto x^\beta \left(1 - x^{\frac{\beta}{\beta+1}} \right)^{-(\beta+1)}$, for two different values of the parameter $\beta$. Clearly, $\gamma$ uniquely determines $\tilde \gamma$, which then uniquely determines $x$.}
\end{center}
\end{figure}

We choose $\kappa$ such that $$\frac{s_1}{s_2} = \left(\frac{\kappa}{1-\kappa}\right)^\frac1\beta \left(\frac{a+b}{c}\right)^\frac1\beta =x.$$
Then we have $$s = s_2(1-x^\frac{\beta}{\beta+1}), \quad C_1 = \frac{2}{s_1^\beta}(1-x^\frac{\beta}{\beta+1}).$$
Recalling that $s_2 = \kappa/(a+c_\beta)^\frac1\beta$, we obtain that 
$$s \underset{\gamma \to 0}{\sim} \frac{\gamma c}{(a+c_\beta)^{1+\frac1\beta}}$$
which gives \cref{comportementdes} since $c_\beta \underset{\beta \to 1}{\longrightarrow} \infty$. \newline
Let us now investigate the dependency of $T'$ on $\mu$. $C_2$ has a long expression, but the important thing is that there exists $M>0$ such that for any $\mu >0$ $$C_2 \le \frac{M}{\epsilon^6}T_{max}e^{(1+\tilde{\gamma})\mu(R_1+\epsilon)^2}.$$ One can compute an estimate of $M$. Since $$\max_{x>0}((1+x)e^{-\tilde{\gamma}(R_1+\epsilon)^2x}) \le \max\left(1,\left(\frac{1}{\tilde{\gamma}(R_1+\epsilon)^2}\right)\right) \le \max\left(1,\left(\frac{1}{\tilde{\gamma}R_1^2}\right)\right),$$ we obtain a constant $C$ independant of $\gamma$ and $\epsilon$ such that \begin{equation} \label{estimedeM} M \le \frac{C}{\tilde{\gamma}}.\end{equation}
Let $y \in L^2(\mc{B}_R)$, we denote $y_\mu(t):= e^{-t\grush}y$ the solution of \eqref{harmonicheat}. 
{Combining \cref{dissipationgrushin} and the estimate \eqref{applithmharmonicheat}}, we obtain that for any $T > T'$,
\begin{align*}
\norm{y_\mu(T)}{L^2(\mc{B}_R)}^2 &\le \norm{y_\mu(T')}{L^2(\mc{B}_R)}^2 \\
& \le \frac{C}{\epsilon^6} (1+\mu) e^{\mu (R_1+\epsilon)^2}e^\frac{2}{(sT')^\beta} \int_0^{T'} \norm{y_\mu(t)}{L^2(\omega_\Gamma)}^2dt.
\end{align*}
{Since} 
$$\exp\left(\frac{2}{(sT')^\beta}\right) = \exp\left(\frac{2\log(C_2)}{s_2^\beta (1-x^{\frac{\beta}{\beta+1}})^\beta C_1} \right) = \exp\left(\frac{2\log(C_2)}{s_2^\beta (1-x^{\frac{\beta}{\beta+1}})^\beta \frac{2}{s_1^\beta}(1-x^{\frac{\beta}{\beta+1}})}\right),$$
{the definition of $\kappa$ leads to}
$$\exp\left(\frac{2}{(sT')^\beta}\right) = \exp\left(\tilde{\gamma}\ln(C_2)\right) = C_2^{\tilde{\gamma}} \le \frac{(MT_{max})^{\tilde{\gamma}}}{\epsilon^{6\tilde{\gamma}}}e^{(\tilde{\gamma}+\tilde{\gamma}^2)\mu (R_1+\epsilon)^2}=  \frac{M^{\tilde{\gamma}}}{\epsilon^{6\tilde{\gamma}}}e^{\gamma\mu (R_1+ \epsilon)^2},$$
{from which we obtain}
\begin{equation}
\label{Tge}
\norm{y_\mu(T)}{L^2(\mc{B}_R)}^2 \le \frac{C}{\epsilon^{6\tilde \gamma}} M^{\tilde{\gamma}}T_{max}^{\tilde \gamma}(1+ \mu) e^{(1+\gamma)\mu (R_1+\epsilon)^2}\int_0^{T'} \norm{y_\mu(t)}{L^2(\omega_\Gamma)}^2dt.
\end{equation}
Combining the estimates for $T\le T'$ \eqref{applithmharmonicheat}, and for $T > T'$ \eqref{Tge}, using that $\tilde \gamma <1$, we get the observability inequality given in \cref{observabilityharmonicheatsectionannulus} which ends the proof.

\end{proof}

\subsubsection{Observability of the Baouendi-Grushin equation from $\omega_\Gamma \times \tilde \Omega$}
\label{conclusioninterne}
{We now follow the strategy briefly described in the introduction to obtain the observability of the Baouendi-Grushin system in large time. First, we combine the precise observability estimate \cref{observabilityharmonicheatsectionannulus} with the dissipation given by \cref{dissipationgrushin}, to obtain the observability of the family of harmonic heat equations \eqref{harmonicheat} from $\omega_\Gamma$, uniformly in the parameter $\mu >0$. } 

{
\begin{prop} \label{prop_uniform_observa_wrt_mu_intern}
For $0<R_1 < R_2 < R$ and $\Gamma$ a nonempty open subset of $\sphere$, set $\omega_\Gamma$ defined by \eqref{defdesectionanneau}. 
Define
$$
T^* = \frac{R_1^2}{2d}.
$$
For  all $\beta > 1$, there exists $c_\beta > 0$, verifying $c_\beta \xrightarrow{\beta \rightarrow 1} \infty$, such that for all $\mu > 0$ and all $T>T^*$, any $y_\mu \in C^0([0,T];L^2(\mc{B}_R)) \cap L^2((0,T);H^1_0(\mc{B}_R))$ solution of \eqref{harmonicheat} satisfies
$$
\int_{\mc{B}_R} \vert y_\mu(T) \vert^2 dx \leq e^{\frac{c_\beta}{(T-T^*)^{2\beta}}} \int_0^T \int_{\omega_\Gamma} \vert y_\mu \vert^2 .
$$
\end{prop}
}
 
\begin{proof} {We set $T_{max} = 2 T^*$, and prove the result for $T$ in $(T^*,T_{max}]$. The result follows for $T \geq T_{max}$
from the dissipative properties of the harmonic heat system \eqref{dissipharmonicheatL2}. For $\mu > 0$, and $y \in L^2(\mc{B}_R)$, we denote $y_\mu = e^{-t G_\mu} y$
the solution of \eqref{harmonicheat}.}

Let $T^*<T< T_{max}$, $\epsilon >0$ to be fixed later, and set $0<\gamma < \frac15$ such that $T(1-2\gamma) = T^*(1+\gamma)$. Denote $\delta_\gamma = T\gamma$.
Since the eigenvalues of $G_\mu$ are larger than $d\mu$ (see \cite[section 4.1]{Beauchard2020}) we have by \cref{dissipationgrushin} that
$$\norm{y_\mu(T)}{L^2(\mc{B}_R)}^2 \le \exp\left(-2d\mu(T-\delta_\gamma)\right)\norm{y_\mu(\delta_\gamma)}{L^2(\mc{B}_R)}^2.$$
{We now} estimate $\norm{y_\mu(\delta_\gamma)}{L^2(\mc{B}_R)}^2$. To this end, we apply \cref{obsharmonicheatsectionannulus} and  obtain
\begin{multline*}
\norm{y_\mu(T)}{L^2(\mc{B}_R)}^2\le \frac{C}{\epsilon^6} (1+\mu) \exp(2/(s\delta_\gamma)^\beta)\\
\times \exp\left(\mu \left((1+\gamma )(R_1+\epsilon)^2- 
2d(T-\delta_\gamma)\right)\right)\int_0^T \int_{\omega_\Gamma}\left| y_\mu(t)\right|^2dx dt.
\end{multline*}
We now set
\begin{equation}
\label{defepsilon}
\epsilon = \frac{R_1}{2}\left(\sqrt{\frac{T(1-\gamma)}{T^*(1+\gamma)}}-1\right)=\frac{R_1}{2}\left(\sqrt{\frac{(1-\gamma)}{(1-2\gamma)}}-1\right)
\end{equation}
As $T^* = \frac{R_1^2}{2d}$, and with the choice of $\gamma$, $\delta_\gamma$ and $\epsilon$, we have
$$2d(T - \delta_\gamma) = R_1^2 \frac{T(1-\gamma)}{T^*} > (1+\gamma)(R_1+\epsilon)^2$$
Consequently the observability cost goes to zero as $\mu$ goes to infinity, and we get an estimate uniform in $\mu$. It gives the uniform observability of the harmonic heat equation for $T >T^*$.\newline 
Let us compute precisely the cost of observability when $T \to T^*$. {By definition, we have}
$$\gamma = \frac{T-T^*}{T^* + 2T} \quad \text{and} \quad \delta_\gamma = \frac{T(T-T^*)}{T^* + 2T}.$$
{Using} \eqref{comportementdes}, we obtain the existence of $C_2(\beta)>0$ such that
$$\exp(2/(s\delta_\gamma)^\beta) \le \exp(C_2(\beta)/(T-T^*)^{2\beta}).$$
It remains to check that the constant $(C/\epsilon^6)$, which depends on $\gamma$ and the uniform estimate of \begin{equation} \label{partieenp} (1+\mu) \exp\left(\mu((1+\gamma) R^2 - 2d(T-\delta_\gamma)\right), \end{equation} do not have a blow up rate worse than $\exp(C_2(\beta)/(T-T^*)^{2\beta})$ when $T \to T^*$.
{Estimate \cref{estimedeM} gives} $$C = (T_{max}M)^{\tilde{\gamma}} \le \exp(-\tilde C\gamma \log(\gamma))$$ which is bounded when $\gamma \to 0$. \newline
Moreover, by \eqref{defepsilon} we have 
$$\frac{1}{\epsilon^6} \underset{\gamma \to 0}{\sim} \frac{4^6}{R_1^6 \gamma^6} $$ which {is of order $(T-T^*)^{-6}$}.
Finally using that 
\begin{multline*} \mu((1+\gamma) (R_1+\epsilon)^2 - 2d(T-\delta_\gamma) \\
= -R_1^2\mu\left(-\frac34\frac{(1-\gamma)(1+\gamma)}{1-2\gamma}+ \frac12(1+\gamma)\sqrt{\frac{1-\gamma}{1+\gamma}} + \frac14 (1+\gamma)\right),\\
 \underset{\gamma \to 0}{\sim} -\frac12 \gamma R_1^2\mu, \quad \hspace{8cm}
\end{multline*}
 and that 
$$\max_{x>0}(xe^{-Cx}) \le \frac{1}{C},$$
we obtain a cost not worse than $1/(T-T^*)$ for the part \eqref{partieenp}. 
{The result follows}.
\end{proof}

Using the strategy developed in \cref{section_strategy_HH_Grushin}, we immediately deduce from \cref{prop_uniform_observa_wrt_mu_intern}
the observability of the Baouendi-Grushin system \eqref{grushin} from $\omega_\Gamma \times \tilde \Omega$ for $T$ greater than $T^*$, with 
a cost $K_T$ satisfying
$$
K_T \leq e^{\frac{c_\beta}{(T-T^*)^{2\beta}}}.
$$

\subsubsection{Observability of the Baouendi-Grushin equation from $\omega \times \tilde \Omega$}
\label{generalopensubset}

We now consider $\omega$ to be any open subset of $\mathcal{B}_R$, and define
$$
T^* = \frac{\inf_{x\in \omega} \Vert x \Vert^2}{2d}.
$$
Our goal if to prove that the Baouendi-Grushin equation is observable from $\omega \times \tilde \Omega$ in any time $T>T_*$. 

Let $T>T^*$. There exists $x_* \in \omega$ such that $2 d\, T^*< \Vert x_* \Vert^2 < 2dT$. 
Let $\delta>0$ be such that the ball $\mathcal{B}_\delta(x_*)$ is contained in $\omega$.
For $\alpha >0$, we define
$$
\Gamma_\alpha = \left\lbrace \theta \in \mathbb{S}^{d-1}, \scalar{\theta}{x_*}{\R^d} > r_* (1-\alpha^2) \right\rbrace,$$
and
$$
\omega_{\Gamma_\alpha} = \left\lbrace x = r \theta, \ r_*< r < r_* + \alpha, \ \theta \in \Gamma_\alpha \right\rbrace,
$$
where $r_* =  \Vert x_*\Vert$.

\medskip 

From now on, fix $\alpha>0$ such that $\alpha < \frac{\delta}{2} \min (1, (\sqrt{2} r_*)^{-1} )$. We note that for all $\theta \in \Gamma_\alpha$, 
$$
\Vert r_* \theta - x_*\Vert^2 = 2 r_*^2 - 2 r_* (\theta,x_*) < 2 \alpha^2 r_*^2 < \frac{\delta^2}{4}.
$$
As a result, for any $x = r\theta$ in $\omega_{\Gamma_\alpha}$, $r = \Vert x \Vert$, we have
$$
\Vert x - x_* \Vert \leq \vert r - r_*\vert + \Vert r_*\theta - x_*\Vert < \alpha + \frac{\delta}{2} < \delta.
$$
It follows that $\omega_{\Gamma_\alpha} \subset \mathcal{B}_\delta(x_*) \subset \omega$.
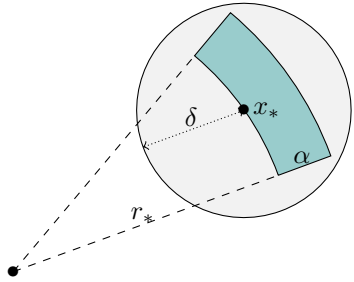
\begin{figure}[h!]
\begin{center}
\begin{tikzpicture}[scale=0.75,rotate=20]

\draw[fill=gray!10] (4.82,1.29) circle (1.9) ;
\draw[fill=teal!40] (5,0) arc (0:30:5) -- (5.19,3) arc (30:0:5.99) -- (5,0);
\draw (4.82,1.29) node[right] {$x_*$} node{$\bullet$};
\draw (0,0) node{$\bullet$};
\draw[dashed] (0,0) -- (5,0);
\draw[dashed] (0,0) -- (4.33,2.5);
\draw[densely dotted,<->] (4.82,1.29) -- (2.92,1.29);
\node at (3.9,1.50) {$\delta$};
\node at (5.5,0.15) {$\alpha$}; 
\node at (2.5,0.15) {$r_*$}; 

\end{tikzpicture}
\caption{In grey, $\mathcal{B}_\delta(x_*)$, and in teal, $\omega_{\Gamma_\alpha}$.}
\end{center}
\end{figure}

\medskip

As $2 d T > r_*^2$, the Baouendi-Grushin equation is observable through $\omega_{\Gamma_\alpha}\times \tilde \Omega$ in time $T$. Therefore, there exits a constant 
$\cc_T$ such that for any function $y$ solution of the Baouendi-Grushin system \eqref{grushin},
$$
\int_{\mathcal{B}_R\times \tilde \Omega} \vert y(T) \vert^2 \leq \cc_T \int_0^T \int_{\omega_{\Gamma_\alpha}\times \tilde \Omega}
\vert y \vert^2 \leq \cc_T \int_0^T \int_{\omega\times \tilde \Omega} \vert y \vert^2 .
$$
The proof is complete.

\subsection{Proof of \cref{mainbord} -- boundary observation}
\subsubsection{Observability  of the harmonic heat equation from a subset of the boundary}

{We now consider the case of boundary observation. To this end, we slightly modify the functional framework: the space $\mathcal{E}$ remains $L^2(\mathcal{B}_R)$, while $\mathcal{F}$ is now taken as $H^1_0(\mathcal{B}_R)$, equipped with the norm
\begin{equation}\label{munorm}
\| u \|_\mu^2 = \|\nabla u\|_{L^2(\mathcal{B}_R)}^2 + \mu^2 \| u \|_{L^2(\mathcal{B}_R)}^2, \quad \forall u \in H^1_0(\mathcal{B}_R).
\end{equation}
This norm is equivalent to the standard $H^1_0$ norm and facilitates the analysis of the observability cost. One can see that observability in this norm implies observability for the standard $H^1_0$ norm.
}The goal of this section is to give a precise estimate of the cost of observability of the harmonic heat equation from $\Gamma_{\!R}$. Our result reads:
\begin{thm} 
\label{obsharmonicheatsubsetboundary}
Let $\Gamma_{\!R} \subset \partial \mc{B}_R$ be a nonempty open subset and $B_{\Gamma_{\!R}} = \indicatrice{\Gamma_{\!R}}\frac{\partial \cdot}{\partial n}$, for any $T_{max}>0$, $\beta >1$ and $0<\gamma<1/2$, there exists $C,s >0$ such that for any $\mu>0$ and $T\in (0,T_{max}]$, $(\grush,B_{\Gamma_{\!R}})$ is observable at time $T$ with
\begin{equation}
\label{observabilityharmonicheatsubsetboundary}
\cost(\grush,B_{\Gamma_{\!R}},T) \le C (1+\mu)^3 e^{\mu (1+\gamma )R^2}e^\frac{2}{(sT)^\beta}.
\end{equation}
Moreover, the behaviour of $s$ is {given by \eqref{comportementdes},} as for the internal case .
\end{thm}
\begin{rem}
The cost $\cost(\grush,B_{\Gamma_{\!R}},T)$ is given for the observability in the $\mu$-norm \eqref{munorm}.
\end{rem}
To prove this, we will apply the abstract restriction result given in \cref{preuvemiller}, to the observability on the whole boundary given by \cref{prop_observability_from_boundary}, which rewrites in the semigroup framework \eqref{notationharmonicheat} as follows.
\begin{cor} 
\label{obsharmonicheatwholeboundary}
Let $B_{\partial \mc{B}_R} = \indicatrice{\partial \mc{B}_R}$.  There exists $C>0$, and for any $\beta >1$, a positive constant $c_\beta$ verifying $c_\beta \xrightarrow{\beta \rightarrow 1} \infty$, such that for any $\mu >0$ and $T>0$, $(\grush,B_{\partial \mc{B}_R})$ is observable at time $T$ with
\begin{equation}
\label{observabilityharmonicheatwholeboundary}
\cost(\grush,B_{\partial \mc{B}_R},T) \le C (1+\mu)^3 e^{\mu R^2}e^\frac{2c_\beta}{T^\beta}.
\end{equation}
\end{cor}

\begin{proof}[of \cref{obsharmonicheatsubsetboundary}] We want to apply \cref{LR}, to deduce the observability for the operator $B_{\Gamma_{\!R}}$ from the observability for $B_{\partial \mc{B}_R}$ which is given by \cref{obsharmonicheatwholeboundary}. We start by checking the assumptions, with $\mc{E}$ and $\mc{F}$ given at the begining of this subsection and $U =  L^2(\partial \mc{B}_R)$ for the observation space. We define for any $\sigma > 0$,
\begin{equation} \label{fsigma} \mc{F}_\sigma = \left\{ \sum_{\underset{k \in \llbracket1, \beta(m,d)\rrbracket}{m<\sqrt{\sigma}}}\scalar{y}{Y_{mk}}{L^2(\sphere)}Y_{mk} \quad \Big{|} \quad y \in H^1_0(\mc{B}_R) \right\} \subset \mc{F},\end{equation}
\newline
\noindent $\bullet$ Admissibility: Let $y \in \mathcal{D}(\grush)$ and $T>0$, we denote $y_\mu(t) = e^{-t\grush}y$ the solution of \eqref{harmonicheat}. By the continuity of the trace operator and the elliptic regularity of the Laplacian, we have that 
$$ \int_0^T \norm{B_{\Gamma_{\!R}} y_\mu(t)}{L^2(\partial \mc{B}_R)}^2dt \le C\int_0^T\norm{\Delta y_\mu(t)}{L^2( \mc{B}_R)}^2 dt.$$
Then by multiplying the equation \eqref{harmonicheat} by $-\Delta y_\mu$, integrating on $\mathcal{B}_R$  and doing an integration by part we obtain
$$\frac12 \frac{d}{dt}\left(\norm{\nabla y_\mu}{L^2( \mc{B}_R)}^2\right) + \norm{\Delta y_\mu}{L^2( \mc{B}_R)}^2 -\int_{\mathcal{B}_R}\mu^2\|x\|^2y_\mu\Delta y_\mu dx= 0. $$
Let us compute and estimate the last term at every time $t$:
\begin{align*}
-\int_{\mathcal{B}_R}\mu^2\|x\|^2 y_\mu \Delta y_\mu dx &= \int_{\mathcal{B}_R}\mu^2\|x\|^2|\nabla y_\mu|^2 dx + \int_{\mathcal{B}_R}2\mu^2 y_\mu x\cdot\nabla y_\mu dx \\
&\ge \int_{\mathcal{B}_R}\mu^2x\cdot\nabla(y_\mu^2) dx \\
&= -\mu^2 d \int_{\mathcal{B}_R} y_\mu^2dx,
\end{align*}
then using the dissipation of the $L^2$-norm we obtain $$\frac12 \frac{d}{dt}\left(\norm{\nabla y_\mu}{L^2( \mc{B}_R)}^2\right) + \norm{\Delta y_\mu}{L^2( \mc{B}_R)}^2  \le \mu^2 d \norm{y}{L^2( \mc{B}_R)}^2$$
We integrate over $(0,T)$ and deduce the admissibility inequality with $Adm_T = \frac{C}{2} + CTd$. \newline \medskip 
\noindent $\bullet$ Dissipation: the proof is exactly the same as for the internal observation, using the spaces $\mc{F}_\sigma$ defined in \eqref{fsigma} and the dissipation \cref{dissipationgrushin} for the $\mu$-norm instead of the $L^2$-norm. \newline
We obtain the dissipation estimate with $C_{dissip} = (1+ R^2)$ and $c = 1/R^2$. \newline \medskip 
\noindent $\bullet$ Relation between $B_{\partial \mathcal{B}_R}$ and $B_{\Gamma_{\!R}}$: Let $y \in \mc{F}_\sigma$ we have to bound
$\displaystyle \norm{\frac{\partial y}{\partial n}}{L^2(\partial \mathcal{B}_R)}^2$ by $\displaystyle \norm{\frac{\partial y}{\partial n}}{L^2(\Gamma_{\!R})}^2.$
As for the internal case we momentarily distinguish $y$ the function expressed is cartesian coordinates and $\hat{y}$ the function expressed in polar coordinates $(r,\theta) \in (0,R) \times \sphere$. We have that 
$$\hat{y}(r,\theta) = \sum_{\underset{k \in \llbracket1, \beta(m,d)\rrbracket}{m < \sqrt{\sigma}}}\hat{y}_{mk}(r) Y_{mk}(\theta).$$ We compute
$$\frac{\partial \hat{y}}{\partial n}(R,\theta) = \sum_{\underset{k \in \llbracket1, \beta(m,d)\rrbracket}{m < \sqrt{\sigma}}}\hat{y}_{mk}'(R) Y_{mk}(\theta).$$ Notice that the normal derivative is a sum of eigenfunctions of the Laplace-Beltrami operator on $\sphere$. We denote $\Gamma = \{ \theta \in \sphere ~;~ R\theta \in \Gamma_{\!R} \}$ which is a nonempty open subset of the sphere, and apply the spectral inequality \cite[Theorem 14.6]{Jerison}, which gives a constant $C>0$ depending on $\Gamma$ such that
$$ \norm{ \sum_{\underset{k \in \llbracket1, \beta(m,d)\rrbracket}{m < \sqrt{\sigma}}}\hat{y}_{mk}'(R) Y_{mk}}{L^2(\sphere)}^2 \le Ce^{C\sqrt{\sqrt{\sigma}(\sqrt{\sigma}+d-2)}}\norm{ \sum_{\underset{k \in \llbracket1, \beta(m,d)\rrbracket}{m < \sqrt{\sigma}}}\hat{y}_{mk}'(R) Y_{mk}}{L^2\left(\Gamma\right)}^2.$$
Here we use that the eigenvalue associated to $Y_{mk}$ is $m(m+d-2)$ (see \eqref{laplacienvp}).
In conclusion, we have obtained that there exists $C_{rel}>0$, $a>0$ depending only on $\Gamma$ and $d$ such that
$$\norm{\frac{\partial \hat{y}}{\partial n} }{L^2(\sphere)}^2 \le C_{rel}e^{2a\sqrt{\sigma}}\norm{\frac{\partial \hat{y}}{\partial n} }{L^2\left(\Gamma\right)}^2.$$
Since $\hat{y}(r,\theta) = y(r\theta)$, and $\frac{\partial \hat{y}}{\partial n}(r,\theta) = \nabla y(r\theta) \cdot \theta$, we finally have
$$\norm{\frac{\partial y}{\partial n} }{L^2(\partial \mc{B}_R)}^2 \le C_{rel}e^{2a\sqrt{\sigma}}\norm{\frac{\partial y}{\partial n} }{L^2\left(\Gamma_{\!R} \right)}^2,$$
which gives the expected inequality. \newline \medskip 
\noindent $\bullet$ Observability for $B_{\partial \mathcal{B}_R}$: it is given by \cref{obsharmonicheatwholeboundary}, with $C_{obs} = C (1+\mu)^3 e^{\mu R^2}$.
\bigskip \newline
Now we can apply \cref{LR}, there exists $s>0$ (which does not depend on $\mu$) and $T'>0$ (which depends on $\mu$ as we will see in the expression given below) such that for any $T \in (0,T']$, $(\grush,B_{\Gamma_{\!R}})$ is observable with
\begin{equation} \label{applithmharmonicheatboundary} \cost(\grush,B_{\Gamma_{\!R}},T) \le  4C_{rel} C (1+\mu)^3 e^{\mu R^2}e^\frac{2}{(sT)^\beta}.\end{equation}
If $T \le T'$, then the result is proven. The problem is that $T'$ depends on $\mu$ and goes to $0$ as $\mu$ goes to infinity. Indeed using the notations of \cref{LR} we have
$$T' = \left\{ \begin{array}{cc} \left(\dfrac{C_1}{\ln(C_2)}\right)^\frac{1}{\beta} &\text{ if } C_2 > 1 \\ T_{max} &\text{ if } C_2 \le 1 \end{array} \right.$$
and $$C_2 = 4(1+R^2)C(1/2 + T_{max}d)C_{rel}C(1+\mu)^3 e^{\mu R^2} + (1+R^2)\exp(-2/s_2^\beta).$$
One can see that $C_2 \to +\infty $ as $\mu \to +\infty$. 
Therefore, we need to be more precise to tackle the case $T>T'$. The strategy is the same as in the internal case, we use parameter $\kappa$ given in \cref{LR} to get the expected estimate.
Let $0<\gamma<1/2$, we choose the same $\tilde{\gamma}$ and $\kappa$ as for the internal case, and get the same equivalence for $s$, \eqref{comportementdes}.\newline
Let us now investigate the dependency of $T'$ on $\mu$. $C_2$ has a long expression, but the important thing is that there exists $M>0$ such that for any $\mu >0$ $$C_2 \le MT_{max} e^{(1+\tilde{\gamma})\mu R^2}.$$ As for the internal case we compute an estimate of $M$, we obtain a constant $C$ independant of $\gamma$ such that \begin{equation} \label{estimedeMboundary} M \le \frac{C}{\tilde{\gamma}^3}.\end{equation}
Let $y \in H^1_0(\mc{B}_R)$, we denote $y_\mu(t):= e^{-t\grush}y$ the solution of \eqref{harmonicheat}. Combining \cref{dissipationgrushin} and the estimate \eqref{applithmharmonicheatboundary}, we obtain that for any $T > T'$ ,
\begin{align*}
\norm{y_\mu(T)}{\mu}^2 &\le (1+ R^2)\norm{y_\mu(T')}{\mu}^2 \\
& \le (1+ R^2)C(1+ \mu)^3 e^{\mu R^2}e^\frac{2}{(sT')^\beta} \int_0^{T'} \norm{\frac{\partial y_\mu(t)}{\partial n}}{L^2(\Gamma_{\!R})}^2dt.
\end{align*}
Since $$\exp\left(\frac{2}{(sT')^\beta}\right) = \exp\left(\frac{2\log(C_2)}{s_2^\beta (1-x^{\frac{\beta}{\beta+1}})^\beta C_1} \right) = \exp\left(\frac{2\log(C_2)}{s_2^\beta (1-x^{\frac{\beta}{\beta+1}})^\beta \frac{2}{s_1^\beta}(1-x^{\frac{\beta}{\beta+1}})}\right),$$
the definition of $\kappa$ leads to
$$\exp\left(\frac{2}{(sT')^\beta}\right) = \exp\left(\tilde{\gamma}\ln(C_2)\right) = C_2^{\tilde{\gamma}} \le (MT_{max})^{\tilde{\gamma}}e^{(\tilde{\gamma}+\tilde{\gamma}^2)\mu R^2}=  (MT_{max})^{\tilde{\gamma}}e^{\gamma\mu R^2},$$
from which we obtain
\begin{equation}
\label{Tgeboundary}
\norm{y_\mu(T)}{\mu}^2 \le C (MT_{max})^{\tilde{\gamma}} (1+\mu)^3 e^{(1+\gamma)\mu R^2}\int_0^{T'} \norm{\frac{\partial y_\mu(t)}{\partial n}}{L^2(\Gamma_{\!R})}^2dt.
\end{equation}
Combining the estimates for $T\le T'$ \eqref{applithmharmonicheatboundary}, and for $T > T'$ \eqref{Tgeboundary}, we get the observability inequality given in \cref{observabilityharmonicheatsubsetboundary} which ends the proof.

\end{proof}

\subsubsection{Observability of the Baouendi-Grushin equation from $\Gamma_{\!R} \times \tilde \Omega$}
{Just as in the case of internal observation, the observability of the Baouendi-Grushin equation from  $\Gamma_{\!R} \times \tilde \Omega$ in large time
is a direct consequence of the following uniform observability result for the harmonic heat equations parametrized by $\mu > 0$.}

{
\begin{prop} \label{prop_uniform_observa_wrt_mu_boundary}
Let $\Gamma_{\! R}$ be a nonempty open subset of $\partial \mc B_R$. 
Define
$$
T^* = \frac{R^2}{2d}.
$$
For  all $\beta > 1$, there exists $c_\beta > 0$, verifying $c_\beta \xrightarrow{\beta \rightarrow 1} \infty$, such that for all $\mu > 0$ and all $T>T^*$, any $y_\mu \in C^0([0,T];L^2(\mc{B}_R)) \cap L^2((0,T);H^1_0(\mc{B}_R))$ solution of \eqref{harmonicheat} satisfies
$$
\Vert y_\mu(T) \Vert_\mu^2 \leq e^{\frac{c_\beta}{(T-T^*)^{2\beta}}} \int_0^T \int_{\Gamma_{\!R}} \left\vert \frac{\partial y_\mu}{\partial n} \right\vert^2 .
$$
\end{prop}
}

\begin{proof}
For $y \in H^1_0(\Omega)$, we denote $y_\mu(t) = e^{-t G_\mu}y$ the solution of \eqref{harmonicheat}.
For the same reasons as for the internal case, we set $T_{max} = 2T^*$.\newline
Let $0<T<T_{max}$, set $0<\gamma<1/2$ such that $T(1-2\gamma) = T^*(1+\gamma)$. Denote $\delta_\gamma = T\gamma$.
Since the eigenvalues of $G_\mu$ are larger than $d\mu$ (see \cite[section 4.1]{Beauchard2020}) we have by \cref{dissipationgrushin} that
$$\norm{y_\mu(T)}{\mu}^2 \le (1+ R^2)\exp\left(-2d\mu(T-\delta_\gamma)\right)\norm{y_\mu(\delta_\gamma)}{\mu}^2.$$
Now we have to estimate $\norm{y_\mu(\delta_\gamma)}{\mu}^2$. To this end, we apply \cref{obsharmonicheatsubsetboundary} and we obtain
\begin{multline*}
\norm{y_\mu(T)}{\mu}^2\le (1+ R^2) C (1+\mu)^3 \exp(2/(s\delta_\gamma)^\beta)\\
\times \exp\left(\mu \left((1+\gamma )R^2- 
2d(T-\delta_\gamma)\right)\right)\int_0^T \int_{\Gamma_{\!R}}\left| \frac{\partial y_\mu(t)}{\partial n}\right|^2d\sigma dt.
\end{multline*}

As $T^* = \frac{R^2}{2d}$ and with the choice of $\gamma$ and $\delta_\gamma$ we have
$$2d(T - \delta_\gamma) = 2dT(1-\gamma) > 2dT^*(1+\gamma) = R^2(1+\gamma).$$
Consequently the observability cost goes to zero as $\mu$ goes to infinity, and we get an estimate uniform in $\mu$. It gives the uniform observability of the harmonic heat equation for $T >T^*$.\newline 
Let us compute precisely the cost of observability when $T \to T^*$. {By definition, we have}
$$\gamma = \frac{T-T^*}{T^* + 2T} \quad \text{and} \quad \delta_\gamma = \frac{T(T-T^*)}{T^* + 2T}.$$
Therefore with the equivalence given in \eqref{comportementdes} we obtain that there exists $C_3(\beta)>0$ such that
$$\exp(2/(s\delta_\gamma)^\beta) \le \exp(C_3(\beta)/(T-T^*)^{2\beta}).$$
Now it remains to check that the constant $C$ which depends on $\gamma$ and the uniform estimate of \begin{equation} \label{partieenpboundary} (1+\mu)^3 \exp\left(\mu((1+\gamma) R^2 - 2d_1(T-\delta_\gamma)\right) \end{equation} do not have a blow up rate worse than $\exp(C_3(\beta)/(T-T^*)^{2\beta})$ when $T \to T^*$.
By the estimate \cref{estimedeMboundary} we have that $$C = (MT_{max})^{\tilde{\gamma}} \le \exp(-\tilde C\gamma \log(\gamma))$$ which is bounded when $\gamma \to 0$.
Finally using that 
Using that 
$$ \mu ((1+\gamma) R^2 - 2d(T-\delta_\gamma) = -R^2\mu\frac{\gamma(1+\gamma)}{1-2\gamma},$$ and that 
$$\max_{x>0}(x^3e^{-Cx}) \le \left(\frac{3}{C}\right)^3,$$
we obtain a cost not worse than $1/(T-T^*)^3$ for the part \eqref{partieenpboundary}. {The result follows}.
\end{proof}

\section{Optimality of the results}
\label{optimality}

\subsection{Non-observability for $T<T^*$}

{In this section, we prove that the Baouendi-Grushin system~\eqref{grushin} is not observable in arbitrary small time from $\omega \times \tilde{\Omega}$, where $\omega$ is an open subset of $\mathcal{B}_R$. Our main result, presented below, gives the necessary conditions stated in \cref{maininterne} and \cref{prop_necessary_condition_dim2and3}. To obtain this result, we follow the strategy developed for the case
$d= \tilde d = 1$ in \cite{Beauchard2013}.
}

\begin{prop} \label{prop_necessary_condition}
Let $\omega$ be an open subset of $\mathcal B_R$, and define 
$$
r = \inf_{x\in \omega} \Vert x \Vert.
$$ 
Suppose $T>0$ is such that the Grushin equation is observable through $\omega \times \tilde \Omega$ in time $T$. Then
$$
T\geq \frac{r^2}{2\,d}.
$$
Furthermore, if $d = 2$ or $3$, $T > \frac{r^2}{2\, d}$.
\end{prop}

For \(\mu \geq 0\), recall that the operator \(G_\mu: L^2(\mathcal{B}_R) \to L^2(\mathcal{B}_R)\) is defined by
\[
G_\mu := -\Delta + \mu^2 \|x\|^2, \quad \mathcal{D}(G_\mu) = H^2 \cap H_0^1(\mathcal{B}_R).
\]
This operator admits a countable family of eigenvalues
\[
\nu_{\mu,1} < \nu_{\mu,2} \leq \nu_{\mu,3} \leq \cdots \leq \nu_{\mu,p} \xrightarrow{p \to \infty} \infty.
\]
One can select a corresponding countable family of eigenfunctions \((\varpi_{\mu,p})_{p \in \mathbb{N}}\), normalized in \(L^2(\mathcal{B}_R)\), such that they form a Hilbert basis for \(L^2(\mathcal{B}_R)\).

In the following, we denote $\nu_\mu$ the smallest eigenvalue of $G_\mu$, that is $\nu_\mu = \nu_{\mu,1}$, and $\varpi_\mu = \varpi_{\mu,1}$
its corresponding normalized and non-negative eigenfunction. From the min-max principle, we have
$$
\nu_\mu = \min_{\mathcal{w} \in H^1_0(\mathcal{B}_R)} \frac{\displaystyle \int_{\mathcal{B}_R} \left( \Vert \nabla \mathcal{w} \Vert^2 +\mu^2 \Vert x \Vert^2 \mathcal{w}^2\right)}{\displaystyle \int_{\mathcal{B}_R} \mathcal{w}^2}.
$$

Besides, setting 
$$
\mathcal V = \left\lbrace \mathcal{w} \in H^1(\R^d), \ x\mathcal{w} \in L^2(\R^d) \right\rbrace,
$$
we recall that 
$$
\min_{\mathcal{w}\in \mathcal V} \frac{\displaystyle \int_{\R^d} (\Vert \nabla \mathcal{w} \Vert^2 + \Vert x \Vert^2 \mathcal{w}^2)}{\displaystyle \int_{\R^d} \mathcal{w}^2} = d = \frac{\displaystyle \int_{\R^d} (\Vert \nabla \mathcal{G} \Vert^2 + \Vert x \Vert^2 \mathcal{G}^2)}{\displaystyle \int_{\R^d} \mathcal{G}^2},
$$
with $\mathcal{G}(x) = e^{- \frac{\Vert x \Vert^2}{2}}$. In particular, 
$
- \Delta \mathcal{G} + \Vert x \Vert^2 \mathcal{G} = d \mathcal{G}.
$

\begin{lem} \label{lemma_bound_firsteigen}
There exists a continuous function $\kappa \colon [0, \infty[ \to [0, \infty[$ such that $\kappa(\mu) \to 0$ as $\mu \to +\infty$, and for all $\mu \geq 0$,
\[
d\mu \leq \nu_\mu \leq d\mu + \kappa(\mu).
\]
\end{lem}

\begin{proof}

$\bullet$ Let $\mathcal{w} \in H^1_0(\mathcal{B}_R)$. Define $\phi: x \in \mathcal{B}_{\sqrt{\mu} R} \mapsto \mathcal{w}\left( \frac{x}{\sqrt \mu}\right)$, which belongs to $H^1_0(\mathcal{B}_{\sqrt \mu R})$. We have
\begin{multline*}
\frac{\displaystyle \int_{\mathcal{B}_R} (\Vert \nabla \mathcal{w} \Vert^2 + \mu^2 \Vert x \Vert^2 \mathcal{w}^2)}{\displaystyle \int_{\mathcal{B}_R} \mathcal{w}^2}
 = \mu\, \frac{\displaystyle \int_{\mathcal{B}_{\sqrt \mu R}} (\Vert \nabla \phi \Vert^2 + \Vert x \Vert^2 \phi^2)}{\displaystyle\int_{\mathcal{B}_{\sqrt \mu R}} \phi^2}
\\ \geq \mu \, \min_{\phi \in H^1_0(\mathcal{B}_{\sqrt\mu R})} \frac{\displaystyle \int_{\mathcal{B}_{\sqrt \mu R}} (\Vert \nabla \phi \Vert^2 + \Vert x \Vert^2 \phi^2)}{\displaystyle \int_{\mathcal{B}_{\sqrt \mu R}} \phi^2} \\
 \geq \mu\, \min_{\phi \in \mathcal V} \frac{\displaystyle \int_{\R^d} (\Vert \nabla \phi \Vert^2 + \Vert x \Vert^2 \phi^2)}{\displaystyle \int_{\R^d} \phi^2} = \mu d.
\end{multline*}
As this inequality is true for any $\mathcal{w}$ in $H^1_0(\mathcal{B}_R)$, this establishes the left inequality.

$\bullet$ Given that \(\nu_\mu\) is always positive and continuously dependent on \(\mu\), proving the right inequality for \(\mu \geq 1\) is sufficient.
We set \(\kappa \colon \mu \geq 1 \mapsto \nu_\mu - \mu d\). The result will follow by showing that the function \(\kappa\) vanishes at infinity.
We define
$$
\mathcal{w} : x\in \mathcal{B}_R \mapsto \mu^{\frac{d}{4}} \left(\mathcal{G}(\sqrt{\mu} x) - \theta\left( \frac{\Vert x \Vert}{R}\right) e^{-\frac{\mu R^2}{2}}\right),
$$
where $\theta : \R \mapsto [0,1]$ is a smooth function verifying $\theta(1) = 1$ and $\theta(s) = 0$ for $ s  \leq \frac{1}{2}$.
By construction, $\mathcal{w}$ belongs to $H^1_0(\mathcal{B}_R)\cap C^\infty(\R^{d})$, and satisfies
$$
\int_{\mathcal{B}_R} \left(\Vert \nabla \mathcal{w} \Vert^2 + \mu^2 \Vert x \Vert^2 \mathcal{w}^2 \right) 
= \mu d \int_{\mathcal{B}_R} \mathcal{w}^2 + \mu^{\frac{d}{4}} e^{-\frac{\mu R^2}{2}} \int_{\mathcal{B}_R} H \mathcal{w},
$$
where 
$$
H = \frac{1}{R^2} \theta''\left(\frac{\Vert x\Vert}{R} \right) + \frac{1}{R \Vert x \Vert}(d - 1) \theta'\left(\frac{\Vert x \Vert}{R}\right) + \left( \mu d - \mu^2 \Vert x \Vert^2 \right) \theta \left( \frac{\Vert x \Vert}{R} \right).
$$
Note that there exists a constant $\cc$ such that $\Vert H \Vert_{L^\infty(\mathcal{B}_R)} \leq \cc \mu^2$, for all $\mu \geq 1$.

We deduce that 
$$
\kappa(\mu) \leq \mu^{\frac{d}{4}} e^{-\frac{\mu R^2}{2}} \frac{\int_{\mathcal{B}_R} H \mathcal{w}}{\int_{\mathcal{B}_R} \mathcal{w}^2}
\leq \mu^{\frac{d}{4}} e^{-\frac{\mu R^2}{2}} \sqrt{\frac{\int_{\mathcal{B}_R} H^2}{\int_{\mathcal{B}_R} \mathcal{w}^2}}.
$$
As on one side
$$
\int_{\mathcal{B}_R}  H^2 \leq \cc \mu^4,
$$
and on the other side
$$
\int_{\mathcal{B}_R} \mathcal{w}^2 \geq \int_{\mathcal{B}_{\frac{R}{2}}} \mathcal{w}^2
 = \mu^{\frac{d}{2}} \int_{\mathcal{B}_{\frac{R}{2}}} e^{-\mu \Vert x \Vert^2}
  = \int_{\mathcal{B}_{\sqrt \mu \frac{R}{2}}} \mathcal{G}^2 \geq \int_{\mathcal{B}_{\frac{R}{2}}} \mathcal{G}^2,
$$
the result follows.
\end{proof}

\begin{lem} \label{lemma_gap_firstandothereigen}
Let \(\nu\) be any eigenvalue of \(G_\mu\) distinct from \(\nu_\mu\). Then,
\[
\nu - d \mu \geq 2\mu.
\]
\end{lem}

\begin{proof}
It suffices to show that \(\nu_{\mu,2} -d \mu \geq 2\mu \). 
We recall that 
$$
\min_{\substack{V \subset \mathcal{V}\\ \mathsf{dim}(V) = 2}} \max_{w \in V} \frac{\displaystyle \left(\int_{\R^d} \Vert \nabla w\Vert^2 + \Vert x \Vert^2 w^2\right)}{\int_{\R^d} w^2} = d + 2,
$$
with 
$$
\mathcal{V} = \left\lbrace v\in H^1(\R^d), \ x v \in L^2(\R^d) \right\rbrace.
$$
Additionnaly, setting $V_2 = \text{Span}(\varpi_\mu,\varpi_{\mu,2})$, we know that
$$
\nu_{\mu,2} = \frac{\displaystyle \int_{\mathcal{B}_R} \left( \Vert \nabla \varpi_{\mu,2} \Vert^2 + \mu^2 \Vert x \Vert^2 \varpi_{\mu,2}^2\right)}{\displaystyle \int_{\mathcal{B}_R} \varpi_{\mu,2}^2}
= \max_{w \in V_2} \frac{\displaystyle \int_{\mathcal{B}_R} \left( \Vert \nabla w \Vert^2 + \mu^2 \Vert x \Vert^2 w^2\right)}{\displaystyle \int_{\mathcal{B}_R} w^2}.
$$
As the operator $T: H^1_0(\mathcal{B}_R) \mapsto H^1_0(\mathcal{B_{\sqrt \mu R}})$, defined by 
$$
T(w) = w\left(\frac{.}{\sqrt \mu}\right), \ \forall w \in H^1_0(\mathcal{B}_R),  
$$
is bijective and satisfies 
$$
\frac{\displaystyle \int_{\mathcal{B}_R} \left(\Vert \nabla w \Vert^2 + \mu^2 \Vert x \Vert^2 w^2\right)}{\displaystyle\int_{\mathcal{B}_R} w^2}
= \mu \frac{\displaystyle \int_{\mathcal{B}_{\sqrt \mu R}} \left( \Vert \nabla (T w) \Vert^2 + \Vert x \Vert^2 (Tw)^2\right)}{\displaystyle\int_{\mathcal{B}_{\sqrt \mu R}} (Tw)^2},
$$
we deduce that
$$
\nu_{\mu,2} = \mu \max_{\tilde w \in \tilde V_2} \frac{\displaystyle\int_{\mathcal{B}_{\sqrt \mu R}} \left( \Vert \nabla \tilde w \Vert^2 +  \Vert x \Vert^2 \tilde w^2\right)}{\int_{\mathcal{B}_{\sqrt \mu R}} \tilde w^2},
$$
with $\tilde V_2 = \text{Span}(T \varpi_\mu, T \varpi_{\mu,2})$. It follows immediately that \(\nu_{\mu,2} \geq \mu (d + 2)\).
\end{proof}

\begin{lem} \label{lemma_bound_eigenfun}
For $r>0$ such that $r< R$, set $\mathcal A = \left\lbrace x \in \mathcal{B}_R, \Vert x \Vert > r \right\rbrace.$
There exists $\cc > 0$ such that, for all $\mu \geq 1$, 
$$
\int_{\mathcal A} \varpi_\mu^2 \leq c \left( \mu^{\frac{d}{2}-1} e^{-\mu r^2} + \mu^{2+\frac{d}{2}} e^{-\mu R^2} \right).
$$
\end{lem}

\begin{proof} In what follows, \( c \) denotes a constant independent of \( \nu \), whose value may vary from line to line.

We set 
$$
g_\mu : x \in \mathcal{B}_R \mapsto C_\mu \mu^{\frac{d}{4}} \left( \mathcal{G}(\sqrt \mu x) - \theta\left(\frac{\Vert x \Vert}{R}\right)  e^{-\frac{\mu R^2}{2}}\right), 
$$
with $\theta : \R \mapsto [0,1]$ smooth verifying $\theta(1) = 1$ and $\theta(s) = 0$ for all $s \leq \frac{1}{2}$, and
$C_\mu$ fixed such that $\Vert g_\mu \Vert_{L^2(\mathcal{B}_R)} = 1$. It has already been shown in the Proof of \cref{lemma_bound_firsteigen} that, for all $\mu \geq 1$, 
$C_\mu  \leq \cc$, and that 
$$
- \Delta g_\mu + \mu^2 \Vert x \Vert^2 g_\mu = d \mu g_\mu + C_\mu \mu^{\frac{d}{4}} e^{- \frac{\mu R^2}{2}}H,
$$
where $\Vert H \Vert_{L^\infty(\mathcal{B}_R)} \leq \cc \mu^2$. Note that, since \(\theta\) takes values in \([0, 1]\), the function \(g_\mu\) is non-negative on \(\mathcal{B}_R\).

\medskip

We note that 
$$\Vert g_\mu \Vert_{L^2(\mathcal{A})} \leq \cc  \mu^{\frac{d}{4}} \left( \sqrt{ \int_{\mathcal{A}} \mathcal{G}(\sqrt \mu x)^2 } +
e^{-\frac{\mu R^2}{2}}
 \right).$$
As 
\begin{multline*}
\mu^{\frac{d}{2}} \int_{\mathcal{A}} \mathcal{G}(\sqrt \mu x)^2 = \mu^{\frac{d}{2}} \int_{\mathcal{A}}
e^{-\mu \Vert x \Vert^2} = \cc \mu^{\frac{d}{2}} \int_{r}^R e^{-\mu r^2} r^{d - 1} dr
\\ \leq \cc \int_{\sqrt \mu r}^\infty e^{-s^2} s^{d - 1} ds 
= \cc \Gamma\left( \frac{d}{2},\mu r^2 \right) \sim_{\mu \rightarrow \infty} \cc \mu^{\frac{d}{2}-1} e^{-\mu r^2},
\end{multline*}
we deduce that
$$
\int_\mathcal{A} g_\mu^2 \leq \cc \left( \mu^{\frac{d}{2}-1} e^{-\mu r^2} + \mu^{\frac{d}{4}} e^{-\mu R^2} \right).
$$
\medskip

A direct computation shows that for all indices $k$,  
$$
\nu_{\mu,k}(\varpi_{\mu,k},g_\mu)_{L^2(\mathcal{B}_R)} = d \mu (g_\mu,\varpi_{\mu,k})_{L^2(\mathcal{B}_R)} + (H,\varpi_{\mu,k})_{L^2(\mathcal{B}_R)}. 
$$
Using \cref{lemma_gap_firstandothereigen}, we deduce that
$$
\sum_{k\geq 2} (g_\mu,\varpi_{\mu,k})_{L^2(\mathcal{B}_R)}^2 = \sum_{k\geq 2}  \frac{(H,\varpi_{\mu,k})_{L^2(\mathcal{B}_R)}^2}{(\nu_{\mu,k}-d\mu)^{2}} \leq \cc \mu^{-2} \Vert H \Vert_{L^2(\mathcal{B}_R)}^2 \leq \cc \mu^{2 +\frac{d}{2}} e^{-\mu R^2}.
$$
As both $\varpi_\mu$ and $g_\mu$ are non-negative on $\mathcal{B}_R$, and $\Vert g_\mu\Vert_{L^2(\mathcal{B}_R)}= 1$, we have
$$
(g_\mu,\varpi_\mu)_{L^2(\mathcal{B}_R)} = \sqrt{1 - \sum_{k\geq 2} (g_\mu,\varpi_{\mu,k})_{L^2(\mathcal{B}_R)}^2}\geq
1 -  \sum_{k\geq 2} (g_\mu,\varpi_{\mu,k})_{L^2(\mathcal{B}_R)}^2.
$$
As a consequence,
\begin{multline*}
\Vert \varpi_{\mu} - g_\mu\Vert^2_{L^2(\mathcal{B}_R)} = 2 - 2 (\varpi_\mu,g_\mu)_{L^2(\mathcal{B}_R)}
= 2 - 2\left( 1 -  \sum_{k\geq 2} (g_\mu,\varpi_{\mu,k})_{L^2(\mathcal{B}_R)}^2 \right)^{\frac{1}{2}} \\
\leq 2 \sum_{k\geq 2} (g_\mu,\varpi_{\mu,k})_{L^2(\mathcal{B}_R)}^2 \leq \cc \mu^{2+\frac{d}{2}} e^{-\mu R^2}.
\end{multline*}
The result follows.
\end{proof}

\begin{proof}[\cref{prop_necessary_condition}]
First, we note that if the Baoouendi-Grushin system is observable through $\omega \times \tilde \Omega$ in time $T$, it is necessarily observable throught
$\mathcal A \times \tilde \Omega$, with
$$
\mathcal{A} = \left\lbrace x \in \R^d, \ r < \Vert x \Vert < R \right\rbrace.
$$
Consequently, there exists a constant $\cc$ such that, for every sufficiently smooth solution \( y \) of the Baouendi-Grushin system \eqref{grushin},
\begin{equation} \label{eq_observ_proof_necessary}
\int_{\mathcal{B}_R \times \tilde \Omega} y(T)^2 \leq \cc \int_0^T \int_{\mathcal{A}\times \tilde \Omega} y^2.
\end{equation}
Let $(\mu_p^2, \phi_p)_{p\in \N}$ be the eigenvalues and eigenfunctions, normalized in $L^2(\tilde \Omega)$, of the Dirichlet-Laplacian on $\tilde \Omega$. We set
$\nu_p$ to be the smallest eigenvalue of $G_{\mu_p}$, and $\varpi_p$ the corresponding non-negative eigenfunction, normalized in $L^2(\mathcal{B}_R)$. Then it is readily seen that the function $y_p$, defined by
$$
y_p : (t,x_1,x_2) \in \R \times \mathcal{B}_R \times \tilde \Omega \mapsto  e^{-\nu_p t} \varpi_p(x_1) \phi_p(x_2),
$$
is a smooth solution of the Baouendi-Grushin system \eqref{grushin}. Applying the observability inequality \eqref{eq_observ_proof_necessary} to $y_p$ immediately yields
$$
e^{-2\nu_p T} \leq \cc \frac{1-e^{-2\nu_p T}}{2 \nu_p} \int_{\mathcal{A}} \varpi_p^2 \leq \cc \frac{1}{2\nu_p} \int_{\mathcal{A}} \varpi_p^2.
$$
\cref{lemma_bound_firsteigen} and \cref{lemma_bound_eigenfun} yields
$$
\mu_p^{2 - \frac{d}{2}} \leq \cc e^{2\kappa(\mu_p)T} \left( e^{\mu_p(2 d T - r^2)} + \mu_p^3 e^{\mu_p(2 d T - R^2)} \right),
$$
where $\cc$ is a constant independant of $p$, and $\kappa : [0,\infty[ \mapsto [0,\infty[$ verifies $\kappa(\mu) \xrightarrow{\mu \rightarrow \infty} 0$. Since $\mu_p$ goes to infinity as $p$ goes to infinity, the result follows. 
\end{proof}

\subsection{Blow-up in the observability cost as $T\rightarrow T^*$}

\cref{maininterne} gives an estimate of the blow-up in the observability cost of the Baouendi-Grushin equation \eqref{grushin}, as $T$ goes to $T^*$, when 
the observation set $\omega \times \tilde \Omega$ satisfies
$$
\omega = \left\lbrace x = r \theta \in \mathbb{R}^d,\ \theta \in \Gamma, \ R_1 < \Vert x \Vert < R_2 \right\rbrace, 
$$
with $\Gamma$ an non-empty open subset of $\mathbb{S}^{d-1}$, and $0<R_1<R_2 < R$. Indeed, we then have that for all $\beta >1$, there exists a positive constant $c_\beta$ such that, for all $T > T^* = \frac{R_1^2}{2d}$, for all $y$ solution of \eqref{grushin},
$$
\int_{\mathcal{B}_R \times \tilde \Omega} \vert y(T) \vert^2 \leq e^{\frac{c_\beta}{(T-T^*)^{2\beta}}} \int_0^T \int_{\omega\times \tilde \Omega} \vert y \vert^2.
$$ 
We claim that the blow-up rate of order $2\beta$ is likely suboptimal. Indeed, when $\omega$ is the complete annulus, we obtain a better explosion rate.

\begin{prop}
For $R_1$, $R_2$ such that $0<R_1 < R_2 < R$, set 
$$
\omega = \left\lbrace x \in \mathbb{R}^d, \ R_1 < \Vert x \Vert< R_2 \right\rbrace.
$$
Then, for all $\beta>1$, there exists a positive constant $c_\beta$ such that, for all $T> T^* = \frac{R_1^2}{2d}$, for all $y$ solution of \eqref{grushin}, 
$$
\int_{\mathcal{B}_R \times \tilde \Omega} \vert y(T) \vert^2 \leq e^{\frac{c_\beta}{(T-T^*)^{\beta}}} \int_0^T \int_{\omega\times \tilde \Omega} \vert y \vert^2.
$$ 
\end{prop}

\begin{proof}
We prove that the harmonic heat equations \eqref{harmonicheat}, parametrized by $\mu> 0$, satisfy the following uniform observability inequality: for all $T> T^*$, for all $y_\mu \in C^0([0,T];L^2(\mathcal{B}_R))$ solution of \eqref{harmonicheat},  
$$
\int_{\mathcal{B}_R} \vert y_\mu(T) \vert^2 \leq e^{\frac{c_\beta}{(T-T^*)^{\beta}}} \int_0^T \int_{\omega} \vert y_\mu \vert^2.
$$
The result is then obtained following the strategy proposed in \cref{section_strategy_HH_Grushin}.

 From \cref{prop_observability_from_annulus}, there exists a positive constant $C$, and, for all $\beta>1$, a constant $c_\beta$ such that for all $\delta > 0$, 
for all $\varepsilon \in (0,1]$,  all $\mu >0$ and all $y_\mu$ solution of \eqref{harmonicheat}, 
$$
\int_{\mathcal{B}_R} \vert y_\mu(\delta) \vert^2 \leq \frac{C}{\varepsilon^6} (1+\mu)e^{\mu (R_1+\varepsilon)^2} e^{\frac{c_\beta}{\delta^\beta}} \int_0^\delta \int_\omega \vert y_\mu \vert^2.
$$
For $T> 0$, and $0< \delta < T$ to be specified later, the known dissipation properties of the harmonic-heat equation \cite[Section 4.1]{Beauchard2020}  combined with the previous estimate yield to 
$$
\int_{\mathcal{B}_R} \vert y_\mu(T) \vert^2
\leq \frac{C}{\varepsilon^6} (1+\mu) e^{-2d \mu \left( T-\delta - \frac{(R_1+\varepsilon)^2}{2d} \right)} e^{\frac{c_\beta}{\delta^\beta}}
\int_0^T \int_\omega \vert y_\mu \vert^2.
$$
We now assume $T > T^*$, and set $\delta = \gamma T$, with $\gamma>0$ defined by
$$
(1-2\gamma) T = (1+\gamma)T^* \Leftrightarrow \gamma = \frac{T-T^*}{2T + T^*}. 
$$
We then set
$$
\varepsilon = R_1 \left(\sqrt{\gamma +1} - 1 \right) = \frac{R_1}{2} \gamma + \underset{\gamma \rightarrow 0}{o}(\gamma). 
$$
For $T - T^*$  small enough, $\delta \in (0,T)$ and $\varepsilon\in (0,1]$. A direct computation shows that
$$
(1+\mu) e^{-2d \mu \left( T-\delta - \frac{(R_1+\varepsilon)^2}{2d} \right)} = (1+\mu) e^{-2d T \gamma \mu} \leq \frac{c}{T-T^*},
$$
for some positive constant $c$. As a consequence, there exists a positive constant $K$ such that  
$$
\frac{C}{\varepsilon^6} (1+\mu) e^{-2d \mu \left( T-\delta - \frac{(R_1+\varepsilon)^2}{2d} \right)} e^{\frac{c_\beta}{\delta^\beta}}
\leq \frac{K}{(T-T^*)^7} e^{\frac{K c_\beta}{(T-T^*)^\beta}}.
$$
The result follows easily.
\end{proof}

Note that  the choice $\beta = 1$ is not admissible in the previous result, since the constant $c_\beta$ tends to infinity as $\beta \rightarrow 1$.
Nevertheless, in view of this result, it is natural to conjecture the following statement: for any non-empty open set $\omega\subset \mathcal{B}_R$, 
there exists a positive constant $c$ such that, 
defining 
$$
T^* = \frac{\inf_{x\in \omega} \Vert x \Vert^2}{2d},
$$
for all $T>T^*$ and all $y$ solution of the Baouendi-Grushin equation \eqref{grushin}, 
$$
\int_{\mathcal{B}_R\times \tilde \Omega} \vert y(T) \vert^2 \leq e^{\frac{c}{(T-T^*)}} \int_0^T \int_{\omega \times \tilde \Omega} \vert y \vert^2.
$$ 

To conclude, let us make a final observation: the loss in the blow-up rate going from the complete annulus to a portion of the annulus is due to the parameter $s$ in
\cref{obsharmonicheatsectionannulus}, which  scales as $T-T^*$ in the proof of \cref{prop_uniform_observa_wrt_mu_intern}. 

We emphasize that, in the proof of \cref{obsharmonicheatsectionannulus}, the parameters are optimized to preserve the behavior with respect to $\mu$ in the observability cost, rather than the behavior with respect to the time horizon $T$. This choice may account for the loss characterized by $s$. Whether it is possible to preserve both behaviors simultaneously remains an open question, which we leave for further study.

\label{subsection_blowup}

\appendix

\section{Functional setting}
\label{appendix_func_setting}

\subsection{Well-posedness of the Baouendi-Grushin equation}
\label{solutiongrushin}

Let $\Omega_k$ be a bounded connected smooth open domain of $\R^{d_k}$, $k = 1$, $2$, with $0 \in \Omega_1$, and $T>0$. 
In the following, we set $\Omega = \Omega_1 \times \Omega_2$, and $Q = (0,T) \times \Omega$. 

\subsubsection{Weak solutions}

For $f$, $g$ in $C^\infty_c(\Omega)$, we set 
$$
[f,g] = \int_\Omega \left( \nabla_{x_1} f \cdot \nabla_{x_1} g  + \Vert x_1 \Vert^2 \nabla_{x_2} f \cdot \nabla_{x_2} g\right),
$$
and $\left\vert[ f ]\right\vert = \sqrt{[f,f]}$. It is readily seen that $\vert[. ]\vert$ is a norm on $C^\infty_c(\Omega)$.
We define $V$ as the closure of $C^\infty_c$ for the $\vert[. ]\vert$. It is readily seen that 
$$
V = \left\lbrace f \in L^2(\Omega), \nabla_{x_1} f \in L^2(\Omega), \Vert x_1\Vert \nabla_{x_2} f \in L^2(\Omega),
f = 0 \text{ on } \partial \Omega\right\rbrace.
$$

\begin{prop}[Weak solution]
Let $y_0 \in L^2(\Omega)$ and $F \in L^2(Q)$. There exists a unique $y \in C^0([0,T];L^2(\Omega)) \cap L^2(0,T; V)$, solution to 
\begin{equation} \label{eq_Grushin_appendix}
\begin{cases}
\partial_t y - \Delta_{x_1} y - \Vert x_1 \Vert^2 \Delta_{x_2} y = F \text{ in } Q, \\
y = 0 \text{ on } (0,T) \times \partial \Omega, \\
y(0,.) = y_0 \text{ in } \Omega.
\end{cases}
\end{equation}
It satisfies 
$
\Vert y(t) \Vert_{L^2(\Omega)} \leq \Vert y_0 \Vert_{L^2(\Omega)} + \sqrt{T} \Vert F \Vert_{L^2(Q)}.
$
\end{prop}

The proof of this result is given in \cite{Beauchard2013} for $d_{1} = d_{2} = 1$. The result is easily generalised to $d_{1}$ and $d_{2}$ greater.

\subsubsection{Solutions by transposition}
\label{solfaible}
We now focus on a weaker concept of solutions by transposition, which allows to consider $L^2$-boundary controls.
Let $y_0$ be in $H^{-1}(\Omega)$, and $v \in L^2((0,T) \times \partial \Omega_1 \times \Omega_2)$. We say that $u \in C^0([0,T]; H^{-1}(\Omega)) \cap L^2(Q)$
is a solution by transposition to 
\begin{equation} \label{eq_pbl_transposition_solution}
\begin{cases}
\partial_t y - \Delta_{x_1} y - \Vert x_1 \Vert^2 \Delta_{x_2} y = 0 \text{ in } Q, \\
y = v \text{ on } (0,T) \times \partial \Omega_{1} \times \Omega_{2}, \\
y = 0 \text{ on } (0,T) \times \Omega_{1} \times \partial \Omega_{2} \\
y(0,.) = y_0 \text{ in } \Omega.
\end{cases}
\end{equation}
if for all $\phi \in C^0([0,T]; H^1_0(\Omega))\cap L^2((0,T);H^2(\Omega))\cap H^1((0,T);L^2(\Omega))$, and all $t \in (0,T]$,
$$
\langle y(t), \phi(t) \rangle - \langle y_0, \phi(0)\rangle
= \int_0^t \int_\Omega y (\partial_t \phi + \Delta_{x_1} \phi + \Vert x_1\Vert^2 \Delta_{x_2} \phi) -
\int_0^t \int_{\partial \Omega_{1} \times \Omega_{2}} v \frac{\partial \phi}{\partial n}.
$$  

\begin{prop}[Solution by transposition]
Problem \eqref{eq_pbl_transposition_solution} admits a unique solution by transposition.
\end{prop}

\subsection{Well-posedness of the harmonic-heat equation}

In this section, we consider a smooth function \( V: \mathbb{R}^d \to [0, \infty) \) and let \( \Omega \) be a bounded, smooth, open domain in \( \mathbb{R}^d \). Let \( T > 0 \).

\begin{prop} \label{prop_wellposedness_harmonicheatequation}
For any initial data \( y_0 \in L^2(\Omega) \), and any source term $F\in L^2((0,T)\times \Omega)$, there exists a unique solution
\[
y \in C^0([0,T]; L^2(\Omega)) \cap L^2((0,T); H^1_0(\Omega))
\]
to the problem:
\[
\begin{cases}
\partial_t y - \Delta y + V y = F & \text{in } (0,T) \times \Omega, \\
y = 0 & \text{on } (0,T) \times \partial \Omega, \\
y(0, \cdot) = y_0 & \text{in } \Omega.
\end{cases}
\]
Moreover, the exits a constant $c$ depending on $V$ such that, for all \( t \in [0,T] \),
\[
\| y(t) \|_{L^2(\Omega)} \leq c( \| y_0 \|_{L^2(\Omega)} + \Vert F \Vert_{L^2((0,t)\times \Omega)}).
\]
Finally, if \( y_0 \in H^1_0(\Omega) \), then the solution satisfies
\[
y \in C^0([0,T]; H^1_0(\Omega)) \cap L^2((0,T); H^2(\Omega)) \cap H^1((0,T); L^2(\Omega)).
\]
\end{prop}

\section{Energy-type estimates}
\label{proofdissipationgrushin}
Recall that $\grush$ is the operator defined in \eqref{notationharmonicheat}. We prove an energy estimate for the solutions of \eqref{harmonicheat}, where the initial data lies potentially in a subset of $H^1_0(\mc{B}_R)$.
\begin{prop}
\label{dissipationgrushin}
Let $\lambda >0$, $\mu > 0$ and $\tilde{F} \subset H^1_0(\mc{B}_R)$ stable by $\grush$ such that the eigenvalues of $G_{\mu_{|\tilde{F}}}$ are greater than $\lambda$. For any $(t, s)$ such that $0 \le s\le t$, for any $y_0 \in \tilde{F}$ the solution $y_\mu$ of \eqref{harmonicheat} satisfies
\begin{equation}\label{dissipharmonicheatL2} \norm{y_\mu(t)}{L^2(\mathcal{B}_R)}^2 \le \exp(-2\lambda(t-s))\norm{y_\mu(s)}{L^2(\mathcal{B}_R)}^2,\end{equation}
\begin{equation}\label{dissipharmonicheatH10}\norm{y_\mu(t)}{\mu}^2 \le (1+R^2)\exp(-2\lambda(t-s))\norm{y_\mu(s)}{\mu}^2,\end{equation}
with $\norm{\cdot}{\mu}$ defined in \eqref{munorm}.
\end{prop}

\begin{proof}
Let $\lambda >0$, $\mu >0$ and $\tilde{F} \subset H^1_0(\mc{B}_R)$ stable by $\grush$ such that the eigenvalues of $G_{\mu_{|\tilde F}}$ are greater than $\lambda$. Let $y_{0,\mu} \in \tilde F$, we denote $y_\mu$ the solution of \eqref{harmonicheat}. By stability of $\tilde F$ by the semigroup we have $y_\mu(t) \in \tilde F$ for any $t \ge 0$.
We first prove the dissipation of the $L^2$-norm. We multiply \eqref{harmonicheat} by $y_\mu$ and integrate by parts to obtain
$$\frac12 \frac{d}{dt}\left(\norm{y_\mu(t)}{L^2(\mathcal{B}_R)}^2\right) = -\scalar{\grush y_\mu(t)}{y_\mu(t)}{L^2(\mathcal{B}_R)}.$$
Since the eigenvalues are greater than $\lambda$ we get
$$\frac12 \frac{d}{dt}\left(\norm{y_\mu(t)}{L^2(\mathcal{B}_R)}^2\right) \le -\lambda\norm{y_\mu(t)}{L^2(\mathcal{B}_R)}^2.$$
We obtain
$$\norm{y_\mu(t)}{L^2(\mathcal{B}_R)}^2 \le e^{-2\lambda(t-s)}\norm{y_\mu(s)}{L^2(\mathcal{B}_R)}^2, \quad \forall 0 \le s\le t.$$
Then we prove the dissipation of the $\mu$-norm. We multiply \eqref{harmonicheat} by $\grush y_\mu$ and integrate by parts to obtain
$$\frac12 \frac{d}{dt}\left(\norm{\nabla y_\mu(t)}{L^2(\mathcal{B}_R)}^2 + \norm{\mu |x| y_\mu(t)}{L^2(\mathcal{B}_R)}^2\right) = -\norm{\grush y_\mu(t)}{L^2(\mathcal{B}_R)}^2.$$
Since the eigenvalues are greater than $\lambda$ we have
\begin{align*}
\norm{\nabla y_\mu(t)}{L^2(\mathcal{B}_R)}^2 &= \scalar{-\Delta y_\mu(t)}{y_\mu(t)}{L^2(\mathcal{B}_R)} \\
&= \scalar{\grush y_\mu(t)}{y_\mu(t)}{L^2(\mathcal{B}_R)} - \norm{\mu |x| y_\mu(t)}{L^2(\mathcal{B}_R)}^2 \\
&\le \frac{1}{\lambda}\norm{\grush y_\mu(t)}{L^2(\mathcal{B}_R)}^2 - \norm{\mu |x| y_\mu(t)}{L^2(\mathcal{B}_R)}^2,
\end{align*}
therefore we get 
$$\frac12 \frac{d}{dt}\!\left(\norm{\nabla y_\mu(t)}{L^2(\mathcal{B}_R)}^2 \!\!+ \norm{\mu |x| y_\mu(t)}{L^2(\mathcal{B}_R)}^2\!\right) \le \!-\lambda \left(\norm{\nabla y_\mu(t)}{L^2(\mathcal{B}_R)}^2\!\! + \norm{\mu |x| y_\mu(t)}{L^2(\mathcal{B}_R)}^2\!\right)\!.$$
We obtain for any $0\le s \le t$,
$$\norm{\nabla y_\mu(t)}{L^2(\mathcal{B}_R)}^2 \!\!+ \norm{\mu |x| y_\mu(t)}{L^2(\mathcal{B}_R)}^2 \le e^{-2\lambda(t-s)} \left(\norm{\nabla y_\mu(s)}{L^2(\mathcal{B}_R)}^2 \!\!+ \norm{\mu |x| y_\mu(s)}{L^2(\mathcal{B}_R)}^2\!\right)\!.$$
Using the fact that $|x| \le R$ we have
 $$\norm{\nabla y_\mu(t)}{L^2(\mathcal{B}_R)}^2 \le e^{-2\lambda(t-s)} \left(\norm{\nabla y_\mu(s)}{L^2(\mathcal{B}_R)}^2 + R^2 \mu^2 \norm{ y_\mu(s)}{L^2(\mathcal{B}_R)}^2\right).$$
therefore combining this estimate with the dissipation of the $L^2$-norm above we obtain the announced estimate
$$\norm{y_\mu(t)}{\mu}^2 \le (1+R^2)\exp\left(-2\lambda(t-s)\right)\norm{y_\mu(s)}{\mu}^2.$$ 
\end{proof}

\begin{lem}[Energy estimate] \label{lemma_standard_energ_est_source}
Let $\Omega$ be a smooth bounded domain of $\R^d$, and $C_\Omega$ the square of the Poincar\'e constant on $\Omega$, defined by
$$
\frac{1}{C_\Omega} = \inf_{\mathcal w \in H^1_0(\Omega)} \frac{\int_\Omega \Vert \nabla \mathcal w \Vert^2}{\int_\Omega \vert \mathcal w \vert^2}.
$$
Let $\omega$ be a smooth  domain such that $\omega\subset \Omega$. 
Let $T>0$, $\Theta : [0,T] \mapsto \R$ smooth such that $\Theta =  1$ on $\left[\frac{T}{3}, 1\right]$. Let $V$
in $L^\infty((0,T) \times \omega)$  be  nonnegative.  Any function $y$ in $C^0([0,T];L^2(\Omega)) \cap L^2((0,T);H^1_0(\Omega))$ such that $\partial_t y - \Delta y + Vy $ belongs to $L^2((0,T) \times \omega)$
satisfies
$$
\int_\omega \vert y(T) \vert^2 \leq \frac{3}{T} \int_{\frac{T}{3}}^{\frac{2T}{3}} \int_\omega \vert y \vert^2 + C_\Omega \int_0^T \int_\omega
\Theta \left\vert \partial_t y - \Delta y + V y \right\vert^2 .
$$
\end{lem}

\begin{proof} We denote by \( C_\omega \) the Poincar\'e constant on \( \omega \), defined analogously to \( C_\Omega \).

Let \( F = \partial_t y - \Delta y + V y \). By density arguments, it suffices to prove the inequality for any \( y \) in
\[
C^0([0,T]; H^1_0(\Omega)) \cap L^2((0,T); H^2(\Omega)) \cap H^1((0,T); L^2(\Omega)),
\]
satisfying \( \partial_t y - \Delta y + V y = F \).

We compute:
$$
\int_\omega F y = \int_\omega \left( \partial_t y - \Delta y + V y\right) y
= \frac{1}{2} \frac{d}{dt} \int_\omega \vert y \vert^2 + \int_\omega \left( \Vert \nabla y \Vert^2 + V \vert y \vert^2\right).
$$
We have
$$
\left\lvert \int_\omega Fy \right\vert \leq \frac{C_\omega}{2} \int_\omega \vert F\vert^2 + \frac{1}{2 C_\omega} \int_\omega \vert y \vert^2
\leq \frac{C_\omega}{2} \int_\Omega \vert F\vert^2 + \frac{1}{2} \int_\omega \Vert \nabla y \Vert^2.
$$
We deduce that
$$
\frac{d}{dt} \int_\omega \vert y \vert^2 \leq C_\omega \int_\omega \vert F \vert^2,
$$
which immediately implies, for all $0\leq t \leq T$, 
$$
\int_\omega \vert y(T) \vert^2 \leq \int_\omega \vert y(t)\vert^2 + C_\omega \int_t^T  \int_\omega \vert F \vert^2 .
$$
Integrating this inequality for $t$ in $\left[ \frac{T}{3}, \frac{2T}{3} \right]$ gives
$$
\frac{T}{3} \int_\omega \vert y(T) \vert^2 \leq \int_{\frac{T}{3}}^{\frac{2T}{3}} \int_\omega \vert y \vert^2 
+ C_\omega\int_{\frac{T}{3}}^{\frac{2T}{3}} \int_t^T \int_\omega \vert F \vert^2
\leq \int_{\frac{T}{3}}^{\frac{2T}{3}} \int_\omega \vert y \vert^2  + C_\omega \frac{T}{3} \int_{\frac{T}{3}}^T  \int_\omega \vert F \vert^2.
$$
The result follows from the definition of $\Theta$, as $C_\omega \leq C_\Omega$.
\end{proof}

\begin{lem} \label{lemma_estimate_nabla_T}
Let $\omega$ be a smooth bounded open domain of $\R^d$, and $V : \overline \omega \mapsto \R$ a smooth positive function.
Then for all $T>0$, for all $y$ in $C^0([0,T]; H^1_0(\Omega)) \cap L^2((0,T); H^2(\Omega)) \cap H^1((0,T); L^2(\Omega)),
$ such that $\partial_t y - \Delta y + V y = 0$ in $(0,T) \times \omega$,
the following estimates holds:
$$
\int_\omega \Vert \nabla y(T) \Vert^2 \leq \frac{3}{T} \int_{\frac{T}{3}}^{\frac{2T}{3}} \int_\omega \Vert \nabla y \Vert^2 +
2 \Vert \Delta V \Vert_\infty \int_{\frac{T}{3}}^{\frac{2T}{3}} \int_\omega \vert y \vert^2.
$$
\end{lem}

\begin{proof}
First, performing the same computation as in the previous proof --- this time with \( F = 0 \)  --- we find that the function
$\displaystyle  t \in (0, T) \mapsto \int_\omega |y|^2 $
is nonincreasing. The following estimate is then easily derived:
$$
\int_{\frac{T}{3}}^T \int_\omega \vert y\vert^2 \leq 2 \int_{\frac{T}{3}}^{\frac{2T}{3}} \int_\omega \vert y \vert^2.
$$

Now, we observe that
\begin{multline*}
0 = - \int_\omega \left( \partial_t y - \Delta y + V y \right) \Delta y
 \\= \frac{1}{2} \frac{d}{dt} \int_\omega \Vert \nabla y \Vert^2 +  \int_\omega \vert \Delta y \vert^2
 +  \int_\omega V \Vert \nabla y \Vert^2 + \frac{1}{2} \int_\omega \nabla V\cdot  \nabla \vert y\vert^2
 \\ = \frac{1}{2} \frac{d}{dt} \int_\omega \Vert \nabla y \Vert^2 +  \int_\omega \vert \Delta y \vert^2
 + \int_\omega V \Vert \nabla y \Vert^2 - \frac{1}{2}\int_\omega \Delta V \vert y \vert^2.
\end{multline*}
It implies that 
$$
\frac{d}{dt} \int_\omega \Vert \nabla y \Vert^2 \leq \int_\omega \vert \Delta V  \vert \vert y \vert^2,
$$
which immediately implies that for all $0\leq \frac{T}{3} \leq t \leq T$,
\begin{multline*}
\int_{\omega} \Vert \nabla y(T) \Vert^2 \leq \int_\omega \Vert\nabla y(t) \Vert^2 + \int_t^T \int_\omega \vert \Delta V \vert \vert y \vert^2
\\ \leq \int_\omega \Vert\nabla y(t) \Vert^2 + \Vert \Delta V \Vert_\infty \int_{\frac{T}{3}}^{T} \int_\omega \vert y \vert^2
\\ \leq \int_\omega \Vert\nabla y(t) \Vert^2  + 2 \Vert \Delta V \Vert_\infty \int_{\frac{T}{3}}^{\frac{2T}{3}} \vert y \vert^2.
\end{multline*}
\end{proof}

\begin{lem}[Caccioppoli-type inequality] \label{lemma_obsnablau->u}
Let \( T > 0 \), and let \( \vartheta : [0, T] \to [0, \infty) \) be a smooth function such that \( \vartheta(0) = 0 \).
Let \( \mathcal{O} \) and \( \omega \) be two bounded open subsets of \( \mathbb{R}^d \) with \( \overline{\mathcal{O}} \subset \omega \), and let \( \chi \in C^\infty_c(\omega) \) be a nonnegative function such that \( \chi \equiv 1 \) in \( \mathcal{O} \).

Then, for all \( \mu \geq 0 \) and for all smooth function \( y \) satisfying
\[
\partial_t y - \Delta y + \mu^2 \|x\|^2 y = 0 \quad \text{in } (0, T) \times \omega,
\]
the following inequality holds:
\[
\int_0^T \int_{\mathcal{O}} \vartheta \| \nabla y \|^2 \, dx \, dt \leq \|\vartheta\|_{W^{1,\infty}} \|\chi\|_{W^{2,\infty}} \int_0^T \int_{\omega} |y|^2 \, dx \, dt.
\]

\end{lem}

\begin{proof}
We have
$$
\int_0^T \int_{\mathcal{O}} \vartheta \Vert \nabla y \Vert^2 dx = \int_0^T \int_{\mathcal{O}} \vartheta \chi\Vert \nabla y \Vert^2 dx
\leq \int_0^T \int_\omega\vartheta \chi\Vert \nabla y \Vert^2 dx.
$$
We obviously have
\begin{multline*}
0 = \int_0^T \int_\omega (\partial_t y - \Delta y + \mu^2 \Vert x \Vert^2 y) \vartheta \chi y \\
= \int_0^T \int_\omega\vartheta \chi \partial_t y y - \int_0^T \int_\omega \vartheta \chi \Delta y y +\int_0^T \int_\omega \vartheta \chi \mu^2 \Vert x \Vert^2\vert y \vert^2.
\end{multline*}
First, we have
$$
\int_0^T \int_\omega\vartheta \chi \partial_t y y = \frac{1}{2} \int_0^T\int_\omega \vartheta \chi \partial_t \vert y \vert^2
= \frac{\vartheta(T)}{2} \int_\omega \chi \vert y(T) \vert^2 - \frac{1}{2} \int_0^T \int_\omega \vartheta ' \chi \vert y \vert^2.
$$
Second, we have
\begin{multline*}
-\int_0^T \int_\omega \vartheta \chi \Delta y y = \int_0^T \int_\omega \vartheta \chi \Vert \nabla y \Vert^2 + \frac{1}{2} \int_0^T \int_\Omega
\vartheta \nabla \chi \cdot \nabla \vert y \vert^2 \\ 
= \int_0^T \int_\omega \vartheta \chi \Vert \nabla y \Vert^2
 - \frac{1}{2} \int_0^T \int_\omega \vartheta \Delta \chi \, \vert y \vert^2.
\end{multline*}
Consequently, we obtain
\begin{multline*}
\int_0^T \int_\omega \vartheta \chi \Vert \nabla y \Vert^2 + \frac{\vartheta(T)}{2} \int_\omega \chi \vert y(T) \vert^2  +\int_0^T \int_\omega \vartheta \chi \mu^2 \Vert x \Vert^2 \vert y \vert^2 \\
=  \frac{1}{2} \int_0^T \int_\omega \left( \vartheta ' \chi + \vartheta \Delta \chi\right) \vert y \vert^2.
\end{multline*}
The result follows.
\end{proof}

\section{Decomposition in a basis of eigenfuctions of the Laplace-Beltrami operator on the sphere}
\label{prooflemH10M}
The eigenvalues of the Laplace-Beltrami operator $-\Delta_{\sphere}$ on the sphere in dimension $d \ge 2$ are $(m(m+d-2))_{m \in \N}$.\newline
For any $m \in \N$, the eigenvalue is of multiplicity
\begin{equation} \label{defbeta} \beta(m,d) := \begin{pmatrix}
        m+d-1 \\
        d-1
    \end{pmatrix} - \begin{pmatrix}
        m+d-3 \\
        d-1
    \end{pmatrix} ~\text{if}~m \ge 1 \text{ and } \beta(0,d) = 1.\end{equation} 
We denote \begin{equation} \label{defymk} \{Y_{mk}~;~m\in \N \quad k \in \llbracket 1, \beta(m,d) \rrbracket \} \end{equation} the associated eigenfunctions which form an orthonormal basis of $L^2(\sphere)$. With these notations we have the identity
    \begin{equation} \label{laplacienvp} -\Delta_{\sphere}Y_{mk} = m(m+d-2)Y_{mk}, \quad \forall m \in \N,~\forall k \in \llbracket 1, \beta(m,d) \rrbracket.\end{equation}
These results can be found in \cite[Chapter III, § 22]{Shubin2001}.
\newline 
We prove that any function in $L^2(\mc{B}_R)$ or $H^1_0(\mc{B}_R)$ can be decomposed as a sum of functions of the radial variable times $Y_{mk}$, as stated in the following lemma.
\begin{lem}
\label{lemme H10M}
    For any $y \in L^2(\mc{B}_R)$ (resp. $y \in H^1_0(\mc{B}_R)$), 
    $$y = \sum_{\underset{k \in \llbracket1, \beta(m,d)\rrbracket}{m =0}}^{\infty}\scalar{y}{Y_{mk}}{L^2(\sphere)}Y_{mk}.$$
\end{lem}
where the notation $$\sum_{\underset{k \in \llbracket1, \beta(m,d)\rrbracket}{m=0}}^\infty\scalar{y}{Y_{mk}}{L^2(\sphere)}Y_{mk}$$ means 
$$ (r,\theta) \in (0,R) \times \sphere \mapsto \sum_{\underset{k \in \llbracket1, \beta(m,d)\rrbracket}{m=0}}^\infty \scalar{y(r,\cdot)}{Y_{mk}}{L^2(\sphere)}Y_{mk}(\theta). $$

\begin{proof}
Since the proof for $L^2(\mc B_R)$ is easier than the proof for $H_0^1(\mathcal{B}_R)$, and uses the same argument (Lebesgue's dominated convergence theorem), we give the details only for the latter. 
Let $y \in H_0^1(\mathcal{B}_R)$, we make the identification between $y$, and the corresponding function in polar coordinates $\hat{y} : (r,\theta) \in (0,R) \times \sphere \mapsto \hat{y}(r, \theta)$ and from now on we use the notation $y$ for both.
First we recall that to compute the $H^1_0$-norm in polar coordinates we have the formula $$\int_{\mathcal{B}_R} |\nabla y|^2 dx = \int_0^R \int_{\sphere} r^{d-1}(\partial_r y)^2 + r^{d-3} |\nabla_{\partial} y|^2 d\sigma dr,$$ where $\nabla_{\partial}$ denotes the gradient with respect to the variable $\theta \in \sphere$.
We use it to prove the convergence of the series in $H^1_0(\mathcal{B}_R)$. First we observe that for a.e $r \in (0,R)$, $y(r, \cdot) \in H^1(\sphere)$. Therefore, as $\left(\dfrac{Y_{mk}}{\sqrt{m(m+d-2) + 1}}\right)_{\underset{k \in \llbracket1, \beta(m,d)\rrbracket}{m \ge 0}}$ is an orthonormal basis of $ H^1(\sphere)$ for the usual $H^1 $-scalar product we have for a.e $r \in (0,R)$,
    $$ y(r,\cdot) = \sum_{\underset{k \in \llbracket1, \beta(m,d)\rrbracket}{m = 0}}^{+ \infty}\scalar{y(r,\cdot)}{\dfrac{Y_{mk}}{\sqrt{m(m+d-2) + 1}}}{H^1(\sphere)}\dfrac{Y_{mk}}{\sqrt{m(m+d-2) + 1}}.$$
    By Green's formula, and the fact that $-\Delta_{\sphere} Y_{mk} = m(m+d-2) Y_{mk}$ we obtain for a.e $r \in (0,R)$
    $$ y(r,\cdot) = \sum_{\underset{k \in \llbracket1, \beta(m,d)\rrbracket}{m = 0}}^{+ \infty}\scalar{y(r,\cdot)}{Y_{mk}}{L^2(\sphere)}Y_{mk}. $$
    Therefore, by Lebegue's dominated convergence theorem we are able to show that
    $$ \int_0^R \int_{\sphere} r^{d-3} \left|\nabla_{\partial} \left( \sum_{\underset{k \in \llbracket1, \beta(m,d)\rrbracket}{m = 0}}^M \scalar{y}{Y_{mk}}{L^2(\sphere)}Y_{mk} -y\right) \right|^2 d\sigma dr \longrightarrow 0 \text{ as } M \rightarrow \infty.$$
 Indeed we have for a.e $r \in (0,R)$ $$ r^{d-3}\int_{\sphere}\left|\nabla_{\partial} \left( \sum_{\underset{k \in \llbracket1, \beta(m,d)\rrbracket}{m = 0}}^M \scalar{y(r,\cdot)}{Y_{mk}}{L^2(\sphere)}Y_{mk} -y(r,\cdot)\right) \right|^2 d\sigma \longrightarrow 0  \text{ as } M \rightarrow \infty,$$ and the domination
 \begin{multline*} r^{d-3}\int_{\sphere}\left|\nabla_{\partial} \left( \sum_{\underset{k \in \llbracket1, \beta(m,d)\rrbracket}{m = 0}}^M \scalar{y(r,\cdot)}{Y_{mk}}{L^2(\sphere)}Y_{mk} -y(r,\cdot)\right) \right|^2 d\sigma \\
 \le \underset{\text{integrable on }(0,R)}{\underbrace{r^{d-3} \int_{\sphere}|\nabla_{\partial} y(r,\cdot)|^2 d\sigma}}.\end{multline*}
 Moreover $$\displaystyle \int_0^R \int_{\sphere} r^{d-1}\left| \partial_{r} \left(\sum_{\underset{k \in \llbracket1, \beta(m,d)\rrbracket}{m = 0}}^M \scalar{y}{Y_{mk}}{L^2(\sphere)}Y_{mk} -y\right)\right| ^2 d\sigma dr \longrightarrow 0 \text{ as } M \rightarrow \infty.$$
 Indeed since $\partial_{r}y(r,\cdot) \in L^2(\sphere)$ we have for a.e $r \in (0,R)$,
 $$ r^{d-1}\int_{\sphere}\left|\sum_{\underset{k \in \llbracket1, \beta(m,d)\rrbracket}{m = 0}}^M \scalar{\partial_{r}y(r,\cdot)}{Y_{mk}}{L^2(\sphere)}Y_{mk} -\partial_{r}y(r,\cdot)\right|^2 d\sigma \longrightarrow 0  \text{ as } M \rightarrow \infty, $$
and the domination 
 \begin{multline*}r^{d-1}\int_{\sphere}\left|\sum_{\underset{k \in \llbracket1, \beta(m,d)\rrbracket}{m = 0}}^M \scalar{\partial_{r}y(r,\cdot)}{Y_{mk}}{L^2(\sphere)}Y_{mk} -\partial_{r}y(r,\cdot)\right|^2 d\sigma \\
  \le \underset{\text{integrable on }(0,R)}{\underbrace{r^{d-1}\int_{\sphere} |\partial_{r}y(r,\cdot)|^2 d\sigma}}.\end{multline*}
\end{proof}

\section{Technical results}

\label{appendix_technical_results}

\begin{lem} \label{lemma_existence_a}
Let \(\alpha \in (0,1)\) and \(\mathcal{c} > 0\). There exists a constant \(a > 0\) such that, for all \(x \geq 0\),
\[
\sinh(2 a x^\alpha) \geq \mathcal{c} x.
\]

Moreover, the infimum of all such \(a\) satisfies
\[
\inf \left\{ a > 0 \mid \sinh(2 a x^\alpha) \geq \mathcal{c} x, \ \forall x > 0 \right\} \to \infty \quad \text{as} \quad \alpha \to 0.
\]
\end{lem}

\begin{proof} Let \( h \colon t > 0 \mapsto \tanh(t) - \alpha t \). Then, the derivative is given by
\( h'(t) = 1 - \alpha - \tanh(t)^2 \). Consequently, there exists a unique \( t_1 >0\) such that
\( h'(t_1) = 0 \), which implies \( \tanh(t_1) = \sqrt{1 - \alpha} \).

From this, we deduce that \( h \) is increasing on \([0, t_1]\) and decreasing on \([t_1, \infty)\). This behavior ensures the existence of a unique \( t_\alpha \geq t_1 \) such that \( h(t_\alpha) = 0 \).

For \( a > 0 \), consider the function
\[
g : x > 0 \mapsto \frac{\sinh(2a x^\alpha)}{x} - \mathcal{c}.
\]
A direct computation shows that \( g'(x) = 0 \) if and only if \( \alpha 2 a x^\alpha = \tanh(2 a x^\alpha) \). Consequently, there exists a unique \( x_c > 0 \) such that \( g'(x_c) = 0 \), which is given by \( 2 a x_c^\alpha = t_\alpha \). As a result, we have
\[
g(x_c) = 2^{\frac{1}{\alpha}} a^{\frac{1}{\alpha}} \frac{\sinh(t_\alpha)}{t_\alpha^{\frac{1}{\alpha}}} - \mathcal{c} \xrightarrow{a \to \infty} +\infty.
\]
This immediately gives $a>0$ such that $g(x) >0$ for all $x>0$. 

A second consequence of the value of \( g(x_c) \) is that
\[
\inf \left\{ a > 0 \mid \sinh(2 a x^\alpha) \geq \mathcal{c} x, \ \forall x > 0 \right\} = a_c,
\]
where
\[
a_c = \frac{\mathcal{c}^\alpha}{2} t_\alpha \sinh(t_\alpha)^{-\alpha}.
\]

We know that \( t_\alpha \geq t_1 = \tanh^{-1}(\sqrt{1-\alpha}) \), which implies that \( t_\alpha \to \infty \) as \( \alpha \to 0 \). Since \( \tanh(t_\alpha) = \alpha t_\alpha \), it is easily deduced that
\[
t_\alpha = \alpha^{-1} + o_{\alpha \to 0}(\alpha^{-1}).
\]

This, in turn, yields \( \sinh(t_\alpha)^{-\alpha} \to e^{-1} \) as \( \alpha \to 0 \). The result follows.
\end{proof}

\begin{lem} \label{lemma_boundonhyperbfunc} For all $x\geq 0$, 
$$
\coth(x) \leq 1 +\frac{1}{x}, \quad 
 \frac{\sinh(4x)}{\sinh(2x)^2} \leq 2+ \frac{1}{x}.
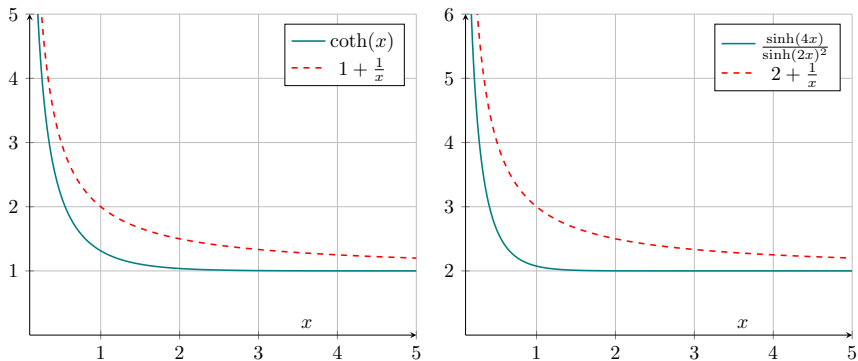
$$

\begin{tikzpicture}[scale=0.75]
\begin{axis}[
    axis lines = middle,
    xlabel = \(x\),
    xmin = 0.1,
    xmax = 5,
    ymin = 0,
    ymax = 5,
    domain = 0.1:5,
    samples = 200,
    legend pos = north east,
    legend style = {cells={align=left}},
    restrict y to domain=-10:10,
    xtick = {1, 2, 3, 4, 5},
    ytick = {1, 2, 3, 4, 5},
    grid = both,
    major grid style = {line width=.2pt,draw=gray!50},
]

\addplot[teal, thick, smooth] {1/tanh(x)};
\addlegendentry{\(\coth(x)\)}

\addplot[red, thick, dashed, smooth] {1 + 1/x};
\addlegendentry{\(1 + \frac{1}{x}\)}

\end{axis}
\end{tikzpicture}
\begin{tikzpicture}[scale=0.75]
\begin{axis}[
    axis lines = middle,
    xlabel = \(x\),
    xmin = 0.1,
    xmax = 5,
    ymin = 1,
    ymax = 6,
    domain = 0.1:5,
    samples = 200,
    legend pos = north east,
    legend style = {cells={align=left}},
    restrict y to domain=-10:20,
    xtick = {1, 2, 3, 4, 5},
    ytick = {2, 3, 4, 5, 6},
    grid = both,
    grid style = {line width=.1pt, draw=gray!10},
    major grid style = {line width=.2pt,draw=gray!50},
]

\addplot[teal, thick, smooth] {sinh(4*x)/((sinh(2*x))^2)};
\addlegendentry{\(\frac{\sinh(4x)}{\sinh(2x)^2}\)}

\addplot[red, thick, dashed, smooth] {2 + 1/x};
\addlegendentry{\(2 + \frac{1}{x}\)}

\end{axis}
\end{tikzpicture}
\end{lem}

\begin{proof}
The first inequality is quite standard and follows from a direct analysis of the function
$\eta : x > 0 \mapsto 1 + \frac{1}{x} - \coth(x)$, which is decreasing and satisfies $\eta(x) \xrightarrow{x \rightarrow 0} 1$.

The second inequality is a direct consequence of the first one, as 
$$
\frac{\sinh(4x)}{\sinh(2x)^2} = 2 \coth(2x).
$$
\end{proof}

\bibliographystyle{plain}
\bibliography{biblio}
\end{document}